\documentclass{article}
\usepackage[top=1truein,bottom=1truein,left=1truein,right=1truein]{geometry}
\usepackage{setspace}
\usepackage{authblk}

\usepackage{amsmath,amssymb,amsthm}
\usepackage{pxfonts,graphicx}
\usepackage{enumerate,mathrsfs,latexsym}

\usepackage{here}
\usepackage{color}
\usepackage{ctable,dcolumn}
\usepackage{mathrsfs}
\usepackage{multirow}

\newtheorem{thm}{Theorem}[section]
\newtheorem{prop}[thm]{Proposition}
\newtheorem{lem}[thm]{Lemma}
\newtheorem{cor}[thm]{Corollary}
\newtheorem{defi}[thm]{Definition}
\newtheorem{rem}[thm]{Remark}

\newtheorem{assump}[thm]{Assumption}
\newtheorem{step}{Step}

\newcommand{\argmin}{\mathop{\rm arg~min}\limits}
\newcommand{\esssup}{\mathop{\rm ess~sup}\limits}

\providecommand{\keywords}[1]{\textbf{Keywords}\quad #1}

\title{Is Volatility Rough ?\footnote{This work was supported
		by JSPS KAKENHI Grant Number JP17J04605.}}
\author[1]{Masaaki Fukasawa\footnote{fukasawa@sigmath.es.osaka-u.ac.jp}}
\author[2]{Tetsuya Takabatake\footnote{takabatake@sigmath.es.osaka-u.ac.jp. Current address: Graduate School of Social Sciences, Hiroshima University, \\1-2-1 Kagamiyama, Higashi-Hiroshima, Hiroshima, Japan, tkbtk@hiroshima-u.ac.jp}}
\author[3]{Rebecca Westphal\footnote{rwestphal@ethz.ch}}
\affil[1,2]{Graduate School of Engineering Science, Osaka University\footnote{1-3 Machikaneyama, Toyonaka, Osaka, Japan}}
\affil[3]{Department of Management, Technology, and Economics, ETH Z\"urich\footnote{Scheuchzerstrasse 7, 8092 Z\"urich, Switzerland}}
\date{}

\begin{document}
\maketitle
\begin{abstract}
	Rough volatility models are continuous time stochastic
 volatility models where the volatility process is driven by a
 fractional Brownian motion with the Hurst parameter smaller than half,
 and have attracted much attention since a seminal paper titled
 ``Volatility is rough'' was posted on SSRN in 2014 showing that
the log realized volatility time series of major stock indices have the
 same scaling property as 
 such a rough fractional Brownian motion has.
We however find by simulations that
the impressive approach tends to suggest the same roughness  irrespectively
whether the 
volatility is actually rough or not; an overlooked estimation error of
 latent volatility often results in an illusive scaling property.
Motivated by this preliminary
 finding,
here we develop a statistical theory for a continuous time fractional
 stochastic volatility model to examine 
whether the Hurst parameter 
is indeed estimated smaller than half, that is, whether the volatility is really rough.
We construct a
 quasi-likelihood estimator 
 and apply it to realized volatility time series.
Our quasi-likelihood is based on the error distribution of the realized
volatility
 and a
 Whittle-type approximation to the auto-covariance of the log-volatility
 process. We prove the consistency of our estimator under high frequency
 asymptotics, and examine by simulations its finite sample performance.
Our empirical study suggests that the volatility is
 indeed rough; actually it is even rougher than considered in the
 literature.
\end{abstract}
\vspace{0.2cm}

\section{Introduction}
Nowadays it is widely recognized that the volatility of an asset price is not a constant but a stochastic process. The property of the process is, however, not very clear because it is not a directly observable process.
Even in a simple continuous framework where 
the volatility process $\sigma$ defined through
\begin{equation}\label{M1}
\mathrm{d}S_u = \sigma_u S_u \mathrm{d}B_u
\end{equation}
with an asset price process $S$ and a Brownian motion $B$,
one can only examine indirectly its properties via a statistic like
the realized variance
\begin{equation*}
\hat{\sigma}_{\delta, t}^2 : = 
\sum_{(t-1)\delta \leq u \leq t\delta}\left| \Delta \log \bar{S}_u \right|^2,
\end{equation*}
where $\bar{S}$ is a piecewise constant process which jumps at every sampling time of $S$ to the observed value of $S$ at the time.
In a hypothetical situation where sampling frequency goes to infinity without any measurement error, 
\begin{equation}\label{L1}
\hat{\sigma}^2_{\delta,t} \to \int_{(t-1)\delta}^{t\delta} \sigma_u^2\, \mathrm{d}u
\end{equation}
in probability; one therefore expects $\hat{\sigma}^2_{\delta,t}$ to
work as a proxy of the unobservable quantity. Since the high frequency
asymptotics does not require any ergodicity or stationarity assumption,
it particularly fits the analysis of recent financial market data, where
the sampling frequency is really high.  
The two remarkable empirical properties of daily realized variance time series
that were already documented in the earliest work by 
Andersen et al.~\cite{ABDL} are 
that their unconditional distributions are approximately log Gaussian, and
that their auto-covariances decay slowly.
Various modifications of the realized variance taking into account
market microstructure noise and asset price jumps have been proposed and
associated limit theorems have been proven in the literature; see
A\"it-Sahalia and Jacod~\cite{AJ} for an overview.

In 2014, an interesting paper,  Gatheral et al.~\cite{GJR}
titled ``Volatility is rough'' was posted on SSRN.
Since then, it has been so  influential  in the community of Mathematical Finance that a number of papers\footnote{
	Antoine Jacquier established and has maintained a website
	
	https://sites.google.com/site/roughvol/home
	
	as a reference point for the fast growing literature of rough volatility.
} 
have already appeared  dealing with the so-called rough volatility models. 
In that paper, the authors
looked at the historical volatility proxy data including those from the Oxford-Man realized library\footnote{
	The Oxford-Man Institute provides daily nonparametric volatility estimates at
	
	https://realized.oxford-man.ox.ac.uk
}.
Let $\hat{\sigma}_t$ be such a volatility proxy as the realized volatility 
\begin{equation*}
\hat{\sigma}_t\equiv\hat{\sigma}_{\delta,t}:=\sqrt{\hat{\sigma}^2_{\delta, t}}
\end{equation*}
for a day $t$  computed from intraday asset price data,
where $\delta$ corresponds to the length of one day.
They did a linear regression to find an impressive fit
\begin{equation}\label{LR}
\log  \frac{1}{n}\sum_{t =1}^n |\log \hat{\sigma}_{t+\Delta} - \log \hat{\sigma}_t|^q
\approx \zeta_q \log \Delta + \eta_q
\end{equation}
for various values of $q$; see Figure~\ref{F1} (left).
Then, for the regression coefficients $\zeta_q$, 
they did another linear regression to find another impressive fit $\zeta_q \approx  Hq$ with $H \approx 0.1$;
see Figure~\ref{F1} (right).
\begin{figure}
	\centering
	\includegraphics[width=6.2cm]{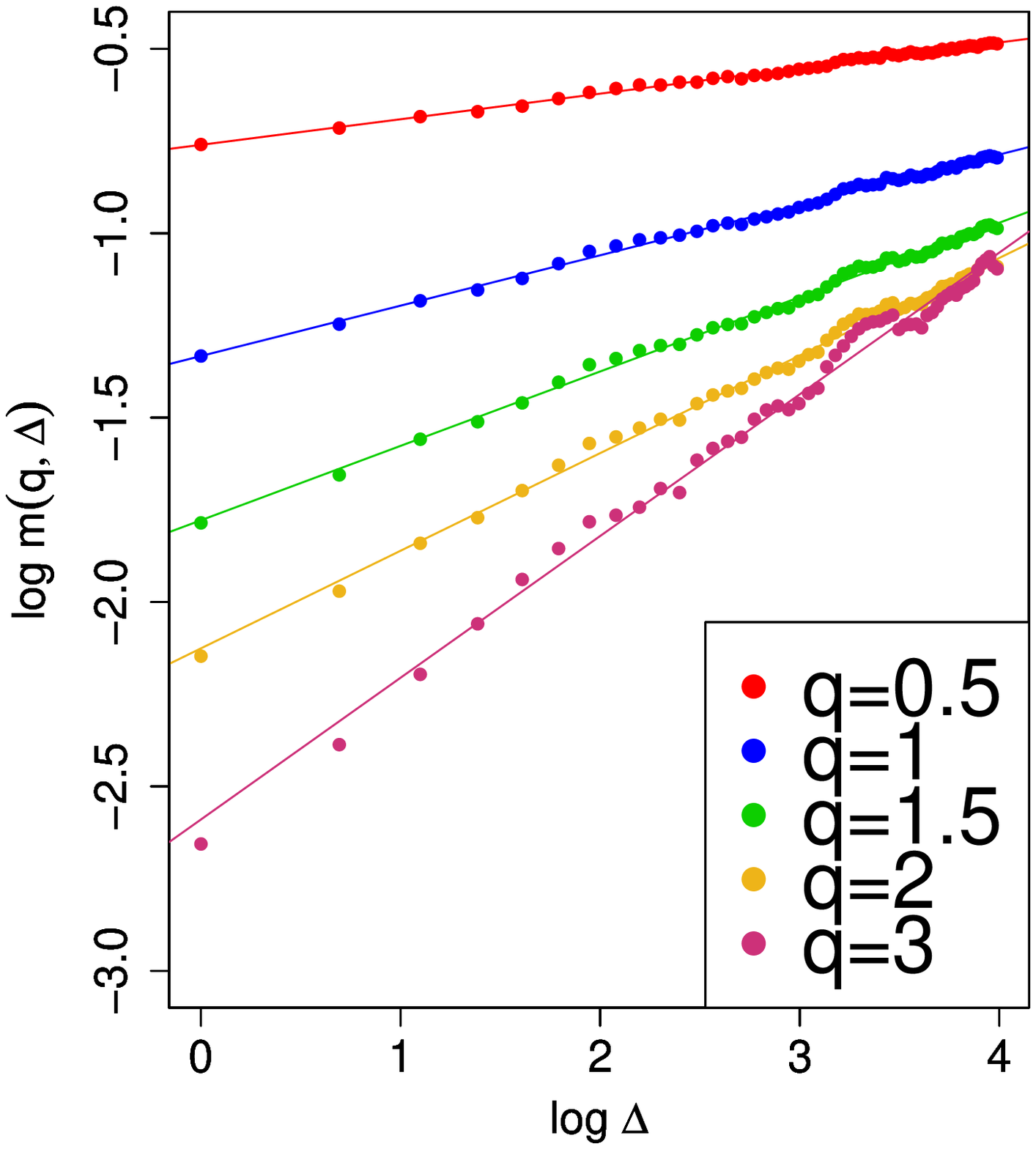} 
	\includegraphics[width=6.2cm]{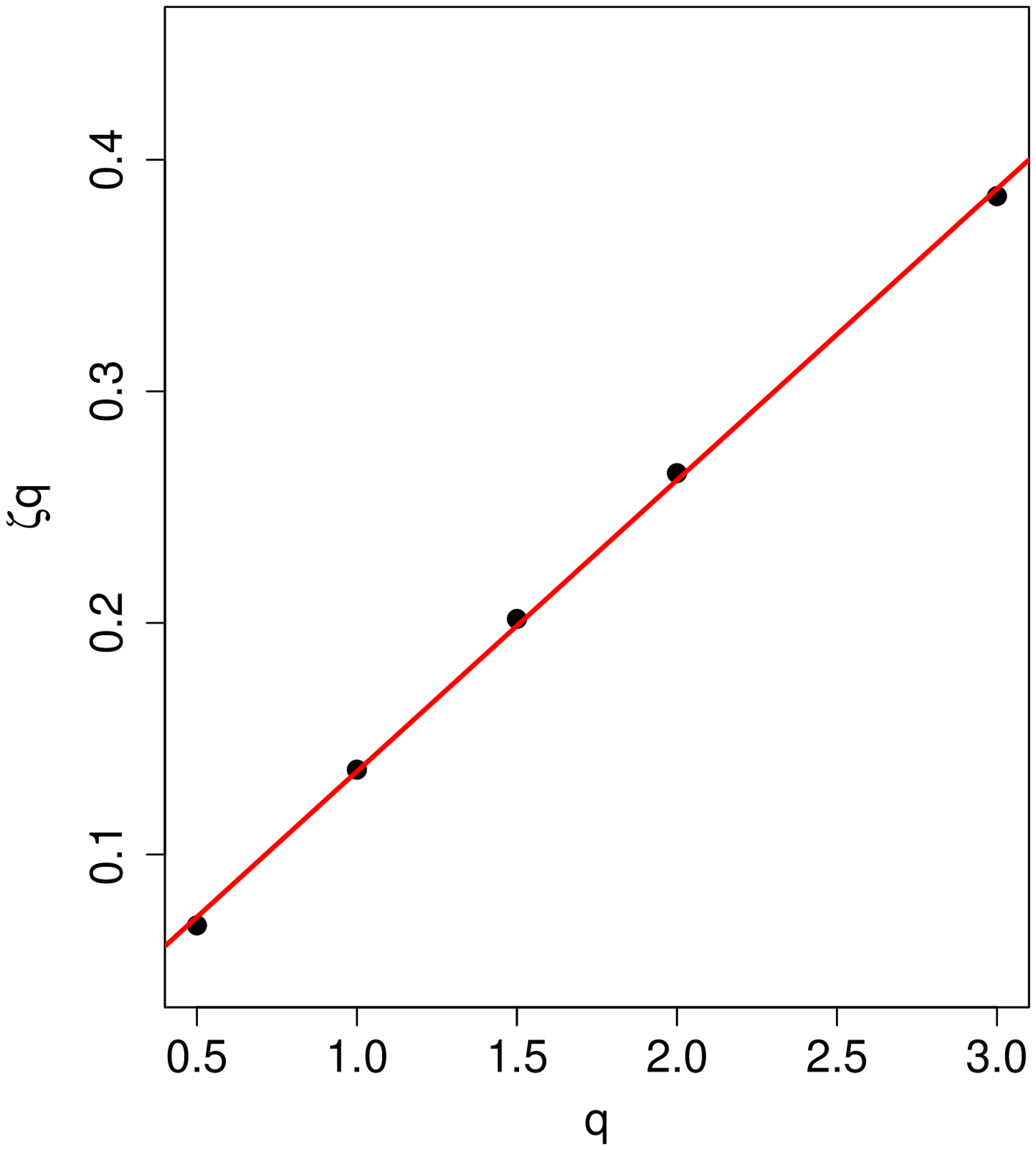}
	\caption{A reproduction of the linear regressions in Gatheral et al.~\cite{GJR} using SPX 5-minute realized volatility from the Oxford-Man Institute's Realized Library.
		Period: 03/01/2000 - 13/07/2018. 
		The regression coefficient (Right): $H =0.1258$.
	}\label{F1}
\end{figure}
Naively, this scaling property together with the above mentioned
stylized fact that the realized variance is approximately log Gaussian suggests a simple dynamics
\begin{equation}\label{M2}
\mathrm{d}\log \sigma^2_u = \eta \mathrm{d}W^H_u, 
\end{equation} 
where $\eta$ is a constant and 
$W^H$ is a fractional Brownian motion\footnote{A fractional Brownian motion $W^H$ is characterized as a continuous centered Gaussian process with $W_0^H=0$ a.s., stationary increments and $E[|W^H_{t +\Delta}-W^H_t|^q] = C_q\Delta^{Hq}$ for any $q, \Delta>0$, where $C_q$ is the absolute $q$th moment of the standard normal distribution. See Mishura~\cite{Mishura} for further detail.
	The fractional Brownian motion in volatility dynamics does not imply an arbitrage opportunity because the asset price process (1) is a local martingale.} with the Hurst parameter $H$.
Note that the estimate $H \approx 0.1$ is not consistent to 
a widespread belief that the volatility is a process of long memory.
Gatheral et al.~\cite{GJR} showed by some simulations that such a ``short memory'' process pretends to be  of long memory.  
They demonstrated also a good prediction performance of this simple model.
The analysis is extended by Bennedsen et al.~\cite{BLP-Decoupling} to a wider set of assets.
The estimate $H \approx 0.1$  means that  the volatility path is rougher than semimartingales,
and is consistent to a power law for the term structure of  implied
volatility skew empirically observed in option markets; see~\cite{ALV,
F11, BFG, F17, FZ, GS, EFGR}. 
A market microstructural foundation of a rough volatility  model is
given by El Euch et al.~\cite{EFR}.

The statement $H\approx 0.1$ by Gatheral et
al.~\cite{GJR} should be, however, understood not as a statistical
estimate but as the proposal of a model which is consistent to a number
of empirical evidences.
In fact, as noted in that paper itself,
what they ``show here is that we cannot find any evidence against
the RFSV~\footnote{Rough Fractional Stochastic Volatility. It
	is a special case of our model (\ref{M3}) below. } model''. 
 Our numerical experiments show that
when using 5-minute realized volatility,
the linear regression methods in  Gatheral et al.~\cite{GJR} or in
Bennedsen et al.~\cite{BLP-Decoupling} often  give almost perfect fit with $H\approx 0.1$ 
irrespectively to the true value of $H$ used to simulate paths;
see Figure~\ref{F2} just for one example, and 
see Westphal~\cite{Westphal} for more  extensive simulation results.
\begin{figure}[t]
	\centering
	\includegraphics[width=6.2cm]{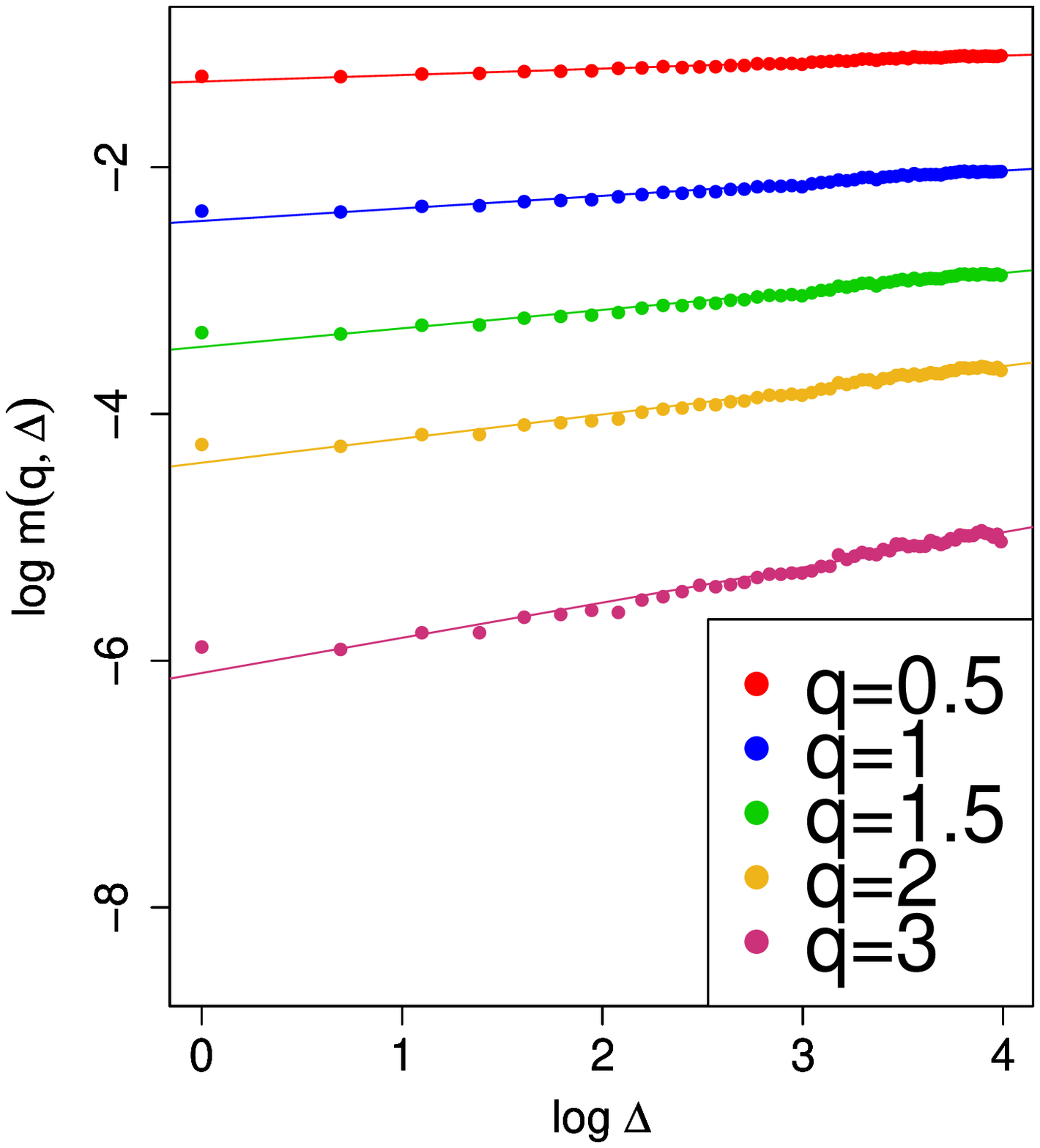} 
	\includegraphics[width=6.2cm]{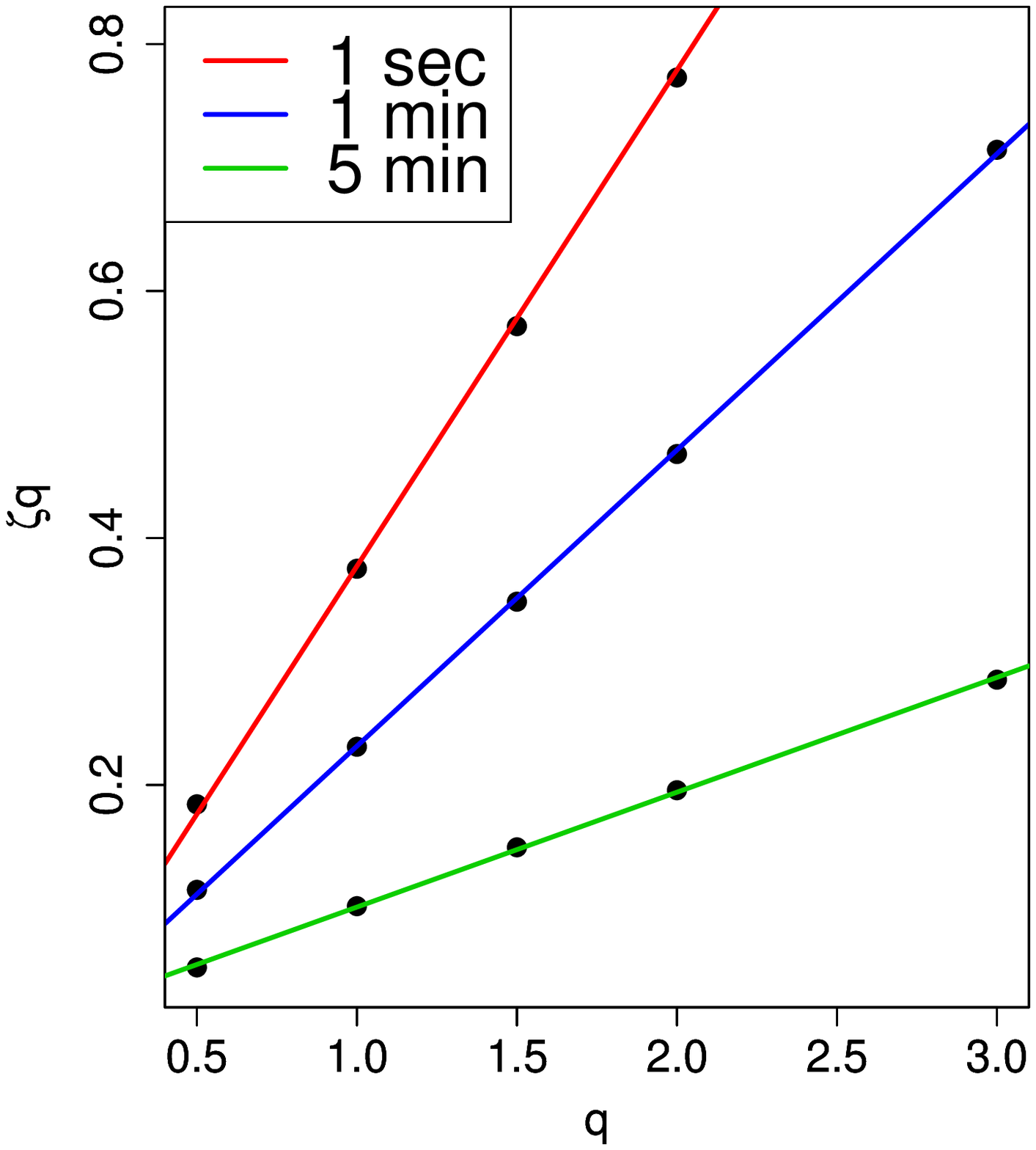} 
	\caption{
		(Left) The linear regression (\ref{LR}) using 5-minute realized volatility of 
		a simulated price path from the model (\ref{M1}) with volatility dynamics
		$ \mathrm{d} \log\sigma_u^2 = 10 (-3.2 -\log \sigma_u^2)\mathrm{d}u + 0.8 \mathrm{d}W_u$, where $W$ is a standard Brownian motion independent of $B$. 
		(Right) For the same simulated path, using realized volatility with different sampling frequencies,
		the regression coefficients $\zeta_q$ of (\ref{LR}) are plotted and regressed on $q$.
		The regression coefficients are $H = 0.4016$ for 1 second, $H = 0.2398$ for 1 minute and
		$H =0.0930 \approx 0.1$ for 5 minutes.
	}\label{F2}
\end{figure}
Figure~\ref{F2} indicates also that this striking phenomenon is due to
the use of a volatility 
proxy; the approximation error of  $\hat{\sigma}_t$ to $\sigma_t$
results in an illusive scaling property. 
These observations from simple numerical experiments bring us a 
question whether the
volatility is really rough, which
the present paper aims at providing the first step to answer.

This is a question about the smoothness of a hidden process. Therefore any nonparametric spot volatility estimation method or filtering approach in the literature is not helpful here; such an estimator is not meant to preserve the regularity of the hidden path. Further, notice that for continuous time models like (\ref{M1}), most of theoretical studies in the high frequency data analysis so far have assumed that the volatility process $\sigma$ is an It\^o semimartingale. This is an indispensable assumption because the analysis is  typically  based on a piecewise constant approximation of $\sigma$, and the path regularity of $\sigma$ determines the convergence rate of the approximation error; we refer again to A\"it-Sahalia and Jacod~\cite{AJ}. There is a work by Rosenbaum~\cite{R08} about fractional volatility models including (\ref{M2}); this however assumes $H \geq 1/2$ a priori  and so, is not helpful here to study whether $H < 1/2$ (rough) or not. 
Other results from high frequency statistics for the model (\ref{M1}) that do not require  $\sigma$ to be an It\^o semimartingale include the most primitive convergence (\ref{L1}), associated central limit theorems by Jacod and Protter~\cite{JP} and Fukasawa~\cite{F10-SPA,F11-AAP}, and some limit theorems for the so-called two-scales, or multi-scales realized volatility that takes the market microstructure noise into account; see A\"it-Sahalia and Jacod~\cite{AJ}.

This paper proposes a novel estimator of the Hurst and diffusion
parameters under a fractional volatility model extending
(\ref{M2}). Taking the difference between $\sigma_t$ and its proxy
$\hat{\sigma}_t$ into account, we derive an estimation function
combining the three ideas: (1)~a normal approximation of the
log-realized volatility estimation error based on the above mentioned
central limit theorem, (2)~a local Gaussian approximation of the
log-realized variance time series, and (3)~a Whittle-type estimation for
high frequency self-similar Gaussian models developed in Fukasawa and
Takabatake~\cite{FT19-BEJ}.
The asymptotic results in this previous work are, however, not directly applicable here because the
observed sequence is not Gaussian but only "locally Gaussian".
The local Gaussian approximation error causes 
several technical difficulties in 
proving the consistency of our estimator.
The consistency result in this paper is the first to show that a
Whittle-type estimation function is effective to a non-Gaussian model 
under high frequency asymptotics.

Our empirical study for major stock indices indicates that $H$ is even smaller than $0.1$; so
our tentative answer to the question is affirmative. 
It is tentative because constructing statistical tests still remains for
future research.
We remark that a model with $H = 0$ formally corresponds to a Gaussian
multiplicative chaos~\cite{Kahane}, or a multifractal process~\cite{BaMu,Bou}.
A framework including those is also a topic for future
research.

The paper is organized as follows.
We propose a model, construct an estimator and state a consistency theorem in Section~\ref{Section.Model.Estimator}, 
examine the finite sample performance of our estimator 
by simulations in Section~\ref{Section.Simulation}, 
and then apply it to the Oxford-Man realized library data to get an estimate of $H$ in Section~\ref{Section.Data.Analysis}.
The proofs are deferred to Appendix.


\section{Model, Quasi-likelihood Estimator, and its Consistency}\label{Section.Model.Estimator}
\subsection{Model}\label{Subsection.Model}
Denote by $(\Omega,\mathcal{F},P,\mathbb{F})$,
$\mathbb{F}=\{\mathcal{F}_u\}_{u\in[0,\infty)}$, a filtered probability
space satisfying the usual condition on which an asset price process $S$
and its volatility process $\sigma$ are defined. Extending the simplest
rough volatility  model (\ref{M2}) and a fractional volatility model of
Comte and Renault~\cite{CR}, we assume the volatility process to satisfy
\begin{equation} \label{M3}
\mathrm{d} \log \sigma^2_u = \kappa_u \mathrm{d}u + \eta \mathrm{d}W^H_u,
\end{equation}
where $\kappa$ is an unknown $\mathbb{F}$-adapted c\`adl\`ag (or c\`agl\`ad) process and $W^H$ is a fractional Brownian motion which is also $\mathbb{F}$-adapted. The parameter to be estimated is $(H, \eta) \in \Theta$, where
$\Theta := \Theta_H \times \Theta_\eta$ is a compact set of the form
$\Theta_H := [H_-,H_+] \subset (0,1]$ and
$\Theta_\eta := [\eta_-,\eta_+] \subset (0,\infty)$.
The true value is denoted by $\vartheta_0 = (H_0,\eta_0)$ and assumed to be an interior point of $\Theta$.
The process $\sigma^2$ is not directly observed and so needs to be estimated from discrete observations of the asset price $S$.
As a proxy of the unobservable $\sigma^2_t$,
we adopt the realized variance  with $m$ equidistant sampling
\begin{equation*} 
\hat{\sigma}^2_t\equiv\hat{\sigma}^2_t(m,\delta):= 
\sum_{j=0}^{m-1} \left|\log S_{(t-1+(j+1)/m)\delta} - \log S_{(t-1+j/m)\delta}\right|^2,\ \ m\in\mathbb{N},\ \ \delta>0,
\end{equation*}
and model directly the law of the proxy error for the log realized variance 
\begin{equation}\label{M4}
\epsilon_t:=\log  \hat{\sigma}^2_t - \log \int_{(t-1)\delta}^{t\delta} \sigma_u^2\, \mathrm{d}u, \ \ t =1, 2, \dots, n+1,
\end{equation} 
as an i.i.d. sequence independent of $W^H$ and normally distributed with mean $0$ and variance $2/m$, where $m$ is the size of intraday price data used to compute $\hat{\sigma}^2_t$. This specification of the law of $\{\epsilon_t\}$ is motivated by the following limit theorem.
\begin{thm}\label{LogRV.Stable.Conv}
	Consider a positive sequence $\{\delta_n\}_{n\in\mathbb{N}}$ and a sequence of natural numbers $\{m_n\}_{n\in\mathbb{N}}$ satisfying $\delta_n\to 0$ and $m_n\to\infty$ as $n\to\infty$, and $\hat\sigma^2_t\equiv\hat\sigma^2_t(m_n,\delta_n)$. Assume that a log-asset price process $\log{S}=A+M$ given by
	\begin{equation}\label{Assumption.logS}
	\mathrm{d} M_u=\sigma_u\, \mathrm{d} B_u,\ \ \mathrm{d} A_u=\psi_u\, \mathrm{d}u,
	\end{equation}
	where $\psi=\{\psi_s\}_{s\in[0,\infty)}$ is a $\mathbb{F}$-adapted locally bounded left-continuous process, $B$ is a standard $\mathbb{F}$-Brownian motion and the volatility process $\sigma^2$ is given in (\ref{M3}). Then we have
	\begin{equation*}
	\left\{
	\sqrt{m_n} \left(\log \hat\sigma^2_t - 
	\log \int_{(t-1)\delta_n}^{t\delta_n} \sigma_u^2\, \mathrm{d}u
	\right)\right\}_{t\in\mathbb{N}}\stackrel{n\to\infty}{\rightarrow} \{\sqrt{2}\xi_t\}_{t\in\mathbb{N}}\ \ \mbox{in law},
	\end{equation*}
	where $\{\xi_t\}_{t\in\mathbb{N}}$ is an i.i.d. standard Gaussian sequence independent of $\mathcal{F}$. 
\end{thm}

The proof of Theorem \ref{LogRV.Stable.Conv} is given in Appendix~\ref{Section.LogRV.Stable.Conv}. Here we give some remarks in order.
\begin{rem}\rm
	We model the volatility dynamics (\ref{M3}) in the business time scale, which means that the time variable $u$ in the model evolves only when a market of the asset is open. Therefore, the volatility is freezing when markets are closed. Including volatility jumps remains for future research.
\end{rem}

\begin{rem}\rm\label{Remark2.Stable.Conv}
	In Theorem~\ref{LogRV.Stable.Conv} and in the sequel, we consider the double high frequency limits $(\delta,1/m)\to(0,0)$.
	For example,  $\delta= 0.04 = 1/250$ corresponding to the 1 day length in a year consisting of 250 business days.
	For the 5-minute realized variance of market data with 6 opening hours, $m = 6\times 60/5 = 72$.
\end{rem}

\begin{rem}\rm
	Daily volatility proxy data including the 5-minute realized variance are readily available thanks to the Oxford-Man Institute, while
	high frequency price data (tick data) are not easily accessible.
	This motivates us to include a proxy as a model element.
	Among many volatility proxies, we adopt the realized variance by the following 4 reasons: i) As mentioned in Introduction, the high frequency limit theorems for the realized variance are valid without assuming the volatility $\sigma$ to be an It\^o semimartingale while those for others are not in general. ii) The asymptotic theory of two-scales and multi-scales realized volatilities assumes the market microstructure noise to be independent of the price and the volatility processes. 
	While this is a popular assumption in the literature, the
	authors do not consider it enough realistic. iii)
	The realized variance with modest frequency like 5 minutes, for which the market microstructure noises are negligible but still high frequency limit theorems are valid, is easy to compute from modest frequency price data that are nowadays easily obtained online for free. 
	iv) As shown in Theorem~\ref{LogRV.Stable.Conv} above, the realized variance with equidistant sampling admits a particularly simple limit law.
	Note that the limit law is different for a different proxy and
	even so  for the realized variance with a different sampling scheme;
	see~\cite{F10-SPA,F11-AAP}.
	It is remarkable that the limit law in Theorem~\ref{LogRV.Stable.Conv} does not depend on $\sigma$, which enables us to quantify the size of the approximation error to $\sigma$ without knowing the exact value of $\sigma$.
\end{rem}

\begin{rem}\rm
	In view of Theorem~\ref{LogRV.Stable.Conv}, for our model (\ref{M4}), more plausible would be  a weaker assumption that the law of $\{\epsilon_t\}_{t\in\mathbb{N}}$ is not exactly but only asymptotically i.i.d. Gaussian with mean $0$ and variance $2/m$.
	We believe that the same quasi-likelihood estimator given below enjoys the same consistency property also under this weaker assumption plus a suitable uniform integrability condition;  we however refrain from increasing the complexity of this already technical and lengthy paper.
\end{rem}

\subsection{Construction of Adapted Whittle Estimator}\label{SubSection.Construction.Estimator}
Here, for a sequence of integers $m_n$ and a positive sequence $\delta_n$, we define a quasi-likelihood estimator of the unknown parameter $\vartheta=(H,\eta)$ based on the log-realized variance increments
\begin{equation*}
Y_t^n := \log \hat{\sigma}^2_{t+1}(m_n,\delta_n) -
\log \hat{\sigma}^2_t(m_n,\delta_n), \ \ t=1,2,\dots, n. 
\end{equation*}
Firstly, we define an estimator of a reparametrized parameter $(H,\nu)$, where $\nu:=\eta\delta_n^{H}\in\Theta_\nu^n:=[\eta_-\delta_n^{H_+},\eta_+\delta_n^{H_-}]$, by 
\begin{align}
&(\widehat{H}_n,\widehat\nu_n) :=\argmin_{(H,\nu)\in\Theta_H\times\Theta_\nu^n} U_n(H,\nu), \nonumber\\
&U_n(H,\nu) := \frac{1}{4\pi}\int_{-\pi}^\pi\left(\log{g_{H,\nu}^n(\lambda)} +\frac{I_n\left(\lambda,\mathbf{Y}_n\right)}{g_{H,\nu}^n(\lambda)}\right)\, \mathrm{d}\lambda, \label{Def.Estimation.Function}
\end{align}
where $\mathbf{Y}_n:=(Y_1^n,\cdots,Y_n^n)$, and $I_n(\cdot,\mathbf{y})$ and $g_{H,\nu}^n$ are a periodogram of $\mathbf{y}\in\mathbb{C}^n$ and an approximate spectral density of $\mathbf{Y}_n$ with respect to the reparametrized parameter $(H,\nu)$ respectively given by
\begin{align}
&I_n(\lambda,\mathbf{y}) :=\frac{1}{2\pi n}\left|\sum_{t=1}^n y_t\exp\left(\sqrt{-1}t\lambda\right)\right|^2,\ \ \mathbf{y}\equiv(y_1,\cdots,y_n)\in\mathbb{C}^n, \label{Def.Periodogram} \\
&g_{H,\nu}^n(\lambda) := \nu^2f_H(\lambda) + \frac{2}{m_n}\ell(\lambda),\hspace{0.2cm}\lambda\in[-\pi,\pi], \nonumber 
\end{align}
where $f_H$ and $\ell$ are given in Appendix~\ref{Appendix.SPD}.
Note that, for each $n\in\mathbb{N}$, the estimator $(\widehat{H}_n,\widehat\nu_n)$ always exists  because $\Theta_H\times\Theta_\nu^n$ is compact. 
Then we define an estimator of the parameter $\eta$ by substituting $\widehat{H}_n$ and $\widehat{\nu}_n$ into the relation $\nu=\eta\delta_n^H$, i.e. an estimator of the original unknown parameter $\vartheta$ is defined by
\begin{equation}\label{Def.Adapted.Whittle.Estimator}
\widehat\vartheta_n := \left(\widehat{H}_n, \widehat\eta_n\right)\hspace{0.2cm}\mbox{with}\hspace{0.2cm}\widehat\eta_n := \delta_n^{-\widehat{H}_n}\widehat\nu_n,\hspace{0.2cm}n\in\mathbb{N}.
\end{equation}
We call the estimator $\widehat\vartheta_n$ as the adapted Whittle
estimator through this paper. 

Now the idea for (\ref{Def.Estimation.Function}) is summarized in the following remark.
\begin{rem}\rm\label{Remark.Idea.QMLE}
	Our idea to derive the approximate likelihood function (\ref{Def.Estimation.Function}) is based on a local approximation of $\mathbf{Y}_n$ by a certain Gaussian vector and the Whittle likelihood of a sequence of the approximate Gaussian vectors. 
	Indeed, the Taylor and the Euler-Maruyama approximations of $\mathbf{Y}_n$ yield
	\begin{equation}\label{Local.Gaussian.Approximation}
	Y_t^n\approx \frac{1}{\delta_n}\int_{t\delta_n}^{(t+1)\delta_n}\eta(W_u^H-W_{u-\delta_n}^H)\, \mathrm{d}u + (\epsilon_{t+1}^n-\epsilon_t^n) =:G_t^n,\ \ t=1,\cdots,n,
	\end{equation}
	as $\delta_n\to 0$, where $\{\epsilon_t^n\}_{t\in\mathbb{Z}}$ is an i.i.d. sequence independent of $W^H$ and normally distributed with mean $0$ and variance $2/m_n$. See Appendix~\ref{Appendix.Approximation.Data} for a precise statement of the above approximation. 
	Furthermore, we can show that a covariance function of the approximate Gaussian vector $\mathbf{G}_n:=(G_1^n,\cdots,G_n^n)$ is characterized by a spectral density $f_{H,\eta}^n$ given by
	\begin{equation*}
	f_{H,\eta}^n(\lambda):= \eta^2\delta_n^{2H}f _H(\lambda) + \frac{2}{m_n}\ell(\lambda),\ \ \lambda\in[-\pi,\pi]. 
	\end{equation*}
	See Appendix~\ref{Appendix.SPD} for more detail. 
	Finally, we adopt the Whittle likelihood of the Gaussian vector $\mathbf{G}_n$, which was investigated in Fukasawa and Takabatake~\cite{FT19-BEJ} under high frequency observations without the noise $\{\epsilon_t^n\}_{t\in\mathbb{N}}$, as an approximate likelihood of $\mathbf{Y}_n$.
\end{rem}
\begin{rem}[\textit{Why we need to reparametrize ?}]\rm
	Under high frequency observations, due to a self-similarity
	property of fractional Gaussian noises, the effects of $\eta$ and $H$
	fuse in the limit and the asymptotic Fisher information matrix becomes
	singular. As a result, it is necessary to reparametrize the parameter
	$\eta$ in order to obtain a limit theorem of estimator. See
	Brouste and Fukasawa~\cite{BF} and Fukasawa and Takabatake
	~\cite{FT19-BEJ} for more details.
\end{rem}

\subsection{Main Theorem}
We state our main theorem in this paper. 
\begin{thm}\label{Consistency.Estimator}
	Assume the true value $\vartheta_0$ is an interior point of $\Theta$ and the following three conditions $(H.\ref{Assumption.HighFrequency})-(H.\ref{Assumption.mn})$ hold:
	\begin{enumerate}[(H.$1$)]
		\item\label{Assumption.HighFrequency}
		$\lim_{n\to\infty}\delta_n=0$ and $\lim_{n\to\infty}m_n=\infty$.
		\item\label{Assumption.Tn} 
		$0<\varliminf_{n\to\infty}T_n\leq \varlimsup_{n\to\infty}T_n <\infty$, where $T_n:=n\delta_n$.
		\item\label{Assumption.mn}
		$\varliminf_{n\to\infty}\inf_{H\in\Theta_H}m_n\delta_n^{2H}=\varliminf_{n\to\infty}m_n\delta_n^{2H_+}>0$.
	\end{enumerate}
	Then a sequence of estimators $\{\widehat\vartheta_n\}_{n\in\mathbb{N}}$ is weakly consistent, i.e. $\widehat\vartheta_n\to\vartheta_0$ in probability as $n\to\infty$.
\end{thm}

The proof of Theorem~\ref{Consistency.Estimator} is deferred to Appendix. Here we make comments on technical difficulties for the proof in the following remark.
\begin{rem}\rm
	One of the difficulties is that the parameter space $\Theta_H\times\Theta_\nu^n$ where the estimation function $U_n(H,\nu)$ is minimized depends on the asymptotic parameter $n\in\mathbb{N}$ and $\lim_{n\to\infty}\Theta_\nu^n=\emptyset$. 
	As a result, $U_n(H,\nu)$ fails to satisfy the identifiability condition of the parameter $(H,\nu)$ in the limit as $n\to\infty$. 
	In order to circumvent this difficulty, we appropriately rescale the estimator $(\widehat{H}_n,\widehat\nu_n)$ and its estimation function $U_n(H,\nu)$, and attempt to find a function which can identify a rescaled parameter in the limit. 
	Actually, we can find a function $U_{n,0}(H,\widetilde\nu)$, where $\widetilde\nu:=\nu\delta_n^{-H_0}\in\Theta_{\widetilde\nu}^n:=[\eta_- \delta_n^{H_+-H_0},\eta_+ \delta_n^{H_- - H_0}]$, which satisfies 
	\begin{equation}\label{Minimization.Rescaled.EF}
	(\widehat{H}_n,\delta_n^{-H_0}\widehat\nu_n) = \argmin_{(H,\widetilde\nu)\in\Theta_H\times\Theta_{\widetilde\nu}^n}U_{n,0}(H,\widetilde\nu),
	\end{equation}
	where 
	\begin{align}
	&U_{n,0}(H,\widetilde\nu) := \frac{1}{4\pi}\int_{-\pi}^\pi\left(\log{h_{H,\widetilde\nu}^n(\lambda)} +\frac{I_n\left(\lambda,\widetilde{\mathbf{Y}}_n\right)}{h_{H,\widetilde\nu}^n(\lambda)}\right)\, \mathrm{d}\lambda, \label{Shadow.Estimation.Function}\\
	&h_{H,\widetilde\nu}^n(\lambda) := \widetilde\nu^2f_H(\lambda) + \frac{2}{m_n\delta_n^{2H_0}}\ell(\lambda) \label{Rescaled.SPD}
	\end{align}
	with $\widetilde{\mathbf{Y}}_n:=\delta_n^{-H_0}\mathbf{Y}_n$. 
	Indeed, $U_n(H,\nu)$ and $U_{n,0}(H,\widetilde\nu)$ are connected by the relation $U_n(H,\nu)=H_0\log\delta_n+U_{n,0}(H,\nu/\delta_n^{H_0})$ so that the estimator $(\widehat{H}_n,\widehat\nu_n)$ also minimizes 
	\begin{equation*}
	(\widehat{H}_n,\widehat\nu_n) = \argmin_{(H,\nu)\in\Theta_H\times\Theta_\nu^n}U_{n,0}(H,\nu/\delta_n^{H_0}).
	\end{equation*}
	As a result, (\ref{Minimization.Rescaled.EF}) follows from the one-to-one correspondence between $\nu$ and $\widetilde\nu$. 
	Then we can show that $U_{n,0}(H,\widetilde\nu)$ and $(\widehat{H}_n,\delta_n^{-H_0}\widehat\nu_n)$ converge to a certain function which can identify the rescaled parameter $(H,\widetilde\nu)$ and to the true value $(H_0,\widetilde\nu_0)=(H_0,\eta_0)$, where denote $\nu_0:=\eta_0\delta_n^{H_0}$ and $\widetilde\nu_0:=\nu_0\delta_n^{-H_0}$, respectively when, at least, the proxy error rapidly vanishes in the sense of (H.\ref{Assumption.mn}). 
	Furthermore, we can also show that the estimator $\widehat\vartheta_n$ converges to $(H_0,\eta_0)$ by using the convergence $(\widehat{H}_n,\delta_n^{-H_0}\widehat\nu_n)$ to $(H_0,\eta_0)$.
	Note that $U_{n,0}(H,\widetilde\nu)$ is not a true estimation function because the true value $H_0$ is used in its definition. It plays, however, the similar role to the usual estimation function due to (\ref{Minimization.Rescaled.EF}). 
	The final remark is that a sequence of rescaled parameter spaces $\{\Theta_{\widetilde\nu}^n\}_{n\in\mathbb{N}}$ converges to the unbounded set $(0,\infty)$ so that several additional cares in the proof are necessary.
\end{rem}
\section{Numerical Study}\label{Section.Simulation}
In this section, we examine the finite sample performance of the adapted Whittle estimator proposed in Section~\ref{SubSection.Construction.Estimator} when the log-volatility dynamics is given by a fractional Ornstein-Uhlenbeck process with mean-reverting property. 
We explain how to simulate a sample path of an asset price process following the fractional volatility model in Section~\ref{Section.Model.Parameters} and how to implement the adapted Whittle estimation in Section~\ref{Section.Implementation.AWE}. We summarize several numerical results in Section~\ref{Section.Numerical.Experiments}.
\subsection{Simulation Method for Asset Price Process}\label{Section.Model.Parameters}
In our numerical studies, we simulate an asset price process whose log-volatility process is given by the fractional Ornstein-Uhlenbeck process, i.e.
\begin{equation}\label{Simulated.Model.Def}
\mathrm{d}\log S_u = \sigma_u\, \mathrm{d}B_u,\ \ \mathrm{d}\log\sigma^2_u = \alpha(c-\log\sigma^2_u)\,\mathrm{d}u + \eta_0\,\mathrm{d}W^{H_0}_u,
\end{equation}
by using the Euler-Maruyama approximation, where $B$ is a Brownian motion independent of $W^{H_0}$. Here we generate the fractional Brownian motion $W^{H_0}$ by using the R-function "SimulateFGN" given in the R-package "FGN". 
We consider the case of $\delta=1/250$ and $T:=n\delta=10$. For the size of the price data $m$ used to compute the realized volatility, we consider three cases: $m=80$ and $80\times 5$ of which the values are corresponding to those of 5-minute and 1-minute realized volatilities respectively.
Moreover, our model parameters are given by $H_0=0.01,0.05,0.1,0.3,0.5,0.7$, $\eta_0=1,2,3$, $\alpha=0.001$, $c=\log\sigma_0^2=-3.2$ and $S_0=100$. 
We generate 100 paths to have 100 samples of the estimator.
\subsection{Implementation of Adapted Whittle Estimator}\label{Section.Implementation.AWE}
Denote by $m\equiv m_n$ and $g_{H,\nu}\equiv g_{H,\nu}^n$ for notational simplicity. 
In order to implement the adapted Whittle estimator, we evaluate the estimation function $U_n(H,\nu)$ by
\begin{equation}\label{Approximate.Estimation.Function}
U_n(H,\nu)\approx\frac{1}{2\pi}\int_\psi^\pi\left(\log{g_{H,\nu}(\lambda)} +\frac{I_n(\lambda,\mathbf{Y}_n)}{g_{H,\nu}(\lambda)}\right)\, \mathrm{d}\lambda +A_{H,\nu}^1\left(\psi\right) +A_{H,\nu}^2\left(\psi\right)
\end{equation}
for sufficiently small $\psi>0$, where the above integral is calculated using the R-function "integrate" and additional correction terms $A_{H,\nu}^1(\psi)$ and $A_{H,\nu}^2(\psi)$ are respectively given by
\begin{align*}
&A_{H,\nu}^1(\psi):=\frac{1}{2\pi}\left(\psi\log(\nu^2C_H) +\psi(\log\psi-1)(1-2H) +\frac{\psi^{2+2H}}{\nu^2C_Hm\pi(2+2H)}\right), \\
&A_{H,\nu}^2(\psi):=\frac{1}{2\pi}\left(a_{H,\nu}(0,\psi)\widehat\gamma_n(0)+2\sum_{\tau=1}^{n-1}a_{H,\nu}(\tau,\psi)\widehat\gamma_n(\tau)\right),\ \ \psi\in(0,\pi],
\end{align*}
with $\widehat\gamma_n(\tau):=\frac{1}{n}\sum_{t=1}^{n-|\tau|}Y_t^nY_{t+|\tau|}^n$ and $a_{H,\nu}(\tau,\psi)\equiv a_{H,\nu}(\tau,\psi,J)$ given by
\begin{equation*}
a_{H,\nu}(\tau,\psi,J):=\frac{1}{2\pi}\sum_{j=0}^J\frac{(-1)^j\tau^{2j}}{(2j)!}\frac{1}{\nu^2C_H}\left(\frac{\psi^{2j+2H}}{2j+2H} -\frac{\psi^{1+2j+4H}}{\nu^2C_Hm\pi(1+2j+4H)}\right)
\end{equation*}
for a sufficiently large $J\in\mathbb{N}$ and each $\tau\in\{0,1,\cdots,n-1\}$. The derivation (\ref{Approximate.Estimation.Function}) is given in Appendix~\ref{Appendix.Numerical.Experiments}. Note that the auto-covariance function $\widehat\gamma$ can be effectively computed using the fast Fourier transform algorithm. 
Moreover, we adopt the Paxson approximation of spectral densities for the spectral density $g_{H,\nu}$ used in (\ref{Approximate.Estimation.Function}), i.e. $g_{H,\nu}$ is approximated by
\begin{align}\label{Paxson.Approximation}
g_{H,\nu}(\lambda)
\approx\nu^2C_H\{2(1-\cos\lambda)\}^2\left\{|\lambda|^{-3-2H}+\sum_{k=1}^Kd_H^1(k,\lambda) +\frac{1}{2}\left(d_H^2(K,\lambda)+d_H^2(K+1,\lambda)\right)\right\} +\frac{2}{m}\ell(\lambda) \nonumber
\end{align}
with a sufficiently large $K\in\mathbb{N}$, where 
\begin{align*}
&d_H^1(x,\lambda):=(2\pi x+\lambda)^{-3-2H}+(2\pi x-\lambda)^{-3-2H}, \\ 
&d_H^2(x,\lambda):=\frac{1}{2\pi(2+2H)}\left\{(2\pi x+\lambda)^{-2-2H}+(2\pi x-\lambda)^{-2-2H}\right\}
\end{align*}
for $x\in(1/2,\infty)$ and $\lambda\in[-\pi,\pi]$; see Fukasawa and Takabatake~\cite{FT19-BEJ-Supplement} for more detail. 
We fix $\psi=10^{-5}$, $K=500$ and $J=20$ in our numerical studies.

Finally, we briefly explain how to numerically evaluate the minimizer $(\widehat{H}_n,\widehat\nu_n)$ of the estimation function $U_n(H,\nu)$. In our numerical studies, we use the R-function "optim" in order to obtain the minimizer and select the option "L-BFGS-B" as the optimization method of $U_n(H,\nu)$. Then we consider the parameter space $\Theta=\Theta_H\times\Theta_\eta=[0.001,0.99]\times[0.1,10]$ and take the true value $(H_0,\nu_0)$, where $\nu_0=\eta_0(1/250)^{H_0}$, as the initial value of the optimization for $U_n(H,\nu)$.
\subsection{The numerical results}\label{Section.Numerical.Experiments}
In Table~\ref{Table.Mean.H} and Table~\ref{Table.Mean.eta}, we give
the mean and variance of $\widehat{H}_n$ and $\widehat\eta_n$
respectively. The tables show that when the Hurst parameter is greater than $0.05$,
both of the Hurst and diffusion parameters are estimated 
reasonably well even with 5-minute realized volatility.
In the case of
$H_0=0.01$, positive and
negative estimation biases are observed for the estimator $\widehat{H}_n$ and
$\widehat\eta_n$ respectively.  
There would be, however, no problem in examining whether the volatility is rough ($H_0< 0.5$ or not) because there are few estimation biases in the case of $H_0\geq 0.05$ and the size of them in the case of $H_0< 0.05$ would  not be too large.
Thus, we conclude that the adapted Whittle estimator gives a reliable answer to our question with data analysis using 5-minute realized volatility.
\ctable[botcap,caption={The mean and variance of the adapted Whittle estimator of the Hurst parameter.},label=Table.Mean.H,pos=H,]{lrrcrrcrr}{}{\FL
	\multicolumn{1}{l}{\bfseries }&\multicolumn{2}{c}{\bfseries $\eta_0$= 1}&\multicolumn{1}{c}{\bfseries }&\multicolumn{2}{c}{\bfseries $\eta_0$= 2}&\multicolumn{1}{c}{\bfseries }&\multicolumn{2}{c}{\bfseries $\eta_0$= 3}\NN
	\cline{2-3} \cline{5-6} \cline{8-9}
	\multicolumn{1}{l}{}&\multicolumn{1}{c}{Mean}&\multicolumn{1}{c}{Variance}&\multicolumn{1}{c}{}&\multicolumn{1}{c}{Mean}&\multicolumn{1}{c}{Variance}&\multicolumn{1}{c}{}&\multicolumn{1}{c}{Mean}&\multicolumn{1}{c}{Variance}\ML
	{\bfseries $H_0$=0.01}&&&&&&&&\NN
	~~1 min &$0.03189$&$0.0006475$&&$0.02437$&$0.0004847$&&$0.02348$&$0.0004543$\NN
	~~5 min&$0.04543$&$0.0011467$&&$0.02659$&$0.0005737$&&$0.02179$&$0.0004091$\ML
	{\bfseries $H_0$=0.05}&&&&&&&&\NN
	~~1 min &$0.06662$&$0.0003567$&&$0.06300$&$0.0002407$&&$0.05853$&$0.0003239$\NN
	~~5 min&$0.06606$&$0.0004830$&&$0.05850$&$0.0002776$&&$0.05405$&$0.0003375$\ML
	{\bfseries $H_0$=0.1}&&&&&&&&\NN
	~~1 min &$0.10717$&$0.0001890$&&$0.10267$&$0.0002609$&&$0.09905$&$0.0002079$\NN
	~~5 min&$0.10527$&$0.0003103$&&$0.09709$&$0.0002675$&&$0.09303$&$0.0002458$\ML
	{\bfseries $H_0$=0.3}&&&&&&&&\NN
	~~1 min &$0.30185$&$0.0002681$&&$0.30102$&$0.0002316$&&$0.29975$&$0.0002678$\NN
	~~5 min&$0.30029$&$0.0005437$&&$0.29672$&$0.0003572$&&$0.29557$&$0.0003542$\ML
	{\bfseries $H_0$=0.5}&&&&&&&&\NN
	~~1 min &$0.50131$&$0.0007236$&&$0.50107$&$0.0003558$&&$0.49946$&$0.0002952$\NN
	~~5 min&$0.49196$&$0.0016275$&&$0.49921$&$0.0006041$&&$0.49863$&$0.0004874$\ML
	{\bfseries $H_0$=0.7}&&&&&&&&\NN
	~~1 min &$0.70793$&$0.0017105$&&$0.70928$&$0.0008751$&&$0.70514$&$0.0005364$\NN
	~~5 min&$0.70212$&$0.0033792$&&$0.71401$&$0.0014436$&&$0.70388$&$0.0014050$\LL
}
\ctable[botcap,caption={The mean and variance of the adapted Whittle estimator of the diffusion parameter.},label=Table.Mean.eta,pos=H,]{lrrcrrcrr}{}{\FL
	\multicolumn{1}{l}{\bfseries }&\multicolumn{2}{c}{\bfseries $\eta_0$= 1}&\multicolumn{1}{c}{\bfseries }&\multicolumn{2}{c}{\bfseries $\eta_0$= 2}&\multicolumn{1}{c}{\bfseries }&\multicolumn{2}{c}{\bfseries $\eta_0$= 3}\NN
	\cline{2-3} \cline{5-6} \cline{8-9}
	\multicolumn{1}{l}{}&\multicolumn{1}{c}{Mean}&\multicolumn{1}{c}{Variance}&\multicolumn{1}{c}{}&\multicolumn{1}{c}{Mean}&\multicolumn{1}{c}{Variance}&\multicolumn{1}{c}{}&\multicolumn{1}{c}{Mean}&\multicolumn{1}{c}{Variance}\ML
	{\bfseries $H_0$=0.01}&&&&&&&&\NN
	~~1 min &$0.8014$&$0.0438057$&&$1.741$&$0.140862$&&$2.687$&$0.296818$\NN
	~~5 min&$0.7293$&$0.0459881$&&$1.728$&$0.147558$&&$2.757$&$0.250566$\ML
	{\bfseries $H_0$=0.05}&&&&&&&&\NN
	~~1 min &$0.9895$&$0.0021232$&&$2.004$&$0.004104$&&$3.114$&$0.038000$\NN
	~~5 min&$1.0111$&$0.0049482$&&$2.067$&$0.010018$&&$3.223$&$0.057503$\ML
	{\bfseries $H_0$=0.1}&&&&&&&&\NN
	~~1 min &$1.0217$&$0.0006927$&&$2.047$&$0.002700$&&$3.091$&$0.004235$\NN
	~~5 min&$1.0341$&$0.0007719$&&$2.063$&$0.002196$&&$3.124$&$0.004053$\ML
	{\bfseries $H_0$=0.3}&&&&&&&&\NN
	~~1 min &$1.0101$&$0.0047328$&&$2.012$&$0.016563$&&$3.026$&$0.046931$\NN
	~~5 min&$1.0117$&$0.0064975$&&$1.993$&$0.020801$&&$2.996$&$0.052943$\ML
	{\bfseries $H_0$=0.5}&&&&&&&&\NN
	~~1 min &$1.0175$&$0.0120100$&&$2.016$&$0.032845$&&$3.008$&$0.068443$\NN
	~~5 min&$0.9987$&$0.0181695$&&$2.009$&$0.040863$&&$3.009$&$0.085446$\ML
	{\bfseries $H_0$=0.7}&&&&&&&&\NN
	~~1 min &$1.0584$&$0.0338098$&&$2.114$&$0.090472$&&$3.109$&$0.123645$\NN
	~~5 min&$1.0582$&$0.0504322$&&$2.160$&$0.136346$&&$3.120$&$0.231304$\LL
}
\section{Application to Daily Realized Volatility Data of Stock Indices}\label{Section.Data.Analysis}
In this section, we apply the adapted Whittle estimator to the 5-minute
daily realized volatility data for  several major
stock indices provided by the Oxford-Man realized library. 
We give the estimated values 
in Section~\ref{Section.Summary.Empirical.Study} and give
an additional discussion in Section~\ref{Section.Additional.Discussion}.
\subsection{Estimation Results}\label{Section.Summary.Empirical.Study}
First of all, we make several remarks on the optimization of the estimation function. 
In our data analysis, we use the same implementation and optimization methods of the estimation function $U_n(H,\nu)$ and the same parameter space $\Theta=\Theta_H\times\Theta_\eta$ mentioned in Section~\ref{Section.Implementation.AWE}.
Then we calculate the optimal value in the candidates of the estimated values each of which is obtained from the optimization method starting at each initial value $(H_{ini},\nu_{ini})$ with $H_{ini}\in\{0.01,0.05,1,2,3,4,5,6,7,8,9\}$ and $\nu_{ini}\in\{0.5,1.5,2.5,3.5\}$.

Next, we briefly explain how to compute the value of $m$ which is the size of price data used to compute the 5-minute daily realized volatility for each stock index. 
In our data analysis below, we consider the following 5 stock indices: S\&P 500, FTSE 100, Nikkei 225, DAX, Russell 3000.
Then we can easily calculate the value of $m$ for each stock index since we know the opening hours of the markets are from 9:30 to 16:00 for S\&P 500, from 8:00 to 16:30 for FTSE 100, from 9:00 to 11:30 and from 12:30 to 15:00 for Nikkei 225, from 9:00 to 17:40 for DAX, from 9:30 to 16:00 for Russell 3000, see Remark~\ref{Remark2.Stable.Conv} for an example of computation of $m$. 
In the first row of Table~\ref{Table.Data}, we summarize the values of $m$ for the stock indices. 
We give the estimated values  $(\widehat{H}_n,\widehat\eta_n)$
in Table~\ref{Table.Data}. 
For all indices, 
the Hurst parameter is estimated between 0.02 and
0.06. Our data analysis suggests $H < 0.5$, that is, 
the volatility is indeed rough; 
it is even
rougher than claimed in 
Gatheral et al.~\cite{GJR} and Bennedsen et al.~\cite{BLP-Decoupling}.
It is noteworthy that the estimate $H < 0.1$ is consistent to 
the calibrated parameters from the option market data in
Bayer et al.~\cite{BFG}.
\ctable[botcap,caption={Estimated value of the adapted Whittle estimator $(\widehat{H}_n,\widehat\eta_n)$ of major stock indices for the period: 02/01/2008-29/12/2017.},label=Table.Data,pos=H,]{lcccccrrcrrcrr}{}{\FL
	\multicolumn{1}{l}{}&\multicolumn{1}{c}{SPX 500}&\multicolumn{1}{c}{FTSE 1000}&\multicolumn{1}{c}{Nikkei 225}&\multicolumn{1}{c}{DAX}&\multicolumn{1}{c}{Russell 3000}\ML
	$m$&$78$&$102$&$60$&$105$&$78$\NN
	$\widehat{H}$&$ 0.04272$&$  0.02255$&$ 0.05928$&$  0.03551$&$ 0.03926$\NN
	$\widehat\eta$&$ 2.53112$&$  2.89098$&$ 2.01702$&$  2.21585$&$ 2.39789$\LL
}
\subsection{Additional Discussion}\label{Section.Additional.Discussion}
In this subsection, we check whether the estimated values given 
in Section~\ref{Section.Summary.Empirical.Study} are adequate from
another aspect. 
More specifically, we apply the linear regression method of Gatheral et
al.~\cite{GJR} 
to simulated 5-minute realized volatility data for which 
the Hurst and diffusion parameters are chosen to be
the estimated values.

In Figure~\ref{Additional.Check1}, we
compare the linear regression results of
SPX and simulated data for the same period of the estimation results given in Table~\ref{Table.Data}.
Taking into account the estimation bias of the adapted Whittle estimator
mentioned in Section~\ref{Section.Numerical.Experiments}, we used
slightly smaller value of the Hurst parameter than its estimated value given in Section
\ref{Section.Summary.Empirical.Study}. 
We obtained similar figures and linear regression coefficients of
$\zeta_q$ against $q$ from the SPX and simulated 5-minute realized
volatilities. 
In particular, we confirm that the linear regression method of Gatheral
et al.~\cite{GJR} does not give a proper estimate of $H$ and 
our estimated value of $(H,\eta)$ does not contradict
the analysis in Gatheral
et al.~\cite{GJR}.
\section{Conclusion}
We have questioned whether the volatility is really rough, that is,
whether the Hurst index of the fractional Brownian motion 
driving the volatility process is smaller than 0.5
or not.
We have proposed an estimator for the Hurst and diffusion
parameters under a fractional stochastic volatility model (\ref{M3}).
We have proved its consistency under high frequency asymptotics.
We have also confirmed by numerical simulations its reasonably good performance
with finite samples. The estimated 
Hurst parameters for various stock
indices and periods are all smaller than 0.06,
indicating that the volatility is rough.
This is however 
a tentative answer to our question; in particular constructing
statistical tests remains for future research.
\begin{figure}[H]
	\begin{tabular}{cc}
		\begin{minipage}[t]{0.45\hsize}
			\centering
			\includegraphics[width=6.2cm]{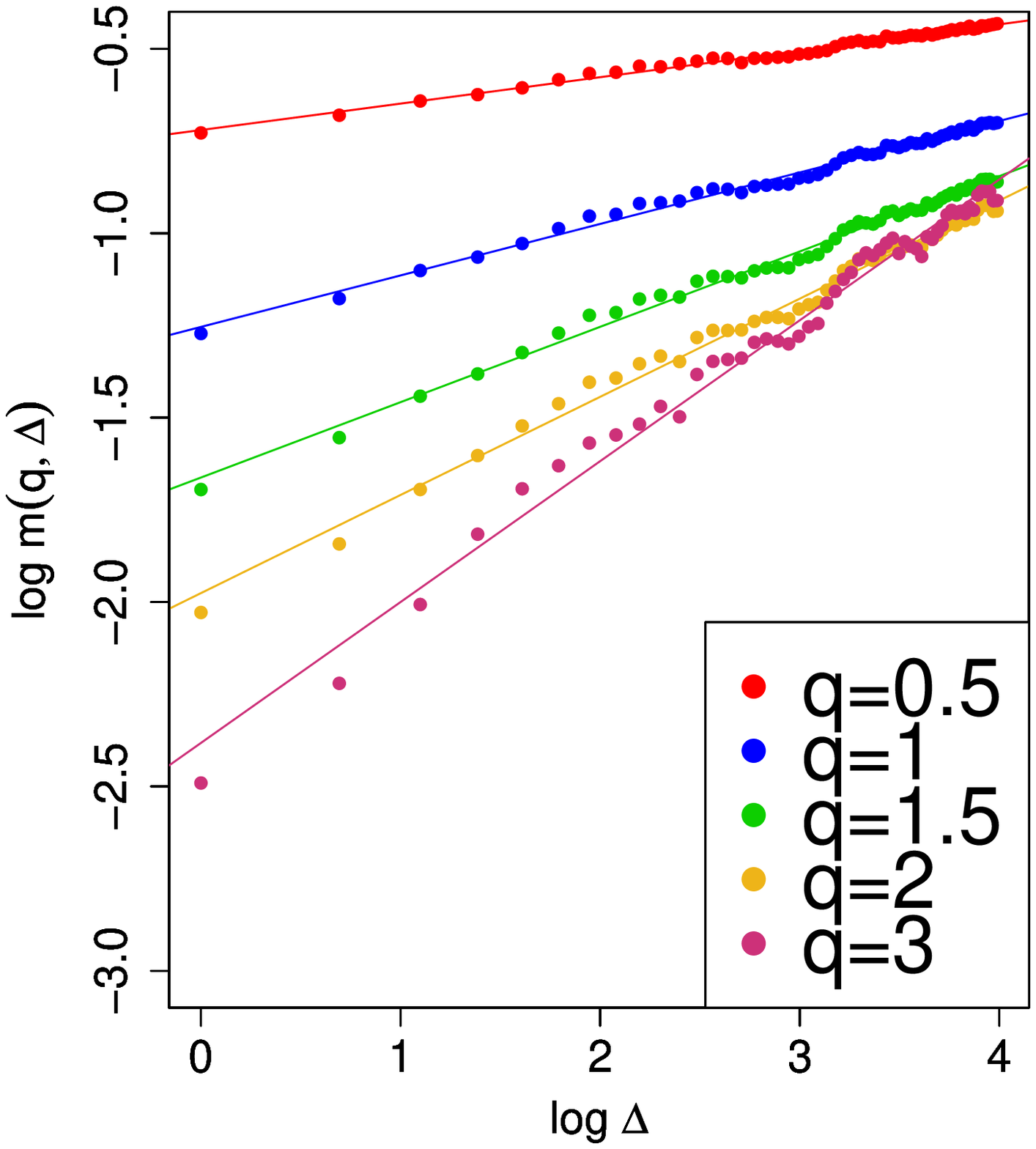}
		\end{minipage} &
		
		\begin{minipage}[t]{0.45\hsize}
			\centering
			\includegraphics[width=6.2cm]{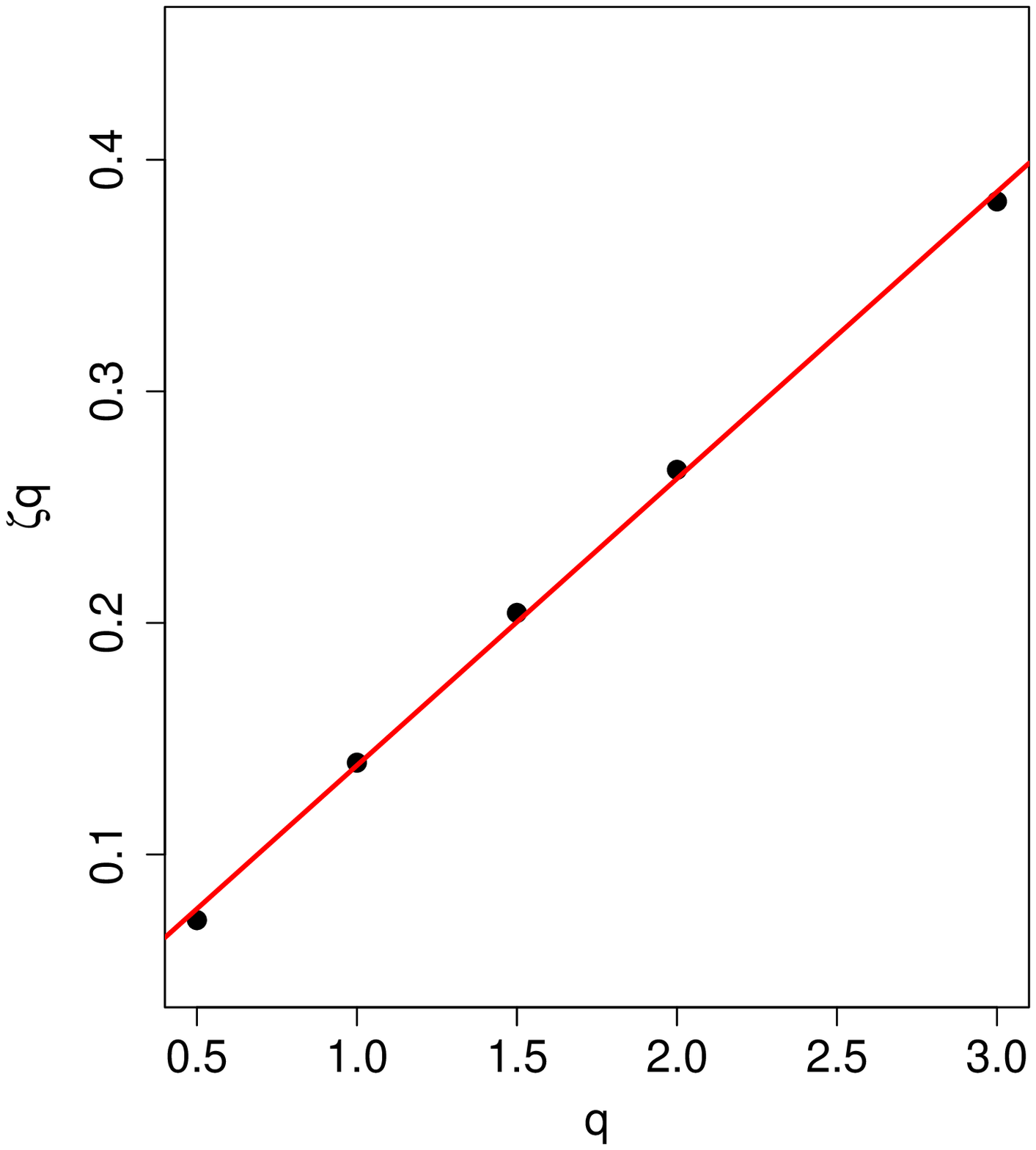}
		\end{minipage} \\
		
		\begin{minipage}[t]{0.45\hsize}
			\centering
			\includegraphics[width=6.2cm]{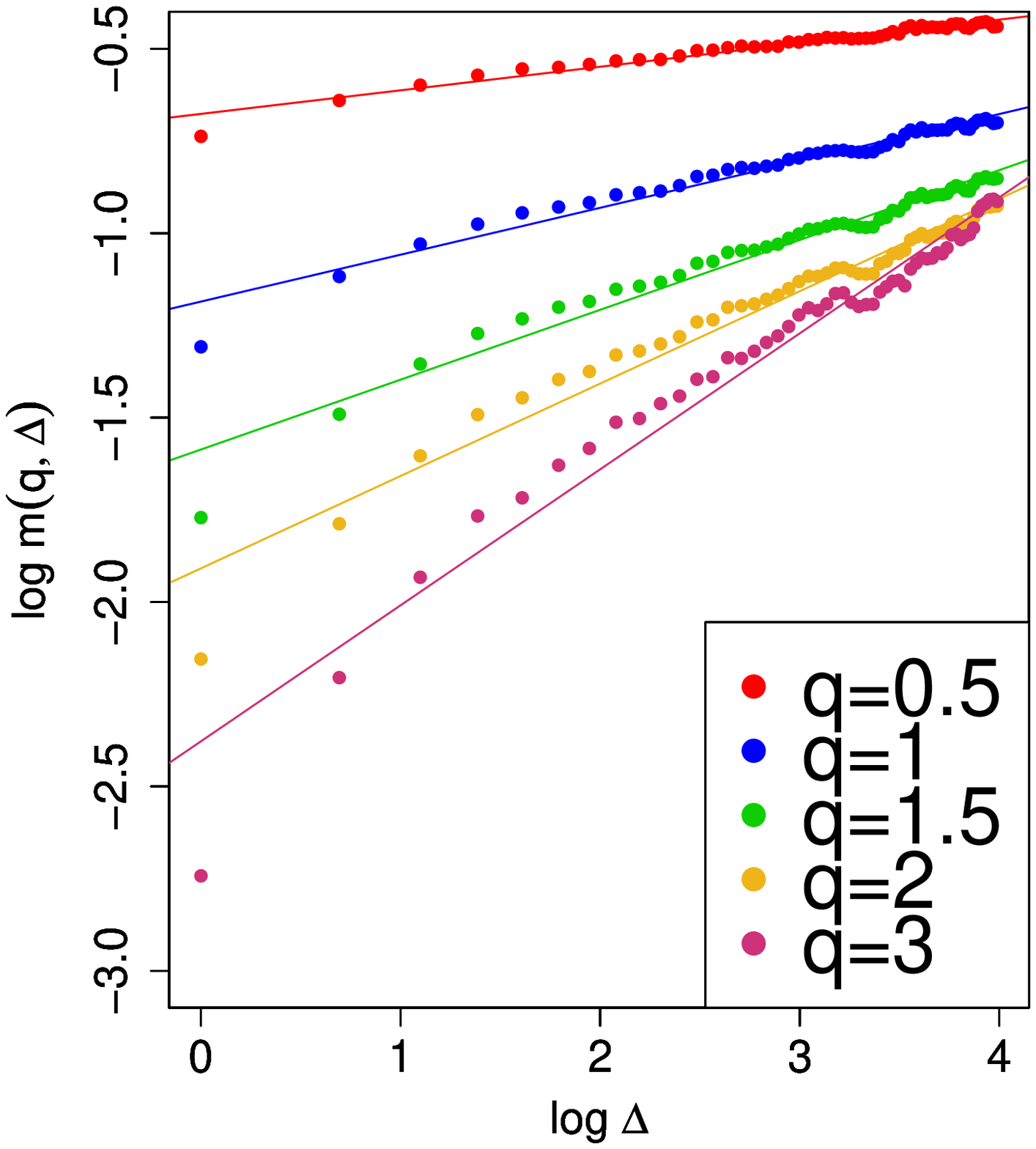}
		\end{minipage} &
		
		\begin{minipage}[t]{0.45\hsize}
			\centering
			\includegraphics[width=6.2cm]{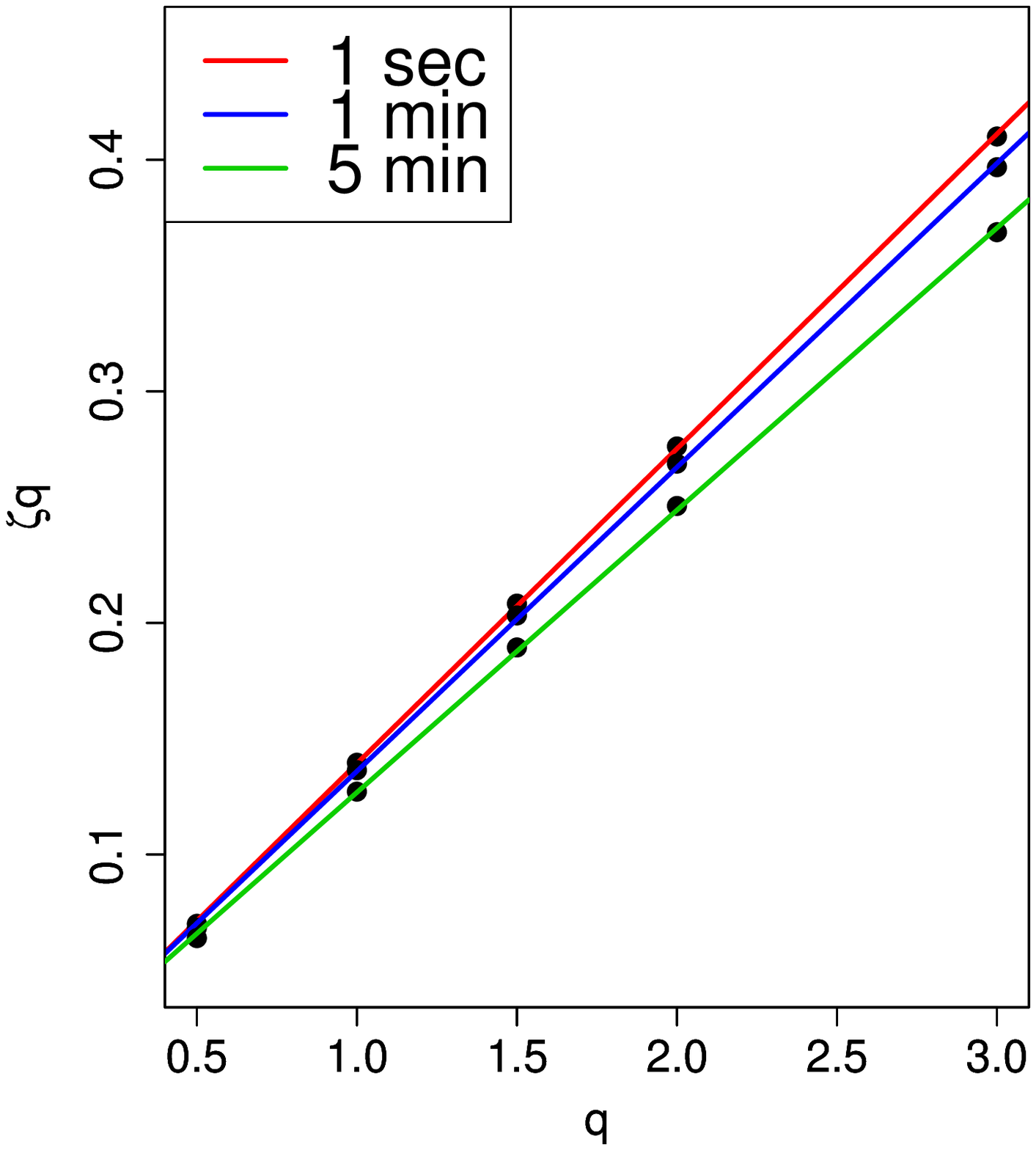}
		\end{minipage}
	\end{tabular}
	\caption{The upper figure is a reproduction of the linear regressions in Gatheral et al.~\cite{GJR} using SPX 5-minute realized volatility from the Oxford-Man Institute's Realized Library. Period: 02/01/2008 - 29/12/2017. The regression coefficient is $H =0.1215$. The lower left figure is the linear regression (\ref{LR}) using 5-minute realized volatility of a simulated price path from the model (\ref{Simulated.Model.Def}) whose parameters are $H=0.03$, $\eta=2.5$, $\alpha=0.005$, $c=-3.2$. For the same simulated path, using realized volatility with different sampling frequencies, the regression coefficients $\zeta_q$ of (\ref{LR}) are plotted and regressed on $q$ in the lower right figure. 
		The regression coefficients are $H = 0.1359$ for 1 second, $H = 0.1313$ for 1 minute and $H =0.1219$ for 5 minutes.
	}
	\label{Additional.Check1}
\end{figure}

\appendix
\section{Notation}
In this section, we summarize notation used throughout the appendix in this paper. 
\subsection{Notation of Bilinear Form}
Denote by $L^1[-\pi,\pi]$ the set of the Lebesgue integrable functions on $[-\pi,\pi]$. Let $\mathbf{x}, \mathbf{y}\in\mathbb{R}^n$ and $k\in L^1[-\pi,\pi]$ be an even function. We define a symmetric bilinear form of $\mathbf{x}$ and $\mathbf{y}$ with a certain symmetric $n\times n$-matrix $\Sigma_n(k)$ by
\begin{equation*}
B_n(\mathbf{x}, \mathbf{y}, k) := \frac{1}{2\pi n}\mathbf{x}^{\mathrm{T}}\Sigma_n(k)\mathbf{y},
\end{equation*}
where $\mathbf{x}^\mathrm{T}$ denotes the transpose of vector $\mathbf{x}$ and $\Sigma_n(k)$ denotes a symmetric matrix whose $(i,j)$-element is given by the $(i-j)$th Fourier coefficient of $k$, denoted by $\widehat{k}(i-j)$, for each $i,j\in\{1,\cdots,n\}$, i.e. 
\begin{equation*}
\widehat{k}(\tau):=\int_{-\pi}^\pi e^{\sqrt{-1}\tau\lambda}k(\lambda)\, \mathrm{d}\lambda,\hspace{0.2cm}\tau\in\mathbb{Z}.
\end{equation*}
In particular, we denote by $Q_n(\mathbf{x}, k):=B_n(\mathbf{x}, \mathbf{x}, k)$. Note that
\begin{equation*}
Q_n(\mathbf{x}, k)=\int_{-\pi}^\pi I_n(\lambda,\mathbf{x})k(\lambda)\, \mathrm{d}\lambda
\end{equation*}
holds, where the periodogram $I_n(\lambda,\mathbf{x})$ is defined in (\ref{Def.Periodogram}).
\subsection{Notation of Stochastic Sequences}
Set $\Lambda_n:=\{1,2,\cdots,n\}$. Denote by $(\Omega,\mathcal{F},P)$ a probability space on which the sequence of observations $\{\mathbf{Y}_n\}_{n\in\mathbb{N}}$ is defined and by $\|\cdot\|_p$ , $p\in[1,\infty]$, the $L^p$-norm on the probability space. Furthermore, we denote $\Delta X_t:=X_{t+1}-X_t$ for a discrete-time stochastic process $X=\{X_t\}_{t\in\mathbb{N}}$ and $t\in\mathbb{N}$, and
\begin{align*}
	&\mathbf{Y}_n^\dagger:=(Y_1^{n,\dagger},Y_2^{n,\dagger},\cdots,Y_n^{n,\dagger}) \hspace{0.2cm}\mbox{with}\hspace{0.2cm}Y_t^{n,\dagger}\equiv Y_t^{n,\dagger,\vartheta}:=\Delta\left(\log\int_{(t-1)\delta_n}^{t\delta_n}\sigma^2_u\, \mathrm{d}u\right), \\
	&\mathbf{G}_n^\dagger:=(G_1^{n,\dagger},G_2^{n,\dagger},\cdots,G_n^{n,\dagger}) \hspace{0.2cm}\mbox{with}\hspace{0.2cm}G_t^{n,\dagger}\equiv G_t^{n,\dagger,\vartheta}:=\Delta\left(\frac{1}{\delta_n}\int_{(t-1)\delta_n}^{t\delta_n}\eta W_u^H\, \mathrm{d}u\right),\\
	&\mathbf{V}_n:= (V_1^n,V_2^n,\cdots,V_n^n) \hspace{0.2cm}\mbox{with}\hspace{0.2cm}V_t^n\equiv V_t^{n,\vartheta}:=\Delta\left(\frac{1}{\delta_n}\int_{(t-1)\delta_n}^{t\delta_n}\log\sigma^2_u\, \mathrm{d}u\right), 
\end{align*}
Recall that
\begin{align*}
&\mathbf{Y}_n:=\left(Y_1^n, Y_2^n, \cdots, Y_n^n\right)\hspace{0.2cm}\mbox{with}\hspace{0.2cm}Y_t^n:=\Delta\left(\log\int_{(t-1)\delta_n}^{t\delta_n}\sigma^2_u\, \mathrm{d}u\right)+\Delta\epsilon_t^n, \\
&\mathbf{G}_n:= \left(G_1^n, G_2^n, \cdots, G_n^n\right)\hspace{0.2cm}\mbox{with}\hspace{0.2cm}G_t^n:=G_t^{n,\dagger}+\Delta\epsilon_t^n,
\end{align*}
where $\{\epsilon_t^n\}_{t\in\mathbb{Z}}$ is an i.i.d. sequence independent of $W^H$ and normally distributed with mean $0$ and variance $2/m_n$. 
For each $K\in\mathbb{N}$ and $\mathbf{p}\equiv(p_1,\cdots,p_K)\in\mathbb{N}^K$, denote by $\mathbf{p}!:=\prod_{k=1}^K p_k!$, $|\mathbf{p}|:=p_1+\cdots+p_K$ and
\begin{equation*}
	\mathbf{Z}_n^\mathbf{p}:=\left(Z_1^{n,\mathbf{p}}, Z_2^{n,\mathbf{p}}, \cdots, Z_{n+1}^{n,\mathbf{p}}\right),\ \ \mathbf{W}_n^\mathbf{p}:=\left(W_1^{n,\mathbf{p}}, W_2^{n,\mathbf{p}}, \cdots, W_{n+1}^{n,\mathbf{p}}\right),
\end{equation*}
where
\begin{align*}
	&Z_t^{n,p_k}:=\frac{1}{\delta_n}\int_{(t-1)\delta_n}^{t\delta_n}\left(\log\sigma^2_u - \log\sigma^2_{(t-1)\delta_n}\right)^{p_k}\, \mathrm{d}u,\ \ Z_t^{n,\mathbf{p}} := \prod_{k=1}^K Z_t^{n,p_k},\\
	&W_t^{n,p_k}:=\frac{1}{\delta_n}\int_{(t-1)\delta_n}^{t\delta_n}\eta^{p_k}\left(W_u^H - W_{(t-1)\delta_n}^H\right)^{p_k}\, \mathrm{d}u,\ \ W_t^{n,\mathbf{p}} := \prod_{k=1}^K W_t^{n,p_k}.
\end{align*} 
Note that for each $t\in\mathbb{Z}$, $P$-a.s. $\omega\in\Omega$,
\begin{align}
	W_t^{n,\mathbf{p}}(\omega) =&\prod_{k=1}^K\int_0^1\eta^{p_k}\left(W_{(u+t-1)\delta_n}^H(\omega) -W_{(t-1)\delta_n}^H(\omega)\right)^{p_k}\, \mathrm{d}u \nonumber \\
	=&\eta^{|\mathbf{p}|}\int_{[0,1]^K}\prod_{k=1}^K\left(W_{(u_k+t-1)\delta_n}^H(\omega) - W_{(t-1)\delta_n}^H(\omega)\right)^{p_k}\, \mathrm{d}\mathbf{u},\hspace{0.2cm}\mathbf{u}\equiv(u_1,\cdots,u_K), \label{rem.change.var}
\end{align}
follows from the change of variables and Fubini's theorem. Moreover, the H\"older continuity of the fractional Brownian motion and $\log\sigma^2$ yield that for any $\psi>0$,
\begin{equation}\label{Asymptotic.Order.Z.W}
	\max_{t\in\Lambda_{n+1}}\left\|W_t^{n,\mathbf{p}}\right\|_\infty =o\left(\delta_n^{|\mathbf{p}|H-\psi}\right),\ \ \max_{t\in\Lambda_{n+1}}\left\|Z_t^{n,\mathbf{p}}\right\|_\infty =o\left(\delta_n^{|\mathbf{p}|H-\psi}\right)\ \ \mbox{as}\ \ n\to\infty.
\end{equation}
\section{Approximation of Data}\label{Appendix.Approximation.Data}
The following proposition gives a precise statement of the approximation (\ref{Local.Gaussian.Approximation}) in Remark~\ref{Remark.Idea.QMLE}, which follows from a Taylor expansion of $\mathbf{Y}_n$ around the Gaussian vector $\mathbf{G}_n$ under high frequency observations, i.e. $\delta_n\to 0$.
\begin{prop}\label{Expansion.Formula.Y}
	For any $\psi\in(0,H)$ and $J\in\mathbb{N}$, there exists a positive random variable $M\equiv M(\psi,J,T,\vartheta)$, which is independent of the asymptotic parameter $n\in\mathbb{N}$, such that 
	\begin{equation}\label{Approximation.Diff.Y}
	\max_{t\in\Lambda_n}\left|Y_t^n -G_t^n -\sum_{j=2}^J\sum_{k=1}^j\frac{(-1)^{k-1}}{k}\sum_{\mathbf{p}\in\mathbb{N}^k, |\mathbf{p}|=j} \frac{1}{\mathbf{p}!}\Delta W_t^{n,\mathbf{p}}\right| \leq M\cdot\delta_n^{\min\{1,(J+1)H-\psi\}}
	\end{equation}
	holds $P$-a.s. for sufficiently small $\delta_n$.
\end{prop}

Note that the lhs of the inequality in Proposition B.1 is dominated as follows:
\begin{align}
	&\max_{t\in\Lambda_n}\left|Y_t^n -G_t^n -\sum_{j=2}^J\sum_{k=1}^j\frac{(-1)^{k-1}}{k}\sum_{\mathbf{p}\in\mathbb{N}^k, |\mathbf{p}|=j} \frac{1}{\mathbf{p}!}\Delta W_t^{n,\mathbf{p}}\right| \label{Approximation.Three.Terms} \\
	\leq&\max_{t\in\Lambda_n}\left|Y_t^{n,\dagger} -V_t^n -\sum_{j=2}^J\sum_{k=1}^j\frac{(-1)^{k-1}}{k}\sum_{\mathbf{p}\in\mathbb{N}^k, |\mathbf{p}|=j} \frac{1}{\mathbf{p}!}\Delta Z_t^{n,\mathbf{p}}\right| \nonumber \\
	&+\max_{t\in\Lambda_n}\left|V_t^n -G_t^{n,\dagger}\right| +\max_{t\in\Lambda_n}\left|\sum_{j=2}^J\sum_{k=1}^j\frac{(-1)^{k-1}}{k}\sum_{\mathbf{p}\in\mathbb{N}^k, |\mathbf{p}|=j} \frac{1}{\mathbf{p}!}\Delta\left(Z_t^{n,\mathbf{p}} -W_t^{n,\mathbf{p}}\right)\right|. \nonumber
\end{align}
In the rest of this section, we evaluate the asymptotic order of the three terms in the rhs of (\ref{Approximation.Three.Terms}) when $\delta_n\to 0$. 
At first, we treat the first term in (\ref{Approximation.Three.Terms}) in the following lemma.
\begin{lem}\label{Taylor.Approximation}
	For any $\psi\in(0,H)$ and $J\in\mathbb{N}$, there exists a positive random variable $M\equiv M(\psi,J,T,\vartheta)$, which is independent of the asymptotic parameter $n\in\mathbb{N}$, such that 
	\begin{equation*}
		\max_{t\in\Lambda_n}\left|Y_t^{n,\dagger} -V_t^n -\sum_{j=2}^J\sum_{k=1}^j\frac{(-1)^{k-1}}{k}\sum_{\mathbf{p}\in\mathbb{N}^k, |\mathbf{p}|=j} \frac{1}{\mathbf{p}!}\Delta Z_t^{n,\mathbf{p}}\right|\leq M\cdot\delta_n^{(J+1)H-\psi}
	\end{equation*}
	holds $P$-a.s. for sufficiently small $\delta_n$.
\end{lem}
\begin{proof}
	At first, Taylor's theorem yields that any infinitely differentiable function $f$ on an $\epsilon$-open ball $B_\epsilon(a)$ at the point of $a\in\mathbb{R}$ is expanded by
	\begin{equation*}
		f(x)=f(a)+\sum_{j=1}^J\frac{f^{(j)}(a)}{j!}x^j +(x-a)^{J+1}\int_0^1\frac{(1-z)^J}{J!} f^{(J+1)}\left(a+z(x-a)\right)\, \mathrm{d}z
	\end{equation*}
	for each $x\in B_\epsilon(a)$ and $J\in\mathbb{N}$. Moreover, if the function $f$ and its derivatives of any order are also continuous on $\overline{B_\epsilon(a)}$, then it holds that
	\begin{equation}\label{Taylor.Remainder.Bounded}
		\sup_{x\in\overline{B_\epsilon(a)}}\left|\int_0^1\frac{(1-z)^J}{J!} f^{(J+1)}\left(a+z(x-a)\right)\, \mathrm{d}z\right|<\infty.
	\end{equation}
	Therefore, using the H\"older continuity of $\log\sigma^2$, we can derive the following Taylor's expansion:
	\begin{align}
		\log\left[\frac{1}{\delta_n}\int_{(t-1)\delta_n}^{t\delta_n}\sigma^2_u\,\mathrm{d}u\right] 
		&= \log\sigma^2_{(t-1)\delta_n} +\log\left[\frac{1}{\delta_n}\int_{(t-1)\delta_n}^{t\delta_n}e^{\log\sigma^2_u-\log\sigma^2_{(t-1)\delta_n}}\,\mathrm{d}u\right] \nonumber \\
		&= \log\sigma^2_{(t-1)\delta_n} +\log\left[1+\sum_{p=1}^J\frac{1}{p!} Z_t^{n,p} +\mathit{o}\left(\delta_n^{(J+1)H-\psi}\right)\right] \nonumber \\
		&= \log\sigma^2_{(t-1)\delta_n} +\frac{1}{\delta_n}\int_{(t-1)\delta_n}^{t\delta_n}\left(\log\sigma^2_u -\log\sigma^2_{(t-1)\delta_n}\right)\, \mathrm{d}u + \sum_{p=2}^J\frac{1}{p!}Z_t^{n,p} \nonumber \\
		&\quad+\sum_{j=2}^J\frac{(-1)^{j-1}}{j}\left\{\sum_{p=1}^J \frac{1}{p!}Z_t^{n,p} +\mathit{o}\left(\delta_n^{(J+1)H-\psi}\right)\right\}^j +\mathit{o}\left(\delta_n^{(J+1)H-\psi}\right) \nonumber \\
		&= \frac{1}{\delta_n}\int_{(t-1)\delta_n}^{t\delta_n}\log\sigma^2_u\, \mathrm{d}u + \sum_{j=2}^J \sum_{k=1}^j \frac{(-1)^{k-1}}{k} \sum_{\mathbf{p}\in\mathbb{N}^k,|\mathbf{p}|=j} \frac{1}{\mathbf{p}!}Z_t^{n,\mathbf{p}} +\mathit{o}\left(\delta_n^{(J+1)H-\psi}\right). \label{Taylor.Calculation.End}
	\end{align}
	Note that the H\"older continuity property also implies that all reminder terms in the above equality are independent of $t\in\Lambda_n$ and $\omega\in\Omega$ if $\delta_n$ is sufficiently small, see also (\ref{Taylor.Remainder.Bounded}). 
	Therefore, the conclusion follows from taking a difference of both sides of (\ref{Taylor.Calculation.End}).
\end{proof}
The second term is also negligible because the following inequality holds.
\begin{lem}\label{Taylor.Expansion.AddNoise}
	The following inequality holds:
	\begin{equation*}
		\max_{t\in\Lambda_n}\left|V_t^n -G_t^{n,\dagger}\right|\leq \left(\sup_{u\in[0,T]}|\kappa_u|\right)\cdot\delta_n.
	\end{equation*}
\end{lem}
Finally, we show the negligibility of the third term. In order to achieve this purpose, it suffices to prove that the error between $\mathbf{Z}_n^\mathbf{p}$ and $\mathbf{W}_n^\mathbf{p}$ is negligible for each $\mathbf{p}\in\mathbb{N}^K$ by using the triangle inequality of $\|\cdot\|_\infty$. Therefore, we show the following result.
\begin{lem}
	For each $\mathbf{p}\equiv(p_1,p_2,\cdots,p_K)\in\mathbb{N}^K$ with $|\mathbf{p}|\geq 2$, the following relation holds for any $\psi\in(0,H)$,
	\begin{equation*}
		\max_{t\in\Lambda_{n+1}}\left|Z_t^{n,\mathbf{p}} -W_t^{n,\mathbf{p}}\right|\leq \left(2^{|\mathbf{p}|}-1\right)\left(A_{H-\psi,\eta}\vee 1\right)^{(|\mathbf{p}|-1)}\left(\sup_{u\in[0,T]}|\kappa_u|\vee 1\right)^{|\mathbf{p}|}\cdot\delta_n^{1+(|\mathbf{p}|-1)(H-\psi)},
	\end{equation*}
	where $A_{H-\psi,\eta}\equiv A_{H-\psi,\eta}(T,\vartheta)$ given by
	\begin{equation*}
		A_{H-\psi,\eta}:= \sup_{s,u\in[0,T]}\frac{\eta|W_u^H-W_s^H|}{|u-s|^{H-\psi}}<\infty.
	\end{equation*}
\end{lem}
\begin{proof}
	At first, consider the case where $K=1$, i.e. $p\in\mathbb{N}$ with $p\geq 2$. 
	Note that the binomial theorem yields that the integrand of $Z_t^{n,\mathbf{p}}$ is given by
	\begin{equation*}
		\left(\log\sigma^2_u - \log\sigma^2_{(t-1)\delta_n}\right)^p = \sum_{j=0}^p \frac{\Gamma(p+1)}{\Gamma(p-j+1)\Gamma(j+1)} \eta^{p-j}\left(W^H_u - W^H_{(t-1)\delta_n}\right)^{p-j}\left(\int_{(t-1)\delta_n}^u\kappa_s\, \mathrm{d}s\right)^j.
	\end{equation*}
	Then $Z_t^{n,p}$ is represented by
	\begin{equation*}
		Z_t^{n,p} = \frac{1}{\delta_n}\int_{(t-1)\delta_n}^{t\delta_n}\eta^p\left(W_u^H - W_{(t-1)\delta_n}^H\right)^p\, \mathrm{d}u + \sum_{j=1}^p \frac{\Gamma(p+1)}{\Gamma(p-j+1)\Gamma(j+1)}R_{t,n,j},
	\end{equation*}
	where
	\begin{equation*}
		R_{t,n,j}^p:= \frac{1}{\delta_n}\int_{(t-1)\delta_n}^{t\delta_n}\eta^{p-j}\left(W_u^H - W_{(t-1)\delta_n}^H\right)^{p-j}\left(\int_{(t-1)\delta_n}^u \kappa_s\, \mathrm{d}s\right)^j\, \mathrm{d}u.
	\end{equation*}
	Therefore, the conclusion when $K=1$ follows from the H\"older continuity of the fractional Brownian motion $W^H$. Next, we consider the case where $K\geq 2$. 
	Then the multinomial theorem yield that
	\begin{align*}
		\max_{t\in\Lambda_{n+1}}\left|Z_t^{n,\mathbf{p}} - W_t^{n,\mathbf{p}}\right| =&\max_{t\in\Lambda_{n+1}}\left|\prod_{k=1}^K \left\{W_t^{n,p_k}+\left(Z_t^{n,p_k}-W_t^{n,p_k}\right)\right\} -W_t^{n,\mathbf{p}}\right| \\
		\leq&\sum_{j_1,j_2,\cdots,j_K}\max_{t\in\Lambda_{n+1}}\left|\prod_{k=1}^K\left(W_t^{n,p_k}\right)^{j_i}\left(Z_t^{n,p_k}-W_t^{n,p_k}\right)^{1-j_i}\right|,
	\end{align*}
	where the last sum is taken over all $j_1,\cdots,j_K\in\{0,1\}$ satisfying that there exists $i\in\{1,\cdots,K\}$ such that $j_i=0$. As a result, the conclusion when $K\geq 2$ follows from (\ref{Asymptotic.Order.Z.W}) and the conclusion when $K=1$.
\end{proof}
\section{Asymptotic Decay of Covariance Function for Stationary Process Associated with Some Functionals of Fractional Brownian Motion}
In this section, we will show an asymptotic decay of covariance function for the stationary process $\mathbf{W}_n^\mathbf{p}$ appeared in the reminder terms of the Taylor approximation given in Proposition \ref{Expansion.Formula.Y}. This result plays a key role in order to prove that the reminder terms $\mathbf{W}_n^\mathbf{p}$ are asymptotically negligible in the case where the consistency of the adapted Whittle estimator holds. We will state the key result in Section~\ref{Statement.Key.Lemma}, several preliminary results used in its proof are summarized in Section~\ref{Section.Preliminaries.Key.Lemma} and its proof is given in Section~\ref{Section.Proof.Key.Lemma}.
\subsection{Notation and Statement of Key Result}\label{Statement.Key.Lemma}
At first, we prepare notation in order to state a general result for Proposition C.1. Denote by $C_{\mathbb{R}}$ a set of real-valued continuous functions on $\mathbb{R}$ and by $\mathcal{B}(C_{\mathbb{R}})$ a Borel $\sigma$-algebra on $C_{\mathbb{R}}$ generated by a topology associated with the compact convergence. 
Let $\mu_H$ be the distribution of the two-sided standard fractional Brownian motion with the Hurst parameter $H\in(0,1]$ on $(C_{\mathbb{R}},\mathcal{B}(C_{\mathbb{R}}))$, and a continuous shift operator $\theta=\{\theta_u\}_{u\in\mathbb{R}}$ be defined by $\theta_ux_\cdot:=x_{\cdot+u}-x_u$ for $(u,x)\in\mathbb{R}\times C_{\mathbb{R}}$. Note that $\mu_H$ is $\theta$-invariant, i.e. $\mu_H\circ\theta_u^{-1}=\mu_H$ for each $u\in\mathbb{R}$ since the fractional Brownian motion enjoys the stationary increments property. Moreover, $U=\{U_u\}_{u\in\mathbb{R}}$ denotes the canonical process on $(C_{\mathbb{R}},\mathcal{B}(C_{\mathbb{R}}))$, i.e. $U_u(x):=x_u$ for each $(u,x)\in\mathbb{R}\times C_{\mathbb{R}}$. Furthermore, for each $\mathbf{p}=(p_1,\cdots,p_K)\in\mathbb{N}^K$, $K\in\mathbb{R}$, and compact set $A_\mathbf{p}\subset\mathbb{R}^K$, we define a functional $F^\mathbf{p}$ by
\begin{equation*}
	F^\mathbf{p}(x) :=\int_{A_\mathbf{p}}\prod_{k=1}^Kx_{u_k}^{p_k}\, \mathrm{d}u_1\cdots du_K =\int_{A_\mathbf{p}}\prod_{k=1}^K\left\{U_{u_k}(x)\right\}^{p_k}\, \mathrm{d}u_1\cdots du_K,
\end{equation*}
for $x=\{x_u\}_{u\in\mathbb{R}}\in C_{\mathbb{R}}$, and set a stochastic process $G^{\mathbf{p}}_u:=F^\mathbf{p}(\theta_u)$ for $u\in\mathbb{R}$.

Next, let us recall the following definition, e.g. see Tudor~\cite{Tudor}, p.172. 
\begin{defi}
	A filter of length $J\in\mathbb{N}$ and order $r\in\mathbb{N}$ is a $(J+1)$-dimensional vector $\mathbf{a}:=\{a_0,a_1,\cdots,a_J\}$ such that for any $k\in\mathbb{N}\cup\{0\}$ with $k<r$,
	\begin{equation}\label{Property.Increment1}
		\sum_{j=0}^J a_j j^k = 0, 
	\end{equation}
	where we use $0^0:=1$ for convenience, and 
	\begin{equation}\label{Property.Increment2}
		\sum_{j=0}^J a_j j^r \neq 0.
	\end{equation}
	Moreover, we also call $\mathbf{a}=\{a_0,a_1,\cdots,a_J\}$ as a filter of length $J$ and order $0$ if it satisfies (\ref{Property.Increment2}) for $r=0$.
\end{defi}
\begin{rem}\rm
	For any filter $\mathbf{a}=\{a_0,a_1,\cdots,a_J\}$ of order $r\in\mathbb{N}$, the property (\ref{Property.Increment1}) yield that for any $k\in\mathbb{N}\cup\{0\}$ with $k<2r$,
	\begin{equation}\label{Vanish.Double.Sum.Filter}
		\sum_{i,j=0}^J a_ia_j (j-i)^k =0.
	\end{equation}
\end{rem}
For a filter $\mathbf{a}=\{a_0,a_1,\cdots,a_J\}$ and a stochastic process $X=\{X_u\}_{u\in\mathbb{R}}$, we define
\begin{equation}\label{Definition.Filter.Discrete.Process}
	\Delta_\mathbf{a}{X}_u:=\sum_{j=0}^Ja_jX_{u-j},\ \ u\in\mathbb{R}.
\end{equation}
For example, if we set $\mathbf{a}=(a_0,a_1)$ with $a_0=-1$, $a_1=1$, then $\mathbf{a}$ is a filter of length $1$ and order $1$, and $\Delta_\mathbf{a}{X}_\cdot=X_\cdot-X_{\cdot-1}$.

Finally, we will state a main result in this section.
\begin{prop}\label{Asymptotic.Decay.G}
	Let $\mathbf{a}$ be a filter of length $J\in\mathbb{N}$ and order $r\in\mathbb{N}\cup\{0\}$.
	Then for any $\mathbf{p}\in\mathbb{N}^K$ with $K\in\mathbb{N}$, the stochastic process $\{\Delta_\mathbf{a} G_u^{\mathbf{p}}\}_{u\in\mathbb{R}}$ is stationary and for any $\mathbf{p}\in\mathbb{N}^K$, $\mathbf{q}\in\mathbb{N}^L$ with $K,L\in\mathbb{N}$ and $u\in\mathbb{R}$,
	\begin{equation}
		\mathrm{Cov}_{\mu_H}\left[\Delta_\mathbf{a} G_u^{\mathbf{p}},\Delta_\mathbf{a} G_{u+\tau}^{\mathbf{q}}\right]= O\left(|{\tau}|^{2H-2-2r}\right)\hspace{0.2cm}\mbox{as $|\tau|\rightarrow \infty$}. \label{Main.Theorem.Decay}
	\end{equation}
\end{prop}
As a corollary of Proposition \ref{Asymptotic.Decay.G}, we can obtain the following result from the self-similarity property of the fractional Brownian motion.
\begin{prop}\label{Key.Lemma.Result.Cor}
	For any $\mathbf{p}\in\mathbb{N}^K$, $\mathbf{q}\in\mathbb{N}^L$ with $K,L\in\mathbb{N}$, the stochastic process $\{W_t^{n,\mathbf{p}}\}_{t\in\mathbb{Z}}$ is stationary for each $n\in\mathbb{N}$ and the following relation holds for any $t\in\mathbb{R}$:
	\begin{equation*}
		\sup_{n\in\mathbb{N}}\left|\delta_n^{-(|\mathbf{p}|+|\mathbf{q}|)H}\mathrm{Cov}\left[\Delta W_t^{n,\mathbf{p}},\Delta W_{t+\tau}^{n,\mathbf{q}}\right]\right|=O\left(|{\tau}|^{2H-4}\right)\hspace{0.2cm}\mbox{as $|\tau|\to\infty$}.
	\end{equation*}
\end{prop}
\subsection{Preliminary Results}\label{Section.Preliminaries.Key.Lemma}
We summarize several preliminary results used in the proof of Proposition~\ref{Asymptotic.Decay.G} in this subsection. The first result is proven in the similar way to that in Billingsley~\cite{Billingsley}, p.230-231.
\begin{prop}\label{Measurability.Functional}
	Let $A\in\mathcal{B}(\mathbb{R}^K)$ with $K\in\mathbb{N}$ and $f$ be a measurable function on $\mathbb{R}^K$. For $x=\{x_u\}_{u\in\mathbb{R}}\in C_{\mathbb{R}}$, we define a functional $F$ by
	\begin{equation}\label{Functional.Form1}
		F(x):=\int_A f(x_{u_1},\cdots,x_{u_K})\, \mathrm{d}u_1\cdots du_K.
	\end{equation}
	Then the functional $F$ is $\mathcal{B}(C_{\mathbb{R}})$-measurable if $(u_1,\cdots,u_K)\mapsto f(x_{u_1},\cdots,x_{u_K})$ is integrable on $A$. 
	Furthermore, if $A$ is compact and $f$ is continuous, then $x\mapsto F(x)$ is also continuous.
\end{prop}
Since the shift operator $\theta$ is continuous and $\mu_H$ is $\theta$-invariant, we can obtain the following result using Proposition~\ref{Measurability.Functional}. 
\begin{cor}\label{Stationarity.Functional}
	Let us consider a functional $F$ of the form (\ref{Functional.Form1}) with a continuous function $f$ and a compact set $A\in\mathcal{B}(\mathbb{R}^K)$. Then a stochastic process $G=\{G_u\}_{u\in\mathbb{R}}$ defined by $G_u(x):=F(\theta_ux)$ for $(u,x)\in\mathbb{R}\times C_{\mathbb{R}}$ is continuous and strong stationary on the probability space $(C_{\mathbb{R}},\mathcal{B}(C_{\mathbb{R}}),\mu_H)$.
\end{cor}
The following result is a consequence of the well-know Wick formula which expresses the higher moments of centered multivariate Gaussian vectors in terms of its second moments, e.g. see Nourdin and Peccati~\cite{Book-Nourdin-Peccati}. 
Given a finite set $b$ the number of which is even, we denote by $\mathcal{P}(b)$ the class of all partitions of $b$ such that each block of a partition $\pi$ contains exactly two elements, and recall $\Lambda_M:=\{1,2,\cdots,M\}$.
\begin{lem}\label{Result2.Diagram.Formula}
	For any $K_0, L_0\in\mathbb{N}$ and $(K_0+L_0)$-dimensional centered Gaussian vector $(X_1,\cdots,X_{K_0+L_0})$, 
	\begin{align*}
		\mathrm{Cov}\left[\prod_{k=1}^{K_0}X_k,\prod_{\ell=1}^{L_0}X_{K_0+\ell}\right]
		=
		\begin{cases}
			\sum_{\pi=\{\{k_1,\ell_1\},\cdots,\{k_{M_0},\ell_{M_0}\}\}\in\mathcal{P}_0(\Lambda_{2M_0})}\mathrm{Cov}[X_{k_1},X_{\ell_1}]\cdots\mathrm{Cov}[X_{k_{M_0}},X_{\ell_{M_0}}]& \mbox{if $K_0+L_0$ is even}, \\
			0& \mbox{if $K_0+L_0$ is odd},
		\end{cases}
	\end{align*}
	where $M_0:=(K_0+L_0)/2$ and $\mathcal{P}_0(\Lambda_{2M_0})$ denotes the subset of $\mathcal{P}(\Lambda_{2M_0})$ whose elements are partitions $\pi=\{\{k_1,\ell_1\},\cdots,\{k_{M_0},\ell_{M_0}\}\}\in\mathcal{P}(\Lambda_{2M_0})$ such that there exists $m\in\Lambda_{M_0}$ satisfying $k_m\leq K_0<\ell_m$.
\end{lem}
\begin{proof}
	Let us consider only the case that both $K_0$ and $L_0$ are even since the other cases are trivial from the Wick formula. Since $K_0$ and $L_0$ are even, the Wick formula yields that 
	\begin{align*}
		E\left[\prod_{k=1}^{K_0+L_0}X_k\right]
		=&\sum_{\{\{k_1,\ell_1\},\cdots,\{k_{M_0},\ell_{M_0}\}\}\in\mathcal{P}(\Lambda_{2M_0})}\mathrm{Cov}[X_{k_1},X_{\ell_1}]\cdots\mathrm{Cov}[X_{k_{M_0}},X_{\ell_{M_0}}] \\
		=&\left(\sum_{\{\{k_1,\ell_1\},\cdots,\{k_{K_0/2},\ell_{K_0/2}\}\}\in\mathcal{P}(\Lambda_{K_0})}\mathrm{Cov}[X_{k_1},X_{\ell_1}]\cdots\mathrm{Cov}[X_{k_{K_0/2}},X_{\ell_{K_0/2}}]\right)\\
		&\times\left(\sum_{\{\{k_1,\ell_1\},\cdots,\{k_{L_0/2},\ell_{L_0/2}\}\}\in\mathcal{P}(\Lambda_{L_0})}\mathrm{Cov}[X_{K_0+k_1},X_{K_0+\ell_1}]\cdots\mathrm{Cov}[X_{K_0+k_{L_0/2}},X_{K_0+\ell_{L_0/2}}]\right) \\
		&+\sum_{\{\{k_1,\ell_1\},\cdots,\{k_{M_0},\ell_{M_0}\}\}\in\mathcal{P}_0(\Lambda_{2M_0})}\mathrm{Cov}[X_{k_1},X_{\ell_1}]\cdots\mathrm{Cov}[X_{k_{M_0}},X_{\ell_{M_0}}] \\
		=&E\left[\prod_{k=1}^{K_0}X_k\right]E\left[\prod_{\ell=1}^{L_0}X_{K_0+\ell}\right] +\sum_{\{\{k_1,\ell_1\},\cdots,\{k_{M_0},\ell_{M_0}\}\}\in\mathcal{P}_0(\Lambda_{2M_0})}\mathrm{Cov}[X_{k_1},X_{\ell_1}]\cdots\mathrm{Cov}[X_{k_{M_0}},X_{\ell_{M_0}}].
	\end{align*}
	Therefore, the conclusion follows.
\end{proof}
\subsection{Proof of Proposition~\ref{Asymptotic.Decay.G}}\label{Section.Proof.Key.Lemma}
Before proving Proposition~\ref{Asymptotic.Decay.G}, we will show the following two lemmas.
Denote by $\gamma_{s,u}(\tau):=\mathrm{Cov}_{\mu_H}[U_{s}(\theta_0),U_{u}(\theta_{\tau})]$ for $s,u,\tau\in\mathbb{R}$.
\begin{lem}\label{Property.gamma.s.u}
	For each $s,u\in\mathbb{R}$, $\tau\mapsto\gamma_{s,u}(\tau)$ is infinitely differentiable a.e. and, for any $k\in\mathbb{N}\cup\{0\}$ and compact set $A\subset\mathbb{R}$, its $k$th derivative satisfies
	\begin{equation*}
		\sup_{s,u\in A}\left|\frac{\partial^k\gamma_{s,u}}{{\partial\tau^k}}(\tau)\right|=O(|\tau|^{2H-2-k})\ \ \mbox{as $|\tau|\to\infty$}.
	\end{equation*}
\end{lem}
\begin{proof}
	Fix $s,u\in\mathbb{R}$ and a compact set $A\subset\mathbb{R}$. Since $\mu_H$ is a distribution of the two-sided standard fractional Brownian motion with the Hurst parameter $H$, we have 
	\begin{equation*}
		\gamma_{s,u}(\tau)=-\frac{1}{2}\left(|\tau+u-s|^{2H}-|\tau+u|^{2H}-|\tau-s|^{2H}+|\tau|^{2H}\right), \ \ \tau\in\mathbb{R}.
	\end{equation*}
	As a result, the first assertion is obvious and for any $k\in\mathbb{N}$, we obtain
	\begin{equation}\label{Representation.Derivative.Covariance}
		\frac{\partial^k\gamma_{s,u}}{{\partial\tau^k}}(\tau)=-\frac{(\mathrm{sgn}(\tau))^k}{2}\prod_{\ell=0}^{k-1}(2H-\ell)\left(|\tau+u-s|^{2H-k}-|\tau+u|^{2H-k}-|\tau-s|^{2H-k}+|\tau|^{2H-k}\right)
	\end{equation}
	if $|\tau|$ is sufficiently large, where $\mathrm{sgn}(\cdot)$ denotes the sign function defined by
	\begin{equation*}
		\mathrm{sgn}(\tau)=
		\begin{cases}
			1& \text{$\tau\geq 0$}, \\
			-1& \text{$\tau< 0$}.
		\end{cases}
	\end{equation*}
	Then the second assertion follows from (\ref{Representation.Derivative.Covariance}) because Taylor's theorem yields that for any $L\in\mathbb{N}$,
	\begin{align*}
		&|\tau+u-s|^{2H-k}-|\tau+u|^{2H-k}-|\tau-s|^{2H-k}+|\tau|^{2H-k} \\
		=&|\tau|^{2H-k}\left\{\left(1+\frac{u-s}{\tau}\right)^{2H-k}-\left(1+\frac{u}{\tau}\right)^{2H-k}-\left(1+\frac{-s}{\tau}\right)^{2H-k}+1\right\} \\
		=&|\tau|^{2H-k}\sum_{\ell_0=1}^{L}\frac{1}{\ell_0!}\left\{\prod_{\ell=0}^{\ell_0-1}(2H-k-\ell)\right\}\left\{(u-s)^{\ell_0}-(-s)^{\ell_0}+u^{\ell_0}\right\}\tau^{-\ell_0} +o(|\tau|^{2H-k-L})
	\end{align*}
	as $|\tau|\to\infty$ uniformly in $s,u\in A$ and $(u-s)^{\ell_0}-(-s)^{\ell_0}+u^{\ell_0}=0$ for $\ell_0=1$.
\end{proof}
\begin{lem}\label{Key.Lemma}
	Let $\mathbf{a}=(a_0,a_1,\cdots,a_J)$ be a filter of length $J\in\mathbb{N}$ and order $r\in\mathbb{N}\cup\{0\}$. For any compact set $A\subset\mathbb{R}$ and $\mathbf{p}=(p_1,\cdots,p_K)\in\mathbb{N}^K$, $\mathbf{q}=(q_1,\cdots,q_L)\in\mathbb{N}^L$ with $K,L\in\mathbb{N}$,
	\begin{equation*}
		\hspace{-1cm}\sup_{s_1,u_1,\cdots,s_v,u_v\in A}\left|\sum_{i,j=0}^J a_ia_j\mathrm{Cov}_{\mu_H}\left[\prod_{k=1}^{K}\left\{U_{s_k}(\theta_i)\right\}^{p_k}, \prod_{\ell=1}^L\left\{U_{u_\ell}(\theta_{j+\tau})\right\}^{q_\ell}\right]\right|=O\left(|\tau|^{2H-2-2r}\right)\ \ \mbox{as $|\tau|\to\infty$}.
	\end{equation*}
\end{lem}
\begin{proof}
	By using Lemma~\ref{Result2.Diagram.Formula} in the case that $K_0:=|\mathbf{p}|$, $L_0:=|\mathbf{q}|$ and $(K_0+L_0)$-dimensional centered Gaussian vector $\mathbf{X}\equiv(X_1,\cdots,X_{K_0+L_0})$ given by
	\begin{align*}
		\mathbf{X}:=(\underbrace{U_{s_1}(\theta_i),\cdots,U_{s_1}(\theta_i)}_{\mbox{$p_1$ times}},\cdots,\underbrace{U_{s_K}(\theta_i),\cdots,U_{s_K}(\theta_i)}_{\mbox{$p_K$ times}},
		\underbrace{U_{u_1}(\theta_{j+\tau}),\cdots,U_{u_1}(\theta_{j+\tau})}_{\mbox{$q_1$ times}},\cdots,\underbrace{U_{u_L}(\theta_{j+\tau}),\cdots,U_{u_L}(\theta_{j+\tau})}_{\mbox{$q_L$ times}}),
	\end{align*} 
	it suffices to prove that for any compact set $A\subset\mathbb{R}$ and $v\in\mathbb{N}$,
	\begin{equation}\label{Key.Lemma.Final}
		\sup_{s_1,u_1,\cdots,s_v,u_v\in A}\left|\sum_{i,j=0}^J a_ia_j \prod_{w=1}^v\mathrm{Cov}_{\mu_H}\left[U_{s_w}(\theta_i),U_{u_w}(\theta_{j+\tau})\right]\right|=O\left(|\tau|^{2H-2-2r}\right)\ \ \mbox{as $|\tau|\to\infty$}
	\end{equation}
	since the stationary increments property of the fractional Brownian motion implies
	\begin{align*}
		&\mathrm{Cov}_{\mu_H}\left[U_{s_1},U_{s_2}\right]=\mathrm{Cov}_{\mu_H}\left[U_{s_1}(\theta_i),U_{s_2}(\theta_i)\right],\\
		&\mathrm{Cov}_{\mu_H}\left[U_{u_1},U_{u_2}\right]=\mathrm{Cov}_{\mu_H}\left[U_{u_1}(\theta_{j+\tau}),U_{u_2}(\theta_{j+\tau})\right]
	\end{align*}
	for any $s_1,s_2,u_1,u_2\in\mathbb{R}$.
	
	Fix a compact set $A\subset\mathbb{R}$ and recall $\gamma_{s,u}(\tau):=\mathrm{Cov}_{\mu_H}[U_{s}(\theta_0),U_{u}(\theta_{\tau})]$. Since Taylor's theorem and Lemma~\ref{Property.gamma.s.u} yield that for any $K\in\mathbb{N}$,
	\begin{equation}\label{gamma.Taylor}
		\sup_{\substack{s,u\in A\\ i,j=0,\cdots,J}}\left|\gamma_{s,u}(\tau+(j-i)) -\sum_{k=0}^{K}\frac{(j-i)^k}{k!}\frac{\partial^k\gamma_{s,u}}{\partial\tau^k}(\tau)\right|=o\left(|\tau|^{2H-2-K}\right)\ \ \mbox{as}\ \ |\tau|\to\infty,
	\end{equation}
	(\ref{Key.Lemma.Final}) in the case of $v=1$ follows from (\ref{Vanish.Double.Sum.Filter}) if we take $K\in\mathbb{N}$ satisfying $K\geq 2r$. Moreover, the Taylor approximation (\ref{gamma.Taylor}), the multinomial theorem and Lemma~\ref{Property.gamma.s.u} yield that
	\begin{equation}\label{gamma.Exp.Shift}
		\sup_{\substack{s_1,u_1,\cdots,s_v,u_v\in A\\ i,j=0,\cdots,J}}\left|\prod_{w=1}^v\gamma_{s_w,u_w}(\tau+(j-i))-\sum_{k_1,\cdots,k_v=0}^{K}\frac{(j-i)^{k_1+\cdots+k_v}}{k_1!\cdots k_v!}\prod_{w=1}^v\frac{\partial^{k_w}\gamma_{s_w,u_w}}{\partial\tau^{k_w}}(\tau)\right|=o\left(|\tau|^{2H-2-K}\right)
	\end{equation}
	as $|\tau|\to\infty$, and (\ref{Vanish.Double.Sum.Filter}) and Lemma~\ref{Property.gamma.s.u} yield that
	\begin{align}\label{gamma.p.filter}
		&\sup_{s_1,u_1,\cdots,s_v,u_v\in A}\left|\sum_{i,j=0}^J a_ia_j\frac{(j-i)^{k_1+\cdots+k_v}}{k_1!\cdots k_v!}\prod_{w=1}^v\frac{\partial^{k_w}\gamma_{s_w,u_w}}{\partial\tau^{k_w}}(\tau)\right|\\
		&\nonumber
		\begin{cases}
			=0&\mbox{if}\ \ \sum_{w=1}^vk_w< 2r,\\
			=O\left(|\tau|^{\sum_{w=1}^v(2H-2-k_w)}\right)\ \ \mbox{as}\ \ |\tau|\to\infty& \mbox{if}\ \ \sum_{w=1}^vk_w\geq 2r.
		\end{cases}
	\end{align}
	Then (\ref{Key.Lemma.Final}) in the case of $v\geq 2$ follows from (\ref{gamma.Exp.Shift}) and (\ref{gamma.p.filter}) if we take $K\in\mathbb{N}$ satisfying $K\geq 2r$. Therefore, we finish the proof.
\end{proof}
\begin{proof}[Proof of Proposition~\ref{Asymptotic.Decay.G}]
	Since $G^{\mathbf{p}}$ is stationary from Corollary~\ref{Stationarity.Functional}, the bilinearity of covariance functions and Fubini's theorem yield that
	\begin{align*}
		\mathrm{Cov}_{\mu_H}\left[\Delta_\mathbf{a} G_u^{\mathbf{p}},\Delta_\mathbf{a} G_{u+\tau}^{\mathbf{q}}\right] =\int_{A_\mathbf{p}\times A_\mathbf{q}} \sum_{i,j=0}^Ja_ia_j\mathrm{Cov}_{\mu_H}\left[\prod_{k=1}^{K}\left\{U_{s_k}(\theta_i)\right\}^{p_k}, \prod_{\ell=1}^L\left\{U_{u_\ell}(\theta_{j+\tau})\right\}^{q_\ell}\right]\, \mathrm{d}s_1\cdots ds_Kdu_1\cdots du_L.
	\end{align*}
	Therefore, the conclusion follows from the above equality and Lemma~\ref{Key.Lemma}.
\end{proof}
\section{Approximating Spectral Density of Data}\label{Appendix.SPD}
\subsection{Spectral Density of Stationary Gaussian Sequence $\{G_t^{n,\dagger}\}_{t\in\mathbb{Z}}$}
Recall that a spectral density of the stationary Gaussian sequence $\{G_t^{n,\dagger}\}_{t\in\mathbb{Z}}$, which is obtained by~\cite{FT19-BEJ}, is characterized by
\begin{align*}
\mathrm{Cov}\left[G_t^{n,\dagger}, G_{t+\tau}^{n,\dagger}\right] =& \frac{\eta^2\delta_n^{2H}}{2} (|\tau+2|^{2H+2}-4|\tau+1|^{2H+2} \\
& \hspace{1cm} +6|\tau|^{2H+2}-4{|\tau-1|}^{2H+2}+{|\tau-2|}^{2H+2}) \\ 
=&\int_{-\pi}^\pi e^{\sqrt{-1}\tau\lambda}\eta^2\delta_n^{2H}f_H(\lambda)\, \mathrm{d}\lambda, \nonumber
\end{align*}
where 
\begin{equation*}
f_H(\lambda):=C_H\{2(1-\cos{\lambda})\}^2\sum_{j\in\mathbb{Z}}\frac{1}{|\lambda+2\pi j|^{3+2H}}
\end{equation*}
with $C_H:=(2\pi)^{-1}\Gamma(2H+1)\sin(\pi H)$. 
The following Lemma shows that the stationary Gaussian sequence $\{G_t^{n,\dagger}\}_{t\in\mathbb{Z}}$ satisfies Assumption 1 in~\cite{FT19-BEJ}, see Section 4.2 in~\cite{FT19-BEJ}. 
\begin{lem} \label{Lemma.SPD}
	The spectral density $f(\lambda, H)$ satisfies the following relations.
	\begin{enumerate}[$(1)$]
		\item For any $H\in\Theta$, $\lambda\mapsto f(\lambda, H)$, $\lambda\in[-\pi, \pi]/\{0\}$, 
		is a non-negative integrable even function with $2\pi$-periodicity. Moreover, it satisfies that
		\begin{equation*}
		f\in\mathcal{C}^{3, 1}\left(\Theta\times[-\pi, \pi]/\{0\}\right).
		\end{equation*}
		\item\label{SPD.Identifiability.Cond} If $(H_1, \eta_1)$ and $(H_2, \eta_2)$ are distinct elements of $\Theta\times\Sigma$, 
		a set $\{\lambda\in [-\pi,\pi]:\eta_1 f(\lambda, H_1)\neq \eta_2 f(\lambda, H_2)\}$ has a positive Lebesgue measure.
		\item\label{SPD.Asymptotic} Let $\alpha(H):=2H-1$ with $H\in(0, 1)$. There exist constants $c_1,c_2>0$ and for any $\iota>0$, there exists a constant $c_{3, \iota}$, which only depends on $\iota$, 
		such that the following conditions hold for every $(H,\lambda)\in\Theta\times[-\pi, \pi]\backslash\{0\}$. 
		\begin{enumerate}[$(a)$]
			\item $c_1|\lambda|^{-\alpha(H)}\leq f(\lambda, H)\leq c_2|\lambda|^{-\alpha(H)}$.
			\item For any $j\in\{1, 2, 3\}$,
			\begin{equation*}
			\left|\frac{\partial^j}{\partial{H}^j}f(\lambda, H)\right|\leq c_{3, \iota}|\lambda|^{-\alpha(H)-\iota},\hspace{0.2cm}
			\left|\frac{\partial^{j+1}}{\partial\lambda \partial{H^j}}f(\lambda, H)\right|\leq c_{3, \iota}|\lambda|^{-\alpha(H)-1-\iota}.
			\end{equation*}
		\end{enumerate}
	\end{enumerate}
\end{lem}
\subsection{Spectral Density of Stationary Gaussian Sequence $\{G_t^n\}_{t\in\mathbb{Z}}$}
We derive a spectral density of the stationary sequence $\{G_t^n\}_{t\in\mathbb{Z}}$ in this subsection. Since $\{\epsilon_t^n\}_{t\in\mathbb{Z}}$ is an i.i.d. sequence, $\{\Delta\epsilon_t^n\}_{t\in\mathbb{Z}}$ is a MA(1) process and its auto-covariance function is given by
\begin{equation*}
\gamma_n(\tau):=\mathrm{Cov}\left[\Delta\epsilon_t^n, \Delta\epsilon_{t+\tau}^n\right]=
\begin{cases}
4/m_n & (\tau=0) \\
-2/m_n & (|\tau|=1) \\
0& (\mbox{otherwise})
\end{cases}.
\end{equation*}
Then its spectral density $\ell_n$ is given by the Fourier series
\begin{equation*}
\ell_n(\lambda):=\frac{1}{2\pi}\sum_{\tau\in\mathbb{Z}}\gamma_n(\tau)e^{\sqrt{-1}\tau\lambda}=\frac{2}{m_n}\ell(\lambda),\ \ \mbox{where}\ \ \ell(\lambda):=\frac{1}{\pi}\left(1-\cos\lambda\right), \ \ \lambda\in[-\pi,\pi].
\end{equation*}
Since $\{G_t^{n,\dagger}\}_{t\in\mathbb{Z}}$ and $\{\epsilon_t^n\}_{t\in\mathbb{Z}}$ are independent, the covariance function of $\{G_t^n\}_{t\in\mathbb{Z}}$ is characterized by
\begin{equation*}
\mathrm{Cov}\left[G_t^n, G_{t+\tau}^n\right]=
\mathrm{Cov}\left[G_t^{n,\dagger}, G_{t+\tau}^{n,\dagger}\right] + \mathrm{Cov}\left[\Delta\epsilon_t^n, \Delta\epsilon_{t+\tau}\right] =\int_{-\pi}^\pi e^{\sqrt{-1}\tau\lambda}f_\vartheta^n(\lambda)\, \mathrm{d}\lambda, 
\end{equation*}
where the spectral density $f_\vartheta^n$ is given by
\begin{equation*}
f_\vartheta^n(\lambda)\equiv f(\lambda,H,\eta,n):=\eta^2\delta_n^{2H}f_H(\lambda)+\frac{2}{m_n}\ell(\lambda),\hspace{0.2cm}\lambda\in[-\pi,\pi],\hspace{0.2cm}\vartheta=(H,\eta).
\end{equation*}
\section{Extension of Some Results in Fox and Taqqu~\cite{FT86, FT87}}
We will show several extended lemmas and theorem developed in Fox and Taqqu~\cite{FT86, FT87} in the case where functions appeared in their results depend on the asymptotic parameter $n\in\mathbb{N}$. They can be easily proven in the similar way to the corresponding results in Fox and Taqqu~\cite{FT86, FT87}; we will however give their concise proofs in Section~\ref{SubSec_FT86_Fourier_Coef} and Section~\ref{SubSec_FT87_Trace} for convenience. The following two results are extensions of Lemma 4 and Lemma 5 in~\cite{FT86} which show an asymptotic decay of the Fourier coefficient. 
\begin{lem}[cf.Lemma 4 and Lemma 5 in~\cite{FT86}]\label{Extension1.FT86}
	Let $\beta\in(-1,0)\cup(0,1)$ and $n\in\mathbb{N}$. Suppose a sequence of $2\pi$-periodic functions $k^n:\mathbb{R}\rightarrow[-\infty, \infty]$, $n\in\mathbb{N}$, satisfies the following conditions:
	\begin{enumerate}[$(1)$]
		\item\label{Beta.Less.Than.0} If $\beta\in(0,1)$, $k^n$ is  continuously differentiable on $[-\pi, \pi]\backslash\{0\}$ for each $n\in\mathbb{N}$ and 
		\begin{equation*}
		\sup_{n\in\mathbb{N},\lambda\in[-\pi, \pi]\backslash\{0\}}|\lambda|^{\beta}\left|k^n(\lambda)\right|<\infty,\ \ \sup_{n\in\mathbb{N},\lambda\in[-\pi, \pi]\backslash\{0\}}|\lambda|^{\beta+1}\left|\frac{\partial k^n}{\partial\lambda}(\lambda)\right|<\infty.
		\end{equation*}
		\item\label{Beta.Greater.Than.0} If $\beta\in(-1,0)$, $k^n$ is integrable and twice continuously differentiable on $[-\pi, \pi]\backslash\{0\}$ for each $n\in\mathbb{N}$ and
		\begin{equation*}
		\sup_{n\in\mathbb{N},\lambda\in[-\pi, \pi]\backslash\{0\}}|\lambda|^{\beta+1}\left|\frac{\partial k^n}{\partial\lambda}(\lambda)\right|<\infty,\ \ \sup_{n\in\mathbb{N},\lambda\in[-\pi, \pi]\backslash\{0\}}|\lambda|^{\beta+2}\left|\frac{\partial^2k^n}{\partial\lambda^2}(\lambda)\right|<\infty.
		\end{equation*}
	\end{enumerate}
	Then the sequence of the Fourier coefficients $\widehat{k^n}(\tau)$, $\tau\in\mathbb{Z}$, satisfies 
	\begin{equation*}
	\sup_{n\in\mathbb{N}}\left|\widehat{k^n}(\tau)\right|=O\left(|\tau|^{\beta-1}\right)\ \ \mbox{as}\ \ |\tau|\to\infty.
	\end{equation*}
\end{lem}
\begin{lem}\label{Extension2.FT86}
	Suppose a sequence of $2\pi$-periodic functions $k^n:\mathbb{R}\rightarrow[-\infty, \infty]$, $n\in\mathbb{N}$, is continuously differentiable on $[-\pi, \pi]\backslash\{0\}$ for each $n\in\mathbb{N}$ and
	\begin{equation*}
	\sup_{n\in\mathbb{N},\lambda\in[-\pi, \pi]\backslash\{0\}}\left|k^n(\lambda)\right|<\infty,\ \ \sup_{n\in\mathbb{N},\lambda\in[-\pi, \pi]\backslash\{0\}}|\lambda|\left|\frac{\partial k^n}{\partial\lambda}(\lambda)\right|<\infty.
	\end{equation*}
	Then the sequence of the Fourier coefficients $\widehat{k^n}(\tau)$, $\tau\in\mathbb{Z}$, satisfies 
	\begin{equation*}
	\sup_{n\in\mathbb{N}}\left|\widehat{k^n}(\tau)\right|=O\left(|\tau|^{-1}\log|\tau|\right)\ \ \mbox{as}\ \ |\tau|\to\infty.
	\end{equation*}
\end{lem}
The following result is an extension of Theorem 1 in~\cite{FT87} in the case where functions appeared in Theorem 1 in~\cite{FT87} depend on the asymptotic parameter $n\in\mathbb{N}$; they however have the same asymptotic behavior at the origin as that assumed in Theorem 1 in~\cite{FT87} uniformly to the asymptotic parameter $n\in\mathbb{N}$ and they uniformly converge to some functions almost everywhere as $n\to\infty$.  
\begin{thm}[cf. Theorem 1 in~\cite{FT87}]\label{Extension.FT87}
	Let $\alpha_1, \alpha_2 <1$ and $p\in\mathbb{N}$. Suppose sequences of even functions $k_1^n, k_2^n:[-\pi,\pi]\to[-\infty,\infty]$ satisfy the following two conditions:
	\begin{enumerate}[$(1)$]
		\item The following relations hold:
		\begin{equation*}
		\sup_{n\in\mathbb{N},\lambda\in[-\pi, \pi]\backslash\{0\}}|\lambda|^{\alpha_1}\left|k_1^n(\lambda)\right|<\infty, \ \ 	\sup_{n\in\mathbb{N},\lambda\in[-\pi, \pi]\backslash\{0\}}|\lambda|^{\alpha_2}\left|k_2^n(\lambda)\right|<\infty.
		\end{equation*}
		\item
		There exist functions $k_1, k_2:[-\pi,\pi]\to[-\infty,\infty]$ such that
		\begin{equation*}
		\lim_{n\to\infty}\esssup_{\lambda\in[-\pi,\pi]}|k_1^n(\lambda)-k_1(\lambda)|=0,\hspace{0.2cm}\lim_{n\to\infty}\esssup_{\lambda\in[-\pi,\pi]}|k_2^n(\lambda)-k_2(\lambda)|=0.
		\end{equation*}
		Moreover, the discontinuities of $k_1$ and $k_2$ have the Lebesgue measure $0$.
	\end{enumerate}
	Under the above conditions, we have
	\begin{enumerate}[$(a)$]
		\item \label{FT87.Thm.Part1}
		If $p(\alpha_1 + \alpha_2)<1$,
		\begin{equation*}
		\lim_{n\to\infty}\frac{1}{n}\mathrm{Tr}\left[\left(\Sigma_n(k_1^n)\Sigma_n(k_2^n)\right)^p\right]=\left(2\pi\right)^{2p-1}\int_{-\pi}^\pi[k_1(\lambda)k_2(\lambda)]^p\, \mathrm{d}\lambda.
		\end{equation*}
		\item \label{FT87.Thm.Part2}
		If $p(\alpha_1 + \alpha_2)\geq 1$, then for any $\psi>0$,
		\begin{equation*}
		\mathrm{Tr}\left[\left(\Sigma_n(k_1^n)\Sigma_n(k_2^n)\right)^p\right] = o\left(n^{p(\alpha_1 + \alpha_2)+\psi}\right)\ \ \mbox{as $n\to\infty$.}
		\end{equation*}
	\end{enumerate}
\end{thm}
\subsection{Proof of Lemma E.1 and Lemma E.2}\label{SubSec_FT86_Fourier_Coef}
\begin{proof}[Proof of Lemma E.1 in Case (1)]
	Consider the case of $\beta\in (0,1)$. Let $\tau\in\mathbb{Z}\backslash\{0\}$. Since $k^n$ is $2\pi$-periodic, we have
	\begin{align*}
		\widehat{k^n}(\tau)=&\int_{-\pi+\pi/|\tau|}^{\pi+\pi/|\tau|} e^{\sqrt{-1}\tau\lambda}k^n(\lambda)\, \mathrm{d}\lambda \\
		=&-\int_{-\pi+\pi/|\tau|}^{\pi+\pi/|\tau|} e^{\sqrt{-1}\tau\left(\lambda-\pi/|\tau|\right)}k^n(\lambda)\, \mathrm{d}\lambda =-\int_{-\pi}^\pi e^{\sqrt{-1}\tau\lambda}k^n\left(\lambda+\frac{\pi}{|\tau|}\right)\, \mathrm{d}\lambda.
	\end{align*}
	As a result, we obtain
	\begin{align}
		2\left|\widehat{k^n}(\tau)\right| &= \left|\int_{-\pi}^{\pi} e^{\sqrt{-1}\tau\lambda}\left[k^n(\lambda) - k^n\left(\lambda+\frac{\pi}{|\tau|}\right)\right]\, \mathrm{d}\lambda\right| \nonumber \\
		&\leq\int_{-\pi}^{\pi}\left|k^n(\lambda) - k^n\left(\lambda+\frac{\pi}{|\tau|}\right)\right|\, \mathrm{d}\lambda =\int_{-\pi}^{-2\pi/|\tau|} +\int_{-2\pi/|\tau|}^{\pi/|\tau|} +\int_{\pi/|\tau|}^{\pi}. \label{Upper.Estimate.1}
	\end{align}
	The assumption implies that 
	\begin{equation*}
		c_1:=\sup_{n\in\mathbb{N},\lambda\in[-\pi, \pi]\backslash\{0\}}\left\{|\lambda|^{\beta}\left|k^n(\lambda)\right|+|\lambda|^{\beta+1}\left|\frac{\partial k^n}{\partial\lambda}(\lambda)\right|\right\}<\infty.
	\end{equation*}
	By the mean value theorem, 
	\begin{align*}
		\int_{-\pi}^{-2\pi/|\tau|}\left|k^n(\lambda)-k^n\left(\lambda+\frac{\pi}{|\tau|}\right)\right|\, \mathrm{d}\lambda &\leq c_1\frac{\pi}{|\tau|} \int_{-\pi}^{-2\pi/|\tau|} \left|\lambda+\frac{\pi}{|\tau|}\right|^{-\beta-1}\, \mathrm{d}\lambda \\
		&= c_1\frac{\pi}{|\tau|} \int_{-\pi+\pi/|\tau|}^{-\pi/|\tau|} |\lambda|^{-\beta-1}\, \mathrm{d}\lambda\\
		&= c_1\pi|\tau|^{\beta-1} \int_{\pi}^{(|\tau|-1)\pi}\lambda^{-\beta-1}\, \mathrm{d}\lambda= O\left(|\tau|^{\beta-1}\right)
	\end{align*}
	as $|\tau|\to\infty$. Note that $\beta>0$ is necessary to obtain the last asymptotic behavior. A similar argument shows that
	\begin{equation*}
		\sup_{n\in\mathbb{N}}\int_{\pi/|\tau|}^{\pi}\left|k^n(\lambda) - k^n\left(\lambda+\frac{\pi}{|\tau|}\right)\right|\, \mathrm{d}\lambda= O\left(|\tau|^{\beta-1}\right)\ \ \mbox{as $|\tau|\to\infty$}.
	\end{equation*}
	We also have
	\begin{align*}
		\int_{-2\pi/|\tau|}^{\pi/|\tau|}\left|k^n(\lambda) - k^n\left(\lambda+\frac{\pi}{|\tau|}\right)\right|\, \mathrm{d}\lambda&\leq \int_{-2\pi/|\tau|}^{\pi/|\tau|}\left|k^n(\lambda)\right|\, \mathrm{d}\lambda+ \int_{-2\pi/|\tau|}^{\pi/|\tau|}\left|k^n\left(\lambda+\frac{\pi}{|\tau|}\right)\right|\, \mathrm{d}\lambda\\
		&\leq c_1\int_{-2\pi/|\tau|}^{\pi/|\tau|} |\lambda|^{-\beta}\, \mathrm{d}\lambda+ c_1\int_{-2\pi/|\tau|}^{\pi/|\tau|} \left|\lambda+\frac{\pi}{|\tau|} \right|^{-\beta}\, \mathrm{d}\lambda\\
		&= 2c_1\int_{-2\pi/|\tau|}^{\pi/|\tau|}|\lambda|^{-\beta}\, \mathrm{d}\lambda= O\left(|\tau|^{\beta-1}\right)\ \ \mbox{as $|\tau|\to\infty$}.
	\end{align*}
	This completes the proof in the case of $\beta\in (0,1)$.
\end{proof}
\begin{proof}[Proof of Lemma E.1 in Case (2)]
	Consider the case of $\beta\in(-1,0)$. Let $\tau\in\mathbb{Z}\backslash\{0\}$. Since the continuity of $k^n$ on $[-\pi,\pi]\backslash\{0\}$ implies $k^n(\pi)=k^n(-\pi)$, the integration by parts formula yields 
	\begin{equation*}
		\widehat{k^n}(\tau)=-\frac{1}{\sqrt{-1}\tau}\int_{-\pi}^{\pi} e^{\sqrt{-1}\tau\lambda}\frac{\partial k^n}{\partial\lambda}(\lambda)\, \mathrm{d}\lambda.
	\end{equation*}
	Moreover, since the derivative $\frac{\partial k^n}{\partial\lambda}$ is also $2\pi$-periodic from the assumption, the argument in the case (1) can be applied so that we obtain
	\begin{equation*}
		\sup_{n\in\mathbb{N}}\left|\widehat{k^n}(\tau)\right|=\frac{1}{|\tau|}O\left(|\tau|^{(\beta-1)-1}\right)=O\left(|\tau|^{\beta-1}\right)\ \ \mbox{as $|\tau|\to\infty$}.
	\end{equation*}
	This completes the proof in the case of $\beta\in(-1,0)$.
\end{proof}
\begin{proof}[Proof of Lemma~E.2]
	The same argument in Lemma~E.1 shows the inequality (\ref{Upper.Estimate.1}). 
	The assumption implies that 
	\begin{equation*}
		c_2:=\sup_{n\in\mathbb{N},\lambda\in[-\pi, \pi]\backslash\{0\}}\left\{\left|k^n(\lambda)\right|+|\lambda|\left|\frac{\partial k^n}{\partial\lambda}(\lambda)\right|\right\}<\infty.
	\end{equation*}
	By the mean value theorem, the similar argument in Lemma~E.1 yields
	\begin{align*}
		\hspace{-1cm}\int_{-\pi}^{-2\pi/|\tau|}\left|k^n(\lambda)-k^n\left(\lambda+\frac{\pi}{|\tau|}\right)\right|\, \mathrm{d}\lambda \leq& c_2\frac{\pi}{|\tau|} \int_\pi^{(|\tau|-1)\pi} \lambda^{-1}\, \mathrm{d}\lambda \\
		=& c_2\frac{\pi}{|\tau|}\{\log((|\tau|-1)\pi) - \log\pi\} =O\left(|\tau|^{-1}\log|\tau|\right)
	\end{align*}
	as $|\tau|\to\infty$. A similar argument shows that
	\begin{equation*}
		\sup_{n\in\mathbb{N}}\int_{\pi/|\tau|}^{\pi}\left|k^n(\lambda)-k^n\left(\lambda+\frac{\pi}{|\tau|}\right)\right|\, \mathrm{d}\lambda =O\left(|\tau|^{-1}\log|\tau|\right)\ \ \mbox{as $|\tau|\to\infty$}.
	\end{equation*}
	Since $k^n(\lambda)$ is bounded a.e. from the assumption, the same argument in Lemma~E.1 yields
	\begin{equation*}
		\sup_{n\in\mathbb{N}}\int_{-2\pi/|\tau|}^{\pi/|\tau|}\left|k^n(\lambda) - k^n\left(\lambda+\frac{\pi}{|\tau|}\right)\right|\, \mathrm{d}\lambda =O\left(|\tau|^{-1}\right)\ \ \mbox{as $|\tau|\to\infty$}.
	\end{equation*}
	This completes the proof of Lemma~E.2.
\end{proof}
\subsection{Proof of Theorem E.3}\label{SubSec_FT87_Trace}
\subsubsection{Outline of Proof of Theorem~E.3}
Fix $p\in\mathbb{N}$ and note that
\begin{align*}
	&\mathrm{Tr}\left[\left(\Sigma_n(k_1^n)\Sigma_n(k_2^n)\right)^p\right] \\ =&\sum_{j_1=0}^{n-1}\cdots\sum_{j_{2p}=0}^{n-1}\widehat{k_1^n}(j_1-j_2)\widehat{k_2^n}(j_2-j_3)\cdots\widehat{k_1^n}(j_{2p-1}-j_{2p})\widehat{k_2^n}(j_{2p}-j_1) \\
	=&\sum_{j_1=0}^{n-1}\cdots\sum_{j_{2p}=0}^{n-1}
	\left(\int_{-\pi}^\pi\cdots\int_{-\pi}^\pi e^{\sqrt{-1}(j_1-j_2)y_1}e^{\sqrt{-1}(j_2-j_3)y_2}\cdots e^{\sqrt{-1}(j_{2p}-j_1)y_{2p}} k_1^n(y_1)k_2^n(y_2)\cdots k_1^n(y_{2p-1})k_2^n(y_{2p})\, \mathrm{d}y_1\cdots \mathrm{d}y_{2p}\right) \\
	=&\int_{U_\pi}P_n(\mathbf{y})Q_n(\mathbf{y})\, \mathrm{d}\mathbf{y},
\end{align*}
where $U_t:=[-t,t]^{2p}$ for $t\in (0, \pi]$ and
\begin{align*}
	&P_n(\mathbf{y}):=h_n^\ast(y_1-y_{2p})h_n^\ast(y_2-y_1)\cdots h_n^\ast(y_{2p}-y_{2p-1}),\ \ h_n^\ast\left(y\right):=\sum_{j=0}^{n-1}e^{\sqrt{-1}jy}, \\
	&Q_n(\mathbf{y}):=k_1^n(y_1)k_2^n(y_2)\cdots k_1^n(y_{2p-1})k_2^n(y_{2p}).
\end{align*}

Following the arguments of Fox and Taqqu~\cite{FT87}, we divide $U_{\pi}$ into three disjoint sets $E_t$, $F_t$, $G$ given by
\begin{align*}
	&E_t:=U_\pi\backslash\{U_t\cup W\},\ \ F_t:=U_t\backslash W,\ \ G:=U_\pi\cap W,
\end{align*}
where $t\in (0, \pi]$ and
\begin{align*}
	&W_j:=\left\{\mathbf{y}=(y_1,\cdots,y_{2p})\in\mathbb{R}^{2p}:|y_j|\leq\frac{|y_{j+1}|}{2}\right\},\hspace{0.2cm}j=1,\cdots,2p, \\
	&W:=W_1\cup W_2\cup\cdots W_{2p}.
\end{align*}
Note that we use the notation $y_{2p+1}\equiv y_1$ for simplicity. 

In order to prove the first result of Theorem~E.3, it suffices to prove that $p(\alpha_1+\alpha_2)<1$ implies the following three results:
\begin{align}
	&\lim_{n\to\infty}\frac{1}{n}\int_{E_t}P_n(\mathbf{y})Q_n(\mathbf{y})\, \mathrm{d}\mathbf{y} =(2\pi)^{2p-1}\int_{t\leq|z|\leq\pi}[f(z)g(z)]^p\, \mathrm{d}z,\ \ \forall t\in(0,1], \label{Conv.Main} \\
	&\lim_{t\to 0}\limsup_{n\to\infty}\frac{1}{n}\int_{F_t}P_n(\mathbf{y})Q_n(\mathbf{y})\, \mathrm{d}\mathbf{y}=0, \label{Conv.Reminder.F} \\
	&\lim_{n\to\infty}\frac{1}{n}\int_{G}P_n(\mathbf{y})Q_n(\mathbf{y})\, \mathrm{d}\mathbf{y}=0. \label{Conv.Reminder.G}
\end{align}
\begin{rem}\rm
	In order to prove (\ref{Conv.Reminder.F}), we will show that $p(\alpha_1+\alpha_2)<1$ implies
	\begin{equation}
		\lim_{t\to 0}\limsup_{n\to\infty}\frac{1}{n}\int_{U_t}P_n(\mathbf{y})Q_n(\mathbf{y})\, \mathrm{d}\mathbf{y}=0. \label{Conv.Reminder.F2}
	\end{equation}
\end{rem}
\begin{rem}\rm
	Since $G=\bigcup_{j=1}^{2p}[U_\pi\cap W_j]$, the relation (\ref{Conv.Reminder.G}) will hold if we prove that $p(\alpha_1+\alpha_2)<1$ implies
	\begin{equation}\label{Conv.Reminder.G2}
		\lim_{n\to\infty}\frac{1}{n}\int_{U_\pi\cap W_j}|P_n(\mathbf{y})Q_n(\mathbf{y})|\, \mathrm{d}\mathbf{y}=0,\ \ j=1,\cdots,2p.
	\end{equation}
	From the definition of $P_n$ and $Q_n$, it is clear that
	\begin{equation*}
		\int_{U_\pi\cap W_1}|P_n(\mathbf{y})Q_n(\mathbf{y})|\, \mathrm{d}\mathbf{y} =\int_{U_\pi\cap W_3}|P_n(\mathbf{y})Q_n(\mathbf{y})|\, \mathrm{d}\mathbf{y} =\cdots =\int_{U_\pi\cap W_{2p-1}}|P_n(\mathbf{y})Q_n(\mathbf{y})|\, \mathrm{d}\mathbf{y}
	\end{equation*}
	and
	\begin{equation*}
		\int_{U_\pi\cap W_2}|P_n(\mathbf{y})Q_n(\mathbf{y})|\, \mathrm{d}\mathbf{y} =\int_{U_\pi\cap W_4}|P_n(\mathbf{y})Q_n(\mathbf{y})|\, \mathrm{d}\mathbf{y} =\cdots =\int_{U_\pi\cap W_{2p}}|P_n(\mathbf{y})Q_n(\mathbf{y})|\, \mathrm{d}\mathbf{y}.
	\end{equation*}
	Because of the symmetry between $\alpha_1$ and $\alpha_2$ in the hypothesis of theorem, it is clear that we prove that $p(\alpha_1+\alpha_2)<1$ implies
	\begin{equation}\label{Conv.Reminder.G3}
		\lim_{n\to\infty}\frac{1}{n}\int_{U_\pi\cap W_1}|P_n(\mathbf{y})Q_n(\mathbf{y})|\, \mathrm{d}\mathbf{y}=0,
	\end{equation}
	then we will have also established 
	\begin{equation*}
		\lim_{n\to\infty}\frac{1}{n}\int_{U_\pi\cap W_2}|P_n(\mathbf{y})Q_n(\mathbf{y})|\, \mathrm{d}\mathbf{y}=0.
	\end{equation*}
	Thus (\ref{Conv.Reminder.G2}) will follow from (\ref{Conv.Reminder.G3}). 
\end{rem}
In conclusion, the first result of Theorem~E.3 will be proven if we show that $p(\alpha_1+\alpha_2)<1$ implies (\ref{Conv.Main}), (\ref{Conv.Reminder.F2}) and (\ref{Conv.Reminder.G3}). 
Moreover, the second result of Theorem~E.3 will be proven if we show that $p(\alpha_1+\alpha_2)\geq 1$ implies
\begin{equation}\label{Conv.Main2}
	\forall\psi>0,\ \ \int_{U_\pi}|P_n(\mathbf{y})Q_n(\mathbf{y})|\, \mathrm{d}\mathbf{y}=O(n^{p(\alpha_1+\alpha_2)+\psi})\ \ \mbox{as $n\to\infty$.} 
\end{equation}
These results will be proven in Section~\ref{Subsection.Proof.Main.Theorem}. 
In the next subsection, we summarize several preliminaries used in the proof of Theorem~E.3 following with Fox and Taqqu~\cite{FT87}. 
\subsubsection{Preliminaries}
To state the lemma, introduce the diagonal
\begin{equation*}
	D:=\{\mathbf{y}=(y_1,\cdots,y_{2p})\in U_\pi : y_1=y_2=\cdots=y_{2p}\}.
\end{equation*}
Let $\mu$ be the measure on $U_\pi$ which is concentrated on $D$ and satisfies $\mu(\{\mathbf{y} : a\leq y_1=y_2=\cdots=y_{2p}\leq b\})=b-a$ for all $-\pi\leq a\leq b\leq\pi$. Thus $\mu$ is Lebesgue measure on $D$, normalized so that $\mu(D)=2\pi$.
\begin{lem}[cf. Lemma 7.1. in~\cite{FT87}]\label{FT87.Lemma7.1.}
	Define a measure $\mu_n$ on $U_\pi$ by
	\begin{equation}\label{Def.mu_n}
		\mu_n(A):=\frac{1}{(2\pi)^{2p-1}n}\int_A P_n(\mathbf{y})\, \mathrm{d}\mathbf{y}
	\end{equation}
	for each measurable set $A\subset U_\pi$. Then $\mu_n$ converges weakly to $\mu$ as $n\to\infty$. 
\end{lem}
For each $n\in\mathbb{N}$, define the function
\begin{equation*}
	h_n(z):=\left\{
	\begin{array}{ll}
		\min\left(\frac{1}{|z+2\pi|},n\right)& \mbox{if $-2\pi\leq z<-\pi$}, \\
		\min\left(\frac{1}{|z|},n\right)& \mbox{if $-\pi\leq z<\pi$}, \\
		\min\left(\frac{1}{|z-2\pi|},n\right)& \mbox{if $\pi\leq z\leq 2\pi$}.
	\end{array}
	\right.
\end{equation*}
and the function $f_n : \mathbb{R}^{2p}\to\mathbb{R}$ by
\begin{align*}
	f_n(\mathbf{y}):=&h_n\left(y_1-y_{2p}\right)h_n\left(y_2-y_1\right)h_n\left(y_3-y_2\right)\cdots h_n\left(y_{2p}-y_{2p-1}\right) \\
	&\times |y_1|^{-\alpha_1}|y_2|^{-\alpha_2}|y_3|^{-\alpha_1}\cdots|y_{2p}|^{-\alpha_2},
\end{align*}
where $\alpha_1,\alpha_2<1$.  
\begin{lem}\label{Lemma.Dominating.Function}
	There exists a constant $c>0$ such that for each $n\in\mathbb{N}$ and $\mathbf{y}\in U_\pi$, 
	\begin{equation*}
		\left|P_n(\mathbf{y})Q_n(\mathbf{y})\right|\leq cf_n(\mathbf{y}).
	\end{equation*}
\end{lem}
\begin{proof}
	As shown in~\cite{FT87}, p.237, we have 
	\begin{equation*}
		|P_n(\mathbf{y})|\leq 4^{2p}h_n\left(y_1-y_{2p}\right)h_n\left(y_2-y_1\right)h_n\left(y_3-y_2\right)\cdots h_n\left(y_{2p}-y_{2p-1}\right)
	\end{equation*}
	for each $n\in\mathbb{N}$ and $\mathbf{y}\in U_\pi$. Therefore, the conclusion follows from the assumption.
\end{proof}
\begin{prop}[cf. Proposition 6.1. in~\cite{FT87}]
	Let $\alpha_1, \alpha_2<1$ and $W_1=\{\mathbf{y}\in\mathbb{R}^{2p}:|y_1|\leq \frac{|y_2|}{2}\}$. 
	\begin{enumerate}[a)]
		\item If $\alpha_1+\alpha_2\leq 0$, then for any $\psi>0$,
		\begin{equation*}
			\int_{U_\pi\cap W_1} f_n(\mathbf{y})\, \mathrm{d}\mathbf{y}=O(n^{\psi})\ \ \mbox{as $n\to\infty$.}
		\end{equation*}
		\item If $\alpha_1+\alpha_2>0$, then for any $\psi>0$,
		\begin{equation*}
			\int_{U_\pi\cap W_1} f_n(\mathbf{y})\, \mathrm{d}\mathbf{y}=O(n^{p(\alpha_1+\alpha_2)+\psi})\ \ \mbox{as $n\to\infty$.}
		\end{equation*}
	\end{enumerate}
\end{prop}
\begin{prop}[cf. Proposition 6.2. in~\cite{FT87}]
	Let $\alpha_1, \alpha_2<1$.
	\begin{enumerate}[a)]
		\item If $p(\alpha_1+\alpha_2)<1$, then
		\begin{equation*}
			\lim_{t\to 0}\limsup_{n\to\infty}\frac{1}{n}\int_{U_t}f_n(\mathbf{y})\, \mathrm{d}\mathbf{y}=0.
		\end{equation*}
		\item If $p(\alpha_1+\alpha_2)\geq 1$, then for any $\psi>0$,
		\begin{equation*}
			\int_{U_\pi} f_n(\mathbf{y})\, \mathrm{d}\mathbf{y}=O(n^{p(\alpha_1+\alpha_2)+\psi})\ \ \mbox{as $n\to\infty$.}
		\end{equation*}
	\end{enumerate}
\end{prop}
\subsubsection{Proof of Theorem~E.3}\label{Subsection.Proof.Main.Theorem}
As mentioned in Fox and Taqqu~\cite{FT87}, p.237-238, the results (\ref{Conv.Reminder.F2}), (\ref{Conv.Reminder.G3}) and (\ref{Conv.Main2}) immediately follow from Proposition 6.1., Proposition 6.2. in~\cite{FT87} in addition to Lemma~\ref{Lemma.Dominating.Function}. 
In the rest of this section, we will prove (\ref{Conv.Main}). Note that
\begin{equation*}
	\frac{1}{n}\int_{E_t}P_n(\mathbf{y})Q_n(\mathbf{y})\, \mathrm{d}\mathbf{y} = (2\pi)^{2p-1}\int_{E_t}Q_n(\mathbf{y})\, \mu_n(\mathbf{dy}),
\end{equation*} 
where $\mu_n$ is given in (\ref{Def.mu_n}), and set 
\begin{equation*}
	Q(\mathbf{y}):=k_1(y_1)k_2(y_2)\cdots k_1(y_{2p-1})k_2(y_{2p}),\ \ \mathbf{y}=(y_1,\cdots,y_{2p})\in E_t.
\end{equation*}
Since the assumptions imply
\begin{equation*}
	\lim_{n\to\infty}\esssup_{\lambda\in[-\pi,\pi]}|Q^n(\lambda)-Q(\lambda)|=0
\end{equation*}
and the limit function $Q$ is continuous a.e. and bounded on $E_t$ for each $t\in(0,\pi]$, see Fox and Taqqu~\cite{FT87}, p.237, for more detail, Lemma 7.1. in Fox and Taqqu~\cite{FT87} yields 
\begin{align*}
	\frac{1}{n}\int_{E_t}P_n(\mathbf{y})Q_n(\mathbf{y})\, \mathrm{d}\mathbf{y}
	&=(2\pi)^{2p-1}\int_{E_t}Q_n(\mathbf{y})\, \mu_n(\mathbf{dy}) \\
	&=(2\pi)^{2p-1}\int_{E_t}\left(Q_n(\mathbf{y})-Q(\mathbf{y})\right)\, \mu_n(\mathbf{dy}) +(2\pi)^{2p-1}\int_{E_t}Q(\mathbf{y})\, \mu_n(\mathbf{dy}) \\
	&\stackrel{n\to\infty}{\rightarrow}(2\pi)^{2p-1}\int_{E_t}Q(\mathbf{y})\, \mu(\mathbf{dy}) =(2\pi)^{2p-1}\int_{[-\pi,\pi]\backslash[-t,t]}[f(z)g(z)]^p\, \mathrm{d}z.
\end{align*}
Therefore, the conclusion follows.
\section{Limit Theorems of Quadratic Forms}
In this section, we derive several limit theorems of the quadratic form of random sequence which are used in the proof of Proposition~\ref{Uniform.Convergence.Contrast} and Proposition~\ref{Scaled.Score.Negligibility} under the following assumptions.
\begin{assump}\label{Assumption.SPD}\rm 
	Recall $\Theta := \Theta_H \times \Theta_\eta$ is a compact set of the form
	$\Theta_H := [H_-,H_+] \subset (0,1]$ and
	$\Theta_\eta := [\eta_-,\eta_+] \subset (0,\infty)$. Let us consider a function $k:[-\pi, \pi]\times\Theta\times\mathbb{N}\to[-\infty,\infty]$, denoted by $k_\vartheta^n(\lambda)\equiv k(\lambda,\vartheta,n)$, be even and integrable on $[-\pi, \pi]$ for each $\vartheta\in\Theta$ and $n\in\mathbb{N}$ and assume there exist monotonically increasing continuous functions $\beta_0, \beta_1:\Theta_H\to(-1,1)$ such that the function $k$ satisfies the conditions (C.\ref{Assumption.LimitSPD})-(C.\ref{Assumption.Derivative.SPD}) below on a restricted parameter space $\Theta_0(\xi):=\Theta_{H,0}(\xi)\times\mathcal{K}$, where $\mathcal{K}$ be a compact interval of $(0,\infty)$ and 
	\begin{equation*}
	\hspace{-1cm}\Theta_{H,0}(\xi):=\{H\in\Theta_H:-\beta_0(H)-\alpha(H_0)\geq -1+\xi, -\beta_1(H)-\alpha(H_0)\geq -1+\xi\},\ \ \xi\in(0,1).
	\end{equation*}
	Here $H_0$ denotes the true value of $H\in\Theta_H$, the function $\alpha:\Theta_H\to(-1,1)$ is given in Lemma~\ref{Lemma.SPD} and we only consider sufficiently small $\xi\in(0,1)$ such that $\mathring{\Theta}_{H,0}(\xi)\neq\emptyset$, where $\mathring{\Theta}_{H,0}(\xi)$ is the set of all interior points of $\Theta_{H,0}(\xi)$.
	\begin{enumerate}[(C.$1$)]
		\item\label{Assumption.LimitSPD} 
		For each $\vartheta\in\Theta_0(\xi)$, there exists a function $k_\vartheta$ such that
		\begin{equation*}
		\lim_{n\to\infty}\esssup_{\lambda\in[-\pi,\pi]}|k_\vartheta^n(\lambda)-k_\vartheta(\lambda)|=0,
		\end{equation*}
		and the discontinuities of $k_\vartheta$ has the Lebesgue measure $0$ for each $\vartheta\in\Theta_0(\xi)$.
		\item \label{Assumption.AsymptoticSPD}
		For each $\vartheta\in\Theta_0(\xi)$, the following relations hold:
		\begin{equation*}
		\sup_{n\in\mathbb{N},\lambda\in[-\pi, \pi]\backslash\{0\}}|\lambda|^{\beta_0(H)}\left|k_\vartheta^n(\lambda)\right|<\infty.
		\end{equation*}
		\item\label{Assumption.Derivative.SPD}
		For each $\lambda\in[-\pi,\pi]\backslash\{0\}$, $k_\vartheta^n(\lambda)$ is differentiable with respect to $\vartheta\in\Theta_0(\xi)$ and its partial derivatives satisfy
		\begin{equation*}
		\sup_{\substack{n\in\mathbb{N},\lambda\in[-\pi, \pi]\backslash\{0\},\\ \vartheta=(\vartheta_1,\vartheta_2)\in\Theta_0(\xi)}}|\lambda|^{\beta_1(\vartheta_1)}\left|\frac{\partial k_\vartheta^n}{\partial\vartheta_j}(\lambda)\right|<\infty,\ \ j=1,2.
		\end{equation*}
	\end{enumerate}
\end{assump}
\subsection{Basic Properties of Bilinear and Quadratic Forms}\label{Basic.Properties.BLF}
At first, we summarize several basic properties of the bilinear form $B_n$ and the quadratic form $Q_n$ as functionals on $L^1[-\pi,\pi]$ without proofs. 
\begin{lem}\label{Basic.BilinearForm}
	Let $\mathbf{x},\mathbf{y}\in\mathbb{C}^n$. The functionals $B_n(\mathbf{x},\mathbf{y},\cdot)$ and $Q_n(\mathbf{x},\cdot)$ on $L^1[-\pi,\pi]$ satisfy the following properties.
	\begin{enumerate}[$(1)$]
		\item \label{Linearlity.Bilinear.Form}
		For each $\mathbf{x},\mathbf{y}\in\mathbb{C}^n$, the functional $B_n(\mathbf{x},\mathbf{y},\cdot)$ is linear on $L^1[-\pi,\pi]$. 
		\item \label{NonDecreasing.QuadraticForm}
		For each $\mathbf{x}\in\mathbb{C}^n$, the functional $Q_n(\mathbf{x},\cdot)$ is non-decreasing on $L^1[-\pi,\pi]$, i.e. for each $k_1, k_2\in L^1[-\pi,\pi]$,
		\begin{equation*}
		Q_n(\mathbf{x}, k_1)\leq  Q_n(\mathbf{x}, k_2)\ \ \mbox{if}\ \ k_1\leq k_2,
		\end{equation*}
		where $k_1\leq k_2$ means $k_1(\lambda)\leq k_2(\lambda)$ for a.e. $\lambda\in[-\pi,\pi]$.
		\item \label{NonNegativity.QuadraticForm}
		For each $\mathbf{x}\in\mathbb{C}^n$, $Q_n(\mathbf{x}, k)\geq 0$ if $k\in L^1[-\pi,\pi]$ satisfies $k\geq 0$.
		\item \label{Positivity.QuadraticForm}
		For each $\mathbf{x}\in\mathbb{C}^n$ with $\mathbf{x}\neq 0$, $Q_n(\mathbf{x}, k)>0$ if $k\in L^1[-\pi,\pi]$ satisfies $k\geq 0$ and the set $\{\lambda\in[-\pi,\pi]:k(\lambda)> 0\}$ has a positive Lebesgue measure.
	\end{enumerate}
\end{lem}
Next lemma is useful to evaluate asymptotic behaviors of bilinear forms.
\begin{lem}\label{BilinearForm.Dominated.Inequality}
	Suppose a sequence of functions $k_\vartheta^n$, $n\in\mathbb{N}$, satisfies the condition (C.\ref{Assumption.AsymptoticSPD}) in Assumption~\ref{Assumption.SPD}. Then there exists an even and $2\pi$-periodic function $k_\vartheta^\dagger$, which is independent of the asymptotic parameter $n\in\mathbb{N}$, such that 
	\begin{equation*}
	\sup_{n\in\mathbb{N}}\left|k_\vartheta^n(\lambda)\right| \leq \left|k_\vartheta^\dagger(\lambda)\right|\hspace{0.2cm}\mbox{and}\hspace{0.2cm}\sup_{\vartheta\in\Theta_0,\lambda\in[-\pi,\pi]\backslash\{0\}}\left\{\left|\lambda\right|^{\beta_0(H)}\left|k_\vartheta^\dagger(\lambda)\right|\right\} < \infty .
	\end{equation*}
	Moreover, the following two inequalities hold for each $\mathbf{x}, \mathbf{y}\in\mathbb{C}^n$ and $\vartheta\in\Theta_0$:
	\begin{align}
	&\left|Q_n\left(\mathbf{x}, k_\vartheta^n\right)\right| \leq Q_n\left(\mathbf{x}, \left|k_\vartheta^n\right|\right) \leq Q_n(\mathbf{x}, k_\vartheta^\dagger), \label{Dominated.QuadraticForm} \\ 
	&\left|B_n\left(\mathbf{x}, \mathbf{y}, k_\vartheta^n\right)\right| \leq 2\sqrt{Q_n\left(\mathbf{x}, k_\vartheta^\dagger\right)} \sqrt{Q_n\left(\mathbf{y}, k_\vartheta^\dagger\right)}.
	\end{align}
\end{lem}
\begin{proof}
	Define a function $k_\vartheta^\dagger$ by
	\begin{align*}
	&k_\vartheta^\dagger(\lambda) := c\{2(1-\cos\lambda)\}\sum_{j\in\mathbb{Z}}|\lambda + 2\pi j|^{-\beta_0(H)-2}\\ 
	&\mbox{with}\ \ c:=  \sup_{\vartheta\in\Theta_0,\lambda\in[-\pi,\pi]\backslash\{0\},n\in\mathbb{N}}\left\{\frac{|\lambda|^2}{2(1-\cos\lambda)}\cdot\left|\lambda\right|^{\beta_0(H)}\left|k_\vartheta^n(\lambda)\right|\right\}.
	\end{align*}
	Then it is obvious that the function $k_\vartheta^\dagger$ satisfies all conditions mentioned at the beginning. 
	Moreover, the first inequality immediately follows from Lemma~\ref{Basic.BilinearForm} (\ref{NonDecreasing.QuadraticForm}). 
	In the rest of this proof, we will prove the second inequality. Decompose $k_\vartheta^n$ into the following two non-negative functions:
	\begin{equation*}
	k_\vartheta^n(\lambda) = k_{\vartheta,+}^n(\lambda) - k_{\vartheta,-}^n(\lambda),\ \ \mbox{where}\ \ k_{\vartheta,+}^n(\lambda):=\max(k_\vartheta^n(\lambda), 0), \ \ k_{\vartheta,-}^n(\lambda):=\max(-k_\vartheta^n(\lambda), 0).
	\end{equation*}
	Note that both of $k_{\vartheta,+}^n$ and $k_{\vartheta,-}^n$ are even functions and satisfy the condition (C.\ref{Assumption.AsymptoticSPD}) from the assumptions of $k_\vartheta^n$. At first, consider the case where both of $k_{\vartheta,+}^n$ and $k_{\vartheta,-}^n$ are positive almost everywhere. 
	Since Lemma~\ref{Basic.BilinearForm} (\ref{Positivity.QuadraticForm}) yields the matrix $\Sigma_n(k_\vartheta^n)$ is positive definite, Lemma~\ref{Basic.BilinearForm} (\ref{Linearlity.Bilinear.Form}), Schwartz's inequality of bilinear forms and (\ref{Dominated.QuadraticForm}) yield that for each $\mathbf{x},\mathbf{y}\in\mathbb{C}^n$,
	\begin{align*}
	\left|B_n\left(\mathbf{x}, \mathbf{y}, k_\vartheta^n\right)\right| \leq&\left|B_n\left(\mathbf{x}, \mathbf{y}, k_{\vartheta,+}^n\right)\right| +  \left|B_n\left(\mathbf{x}, \mathbf{y}, k_{\vartheta,-}^n\right)\right| \\
	\leq&\sum_{i\in\{+,-\}}\sqrt{Q_n\left(\mathbf{x}, k_{\vartheta,i}^n\right)} \sqrt{Q_n\left(\mathbf{y}, k_{\vartheta,i}^n\right)} \leq 2\sqrt{Q_n\left(\mathbf{x}, k_\vartheta^\dagger\right)} \sqrt{Q_n\left(\mathbf{y}, k_\vartheta^\dagger\right)}.
	\end{align*}
	Note that the above inequalities also follows even if $k_{\vartheta,+}^n\equiv 0$ or $k_{\vartheta,-}^n\equiv 0$. Therefore, the conclusion follows.
\end{proof}
The following result immediately follows from Lemma~\ref{Basic.BilinearForm} and Lemma~\ref{BilinearForm.Dominated.Inequality}.
\begin{cor}\label{Dominating.Inequality.Y}
	Let $J\in\mathbb{N}$ and suppose a sequence of functions $k_\vartheta^n$, $n\in\mathbb{N}$, satisfies the condition (C.\ref{Assumption.AsymptoticSPD}) in Assumption~\ref{Assumption.SPD}.
	For any $n$-dimensional vector of the form $\mathbf{y}:=\sum_{j=0}^Ja_j\mathbf{w}_j$ with $\mathbf{w}_j\in\mathbb{C}^n$ and $a_j\in\mathbb{C}$ for $j\in\{0,1,2,\cdots,J\}$, the following inequality holds:
	\begin{align*}
	\left|Q_n\left(\mathbf{y}, k_\vartheta^n\right)-Q_n\left(a_0\mathbf{w}_0, k_\vartheta^n\right)\right| &\leq\sum_{i=0}^J\sum_{j=1}^J|a_i||a_j|\left|B_n\left(\mathbf{w}_i, \mathbf{w}_j, k_\vartheta^n\right)\right| \\
	&\leq 2\sum_{i=0}^J\sum_{j=1}^J|a_i||a_j|\sqrt{Q_n\left(\mathbf{w}_i, k_\vartheta^\dagger\right)}\sqrt{Q_n\left(\mathbf{w}_j, k_\vartheta^\dagger\right)},
	\end{align*}
	where $k_\vartheta^\dagger$ is given in Lemma~\ref{BilinearForm.Dominated.Inequality}.
\end{cor}
\subsection{Pointwise Convergence of Gaussian Quadratic Form}
Denote by $\widetilde{\mathbf{G}}_n:=\delta_n^{-H_0}\mathbf{G}_n$. In the next lemma, we show a pointwise convergence of the quadratic form of the stationary Gaussian sequence $\widetilde{\mathbf{G}}_n$, $n\in\mathbb{N}$.
\begin{lem}\label{PointwiseConvergence.QuadraticForm.Gaussian}
	Suppose a sequence of functions $k_\vartheta^n$, $n\in\mathbb{N}$, satisfies the conditions (C.\ref{Assumption.LimitSPD}) and (C.\ref{Assumption.AsymptoticSPD}) in Assumption~\ref{Assumption.SPD}. Under the conditions $(H.\ref{Assumption.HighFrequency})$ and $(H.\ref{Assumption.mn})$, the following convergence holds for each $\vartheta\in\Theta_0(\xi)$:
	\begin{equation*}
	\lim_{n\to\infty}\left\|Q_n\left(\widetilde{\mathbf{G}}_n, k_\vartheta^n\right) - Q_{\vartheta_0}\left(k_\vartheta\right)\right\|_2=0,
	\end{equation*}
	where $k_\vartheta$ is the limit function given in (C.\ref{Assumption.LimitSPD}) and
	\begin{equation}\label{Def.Q.Limit}
	Q_{\vartheta_0}(k_\vartheta) := \int_{-\pi}^\pi\eta_0^2f_{H_0}(\lambda)k_\vartheta(\lambda)\, \mathrm{d}\lambda.
	\end{equation}
\end{lem}
\begin{proof}
	At first, we obtain 
	\begin{align*}
	\left\|Q_n\left(\widetilde{\mathbf{G}}_n, k_\vartheta^n\right) - Q_{\vartheta_0}\left(k_\vartheta\right)\right\|_2^2
	&=\mbox{Var}\left[Q_n\left(\widetilde{\mathbf{G}}_n, k_\vartheta^n\right)\right] +\left\{\mbox{E}\left[Q_n\left(\widetilde{\mathbf{G}}_n, k_\vartheta^n\right)\right] -Q_{\vartheta_0}\left(k_\vartheta\right)\right\}^2 \\
	&=\frac{2}{(2\pi n)^2}\mbox{Tr}\left[\left(\Sigma_n\left(h_{\vartheta_0}^n\right)\Sigma_n\left(k_\vartheta^n\right)\right)^2\right] +\left(\frac{1}{2\pi n}\mbox{Tr}\left[\Sigma_n\left(h_{\vartheta_0}^n\right)\Sigma_n\left(k_\vartheta^n\right)\right] -Q_{\vartheta_0}\left(k_\vartheta\right)\right)^2,
	\end{align*}
	where $h_{\widetilde\vartheta}^n\equiv h_{H,\widetilde\nu}^n$ is given in (\ref{Rescaled.SPD}). Note that $\widetilde\vartheta_0:=(H_0,\delta_n^{-H_0}\nu_0)=\vartheta_0$. 
	Since $\vartheta=(H,\nu)\in\Theta_0(\xi)$ implies $\beta_0(H)+\alpha(H_0)<1$ and under the conditions $(H.\ref{Assumption.HighFrequency})$ and $(H.\ref{Assumption.mn})$, we have
	\begin{equation*}
	\sup_{\lambda\in[-\pi,\pi]}|h_{\vartheta_0}^n(\lambda)-\eta_0^2f_{H_0}(\lambda)|=\frac{1}{m_n\delta_n^{H_0}}\sup_{\lambda\in[-\pi,\pi]}|\ell(\lambda)|\stackrel{n\to\infty}{\rightarrow}0,
	\end{equation*}
	the conclusion follows from the conditions (C.\ref{Assumption.LimitSPD}), (C.\ref{Assumption.AsymptoticSPD}) and Theorem~\ref{Extension.FT87}.
\end{proof}

The following result is easily proven in the similar way to the proof of Lemma~\ref{PointwiseConvergence.QuadraticForm.Gaussian}.
\begin{cor}\label{Convergence.Variance.Gaussian.Quadratic}
	Suppose a sequence of functions $k_\vartheta^n$, $n\in\mathbb{N}$, satisfies the condition (C.\ref{Assumption.AsymptoticSPD}) in Assumption~\ref{Assumption.SPD}. Under the conditions $(H.\ref{Assumption.HighFrequency})$ and $(H.\ref{Assumption.mn})$, the following convergence holds for each $\vartheta=(H,\nu)\in\Theta_0(\xi)$ satisfying $\alpha(H_0)+\beta_0(H)<1/2$:
	\begin{equation*}
	Q_n\left(\widetilde{\mathbf{G}}_n, k_\vartheta^n\right) - E\left[Q_n\left(\widetilde{\mathbf{G}}_n, k_\vartheta^n\right)\right]=O_P\left(1/\sqrt{n}\right)\ \ \mbox{as}\ \ n\to\infty.
	\end{equation*}
\end{cor}

\subsection{Pointwise Convergence of Quadratic Form of Observation $\mathbf{Y}_n$}
Denote by $\widetilde{\mathbf{Y}}_n:=\delta_n^{-H_0}\mathbf{Y}_n$. Our goal in this subsection is to prove that the quadratic form of the rescaled observation $\widetilde{\mathbf{Y}}_n$ and that of the Gaussian vector $\widetilde{\mathbf{G}}_n$ are asymptotically equivalent as $\delta_n\to 0$. Namely, we show the following result.
\begin{prop}\label{Convergence.QuadraticForm.Remainder}
	Suppose a sequence of functions $k_\vartheta^n$, $n\in\mathbb{N}$, satisfies the condition (C.\ref{Assumption.AsymptoticSPD}) in Assumption~\ref{Assumption.SPD}.
	Under the conditions $(H.\ref{Assumption.HighFrequency})-(H.\ref{Assumption.mn})$, there exists a constant $\psi>0$ such that the following convergence holds for each $\vartheta\in\Theta_0(\xi)$:
	\begin{equation*}
	Q_n\left(\mathbf{Y}_n, k_\vartheta^n\right) =Q_n\left(\mathbf{G}_n, k_\vartheta^n\right) +o_P\left(\delta_n^{2H_0+\psi}\right)\ \ \mbox{as}\ \ n\to\infty.
	\end{equation*}
\end{prop}
\begin{proof}
	From Proposition~\ref{Expansion.Formula.Y}, Corollary~\ref{Dominating.Inequality.Y} and Lemma~\ref{PointwiseConvergence.QuadraticForm.Gaussian}, it suffices to prove the following two results for the non-negative function $k_\vartheta^\dagger$ given in Lemma~\ref{BilinearForm.Dominated.Inequality} and each $\vartheta\in\Theta_0(\xi)\times(0,\infty)$:
	\begin{enumerate}[(R.1)]
		\item \label{Lemma.Uniform.Convergence.Wn}
		For any $K\in\mathbb{N}$ and $\mathbf{p}\equiv(p_1,\cdots,p_K)\in\mathbb{N}^K$, the following relation holds:
		\begin{equation*}
		Q_n\left(\Delta\mathbf{W}_n^{\mathbf{p}}, k_\vartheta^\dagger\right) =O_P\left(\delta_n^{2|\mathbf{p}|H_0}\right)\ \ \mbox{as}\ \ n\to\infty.
		\end{equation*}
		\item \label{Negligibility.QuadraticForm.Rn}
		Assume that there exists a positive random variable $A$, which is independent of the asymptotic parameter $n\in\mathbb{N}$, such that a random vector $\mathbf{R}_n:=(R_1^n,R_2^n,\cdots,R_n^n)$ satisfies
		\begin{equation}\label{Assumption.Reminder}
		\sup_{t\in\Lambda_n}\left|R_t^n\right|\leq A\cdot\delta_n.
		\end{equation}
		Then there exists a constant $\psi>0$ such that the following relation holds:
		\begin{equation*}
		Q_n\left(\mathbf{R}_n,k_\vartheta^\dagger\right) =o\left(\delta_n^{2H_0+\psi}\right)\ \ \mbox{as}\ \ n\to\infty.
		\end{equation*}
	\end{enumerate}
	In the rest of this proof section, we prove (R.\ref{Lemma.Uniform.Convergence.Wn}) and (R.\ref{Negligibility.QuadraticForm.Rn}).
\end{proof}

\begin{proof}[Proof of (R.\ref{Lemma.Uniform.Convergence.Wn})]
	Fix $\vartheta\in\Theta_0(\xi)\times(0,\infty)$. At first, Chebyshev's inequality and Lemma~\ref{Basic.BilinearForm} (\ref{NonNegativity.QuadraticForm}) yield that the following inequality holds for any $M>0$:
	\begin{align}
	P\left[\left|Q_n\left(\Delta\mathbf{W}_n^\mathbf{p}, k_\vartheta^\dagger\right)\right|>  M\right] 
	\leq&\frac{1}{ M}E\left[Q_n\left(\Delta\mathbf{W}_n^\mathbf{p}, k_\vartheta^\dagger\right)\right] \nonumber \\
	=&\frac{1}{2\pi M}\cdot\frac{1}{n}\sum_{s,t=1}^n \widehat{k_\vartheta^\dagger}(s-t) \mathrm{Cov}\left[\Delta W_s^{n,\mathbf{p}}, \Delta W_t^{n,\mathbf{p}}\right] \nonumber \\
	=&\frac{1}{2\pi M}\sum_{|\tau|<n}\left(1-\frac{|\tau|}{n}\right)\widehat{k_\vartheta^\dagger}(\tau) \mathrm{Cov}\left[\Delta W_1^{n,\mathbf{p}}, \Delta W_{1+|\tau|}^{n,\mathbf{p}}\right], \label{Negligibility.DeltaWn.1}
	\end{align}
	where the stationarity property of $\{W_t^{n,\mathbf{p}}\}_{t\in\mathbb{Z}}$ is used in the last equality, see Proposition \ref{Key.Lemma.Result.Cor}. Since the function $k_\vartheta^\dagger$ satisfies the all assumptions in Lemma~\ref{Extension1.FT86} and Lemma~\ref{Extension2.FT86} with respect to $\beta\equiv\beta_0(H)$, we obtain
	\begin{equation}\label{k.dagger.Fourier.Decay}
	\widehat{k_\vartheta^\dagger}(\tau)=O\left(|\tau|^{\beta_0(H)-1}\right)\ \ \mbox{as}\ \ |\tau|\to\infty.
	\end{equation}
	As a result, (\ref{k.dagger.Fourier.Decay}) and Proposition~\ref{Key.Lemma.Result.Cor} yield that there exists a constant $c>0$ such that the last quantity of (\ref{Negligibility.DeltaWn.1}) is dominated by 
	\begin{align}
	\frac{1}{2\pi M}\sum_{|\tau|<n}\left|\widehat{k_\vartheta^\dagger}(\tau)\right| \left|\mathrm{Cov}\left[\Delta W_1^p, \Delta W_{1+|\tau|}^p\right]\right| \leq&\frac{c\delta_n^{2|\mathbf{p}|H_0}}{2\pi M}\sum_{|\tau|<n}|\tau|^{\beta_0(H)-1+(2H_0-4)} \nonumber \\
	\leq&\frac{c\delta_n^{2|\mathbf{p}|H_0}}{2\pi M}\sum_{\tau\in\mathbb{Z}}|\tau|^{\beta_0(H)+\alpha(H_0)-4}. \label{Negligibility.DeltaWn.2}
	\end{align}
	Note that the series in (\ref{Negligibility.DeltaWn.2}) converges because $H\in\Theta_{H,0}(\xi)$ implies $\beta_0(H)+\alpha(H_0)-4<-1$. 
	Since the last quantity of (\ref{Negligibility.DeltaWn.2}) is independent of the asymptotic parameter $n\in\mathbb{N}$, the conclusion follows as $M\to\infty$. 
\end{proof}
\begin{proof}[Proof of (R.\ref{Negligibility.QuadraticForm.Rn})]
	Fix $\xi\in(0,1)$. At first, (\ref{Assumption.Reminder}) and (\ref{k.dagger.Fourier.Decay}) yield that there exists a constant $c>0$ such that
	\begin{align}
	Q_n(\mathbf{R}_n, k_\vartheta^\dagger) =& \frac{1}{2\pi n}\sum_{s,t=1}^n \widehat{k_\vartheta^\dagger}(s-t)R_s^nR_t^n \nonumber \\
	\leq&\frac{A^2}{2\pi} \cdot\frac{\delta_n^2}{n}\sum_{s,t=1}^n \left|\widehat{k_\vartheta^\dagger}(s-t)\right| \nonumber\\
	=&\frac{A^2}{2\pi} \cdot\delta_n^2\sum_{|\tau|< n}\left(1-\frac{|\tau|}{n}\right) \left|\widehat{k_\vartheta^\dagger}(\tau)\right| \leq\frac{c A^2}{2\pi} \cdot\delta_n^2\sum_{|\tau|< n}\left|\tau\right|^{\beta_0(H)-1}. \label{An.QuadraticForm.1}
	\end{align}
	Moreover, for any $\psi\in(0,\xi)$, the last quantity of (\ref{An.QuadraticForm.1}) is dominated by
	\begin{equation}
	\delta_n^2\sum_{|\tau|<n}\left|\tau\right|^{\beta_0(H)-1} \leq\delta_n^2\sum_{|\tau|<n}\frac{n^{1-\alpha(H_0)-\psi}}{|\tau|^{1-\alpha(H_0)-\psi}}\left|\tau\right|^{\beta_0(H)-1} \leq\delta_n^{2H_0+\psi}T_n^{2-2H_0-\psi}\sum_{\tau\in\mathbb{Z}}\left|\tau\right|^{\alpha(H_0)+\beta_0(H)-2+\psi}. \label{Rn.Quad.2}
	\end{equation}
	Note that the series in (\ref{Rn.Quad.2}) converges because $\psi\in(0,\xi)$ and $H\in\Theta_{H,0}(\xi)$ imply $\alpha(H_0)+\beta_0(H)-2+\psi\leq -1+\psi-\xi<-1$. Then  the conclusion follows from (\ref{An.QuadraticForm.1}), (\ref{Rn.Quad.2}) and the assumptions $(H.\ref{Assumption.HighFrequency})$ and $(H.\ref{Assumption.Tn})$.
\end{proof}
We can obtain the following result from Lemma~\ref{PointwiseConvergence.QuadraticForm.Gaussian} and Proposition~\ref{Convergence.QuadraticForm.Remainder}.
\begin{cor}\label{Pointwise.Convergence.DeltaY}
	Suppose a sequence of functions $k_\vartheta^n$, $n\in\mathbb{N}$, satisfies the conditions (C.\ref{Assumption.LimitSPD}) and (C.\ref{Assumption.AsymptoticSPD}) in Assumption~\ref{Assumption.SPD}. Under the conditions $(H.\ref{Assumption.HighFrequency})-(H.\ref{Assumption.mn})$, the following convergence holds for each $\vartheta\in\Theta_0(\xi)$:
	\begin{equation*}
	Q_n\left(\widetilde{\mathbf{Y}}_n, k_\vartheta^n\right)=Q_{\vartheta_0}(k_\vartheta)+o_P(1)\ \ \mbox{as}\ \ n\to\infty.
	\end{equation*}
\end{cor}

\subsection{Uniform Convergence of Quadratic Form of Observations $\mathbf{Y}_n$}\label{Appendix.Uniform.Convergence}
In this subsection, we prove a uniform convergence of the quadratic form of $\widetilde{\mathbf{Y}}_n$ which is an extension of Corollary~\ref{Pointwise.Convergence.DeltaY} given in the previous subsection.
\begin{prop}\label{Uniform.Convergence.QuadraticForm.DeltaY}
	Suppose a sequence of functions $k_\vartheta^n$, $n\in\mathbb{N}$, satisfies the conditions (C.\ref{Assumption.LimitSPD})-(C.\ref{Assumption.Derivative.SPD}) in Assumption~\ref{Assumption.SPD}. Under the conditions $(H.\ref{Assumption.HighFrequency})-(H.\ref{Assumption.mn})$, the following uniform convergence holds:
	\begin{equation*}
	\sup_{\vartheta\in\Theta_0(\xi)}\left|Q_n\left(\widetilde{\mathbf{Y}}_n, k_\vartheta^n\right) - Q_{\vartheta_0}(k_\vartheta)\right| = o_P(1)\ \ \mbox{as}\ \ n\to\infty.
	\end{equation*}
\end{prop}
\begin{proof}
	Fix $\xi\in(0,1)$. At first, the compactness of $\Theta_0(\xi)$ yields that for each $r>0$, there exists $j(r)\in\mathbb{N}$ and a finite open covering $\{B_r(\vartheta_i)\}_{i\in\Lambda_{j(r)}}$ given by
	\begin{equation*}
	B_r(\vartheta_i) := \left\{\vartheta\in\Theta_0(\xi):\|\vartheta-\vartheta_i\|_{\mathbb{R}^2}< r\right\}\ \ \mbox{for}\ \  \vartheta_i=(H_i,\eta_i)\in\Theta_0(\xi), i\in\Lambda_{j(r)}.
	\end{equation*}
	Then we obtain the following inequality:
	\begin{align}
	\sup_{\vartheta\in\Theta_0(\xi)}\left|Q_n\left(\widetilde{\mathbf{Y}}_n, k_\vartheta^n\right) - Q_{\vartheta_0}(k_\vartheta)\right|\leq& \sup_{i\in\Lambda_{j(r)},\vartheta\in B_r(\vartheta_i)}\left|Q_n\left(\widetilde{\mathbf{Y}}_n, k_\vartheta^n\right) - Q_{\vartheta_0}(k_\vartheta)\right| \nonumber \\
	\leq& \max_{i\in\Lambda_{j(r)}}\left|Q_n(\widetilde{\mathbf{Y}}_n, k_{\vartheta_i}^n) - Q_{\vartheta_0}(k_{\vartheta_i})\right| \nonumber \\ &+\sup_{\substack{\|\vartheta_1-\vartheta_2\|_{\mathbb{R}^2}<r\\ \vartheta_1,\vartheta_2\in \Theta_0(\xi)\times\mathcal{K}}}\left|Q_{\vartheta_0}(k_{\vartheta_1}) - Q_{\vartheta_0}(k_{\vartheta_2})\right| \nonumber \\
	&+\sup_{i\in\Lambda_{j(r)},\vartheta\in B_r(\vartheta_i)}\left|Q_n\left(\widetilde{\mathbf{Y}}_n, k_{\vartheta}^n\right) - Q_n\left(\widetilde{\mathbf{Y}}_n, k_{\vartheta_i}^n\right)\right|. \label{Decomposition.Uniform.Convergence}
	\end{align}
	Here Corollary~\ref{Pointwise.Convergence.DeltaY} yields that for each $r>0$, the first term of the last quantity of (\ref{Decomposition.Uniform.Convergence}) converges to $0$ as $n\to\infty$. Moreover, the second term of it also converges to $0$ as $r\downarrow 0$ because $\vartheta\mapsto Q_{\vartheta_0}(k_\vartheta)$ is uniformly continuous on $\Theta_0(\xi)$ under the condition (C.\ref{Assumption.Derivative.SPD}). As a result, it suffices to show that the third term of it is negligible for sufficiently small $r>0$ and large $n\in\mathbb{N}$.
	
	Without loss of generality, we assume $r\in(0,\xi/2)$ and
	\begin{equation*}
	\sup_{H_1^\dagger,H_2^\dagger\in\Theta_{H,0}(\xi), |H_1^\dagger-H_2^\dagger|<r}|\beta_1(H_1^\dagger)-\beta_1(H_2^\dagger)\|<\xi/2
	\end{equation*}
	since $\beta_1$ is uniformly continuous on $\Theta_{H,0}(\xi)$. Here the condition (C.\ref{Assumption.Derivative.SPD}) implies
	\begin{equation*}
	c_1:=\sup_{\substack{n\in\mathbb{N},\lambda\in[-\pi, \pi]\backslash\{0\},\\ \vartheta=(H,\eta)\in\Theta_0(\xi)}}|\lambda|^{\beta_1(H)}\left\|\nabla k_\vartheta^n(\lambda)\right\|_{\mathbb{R}^2}<\infty.
	\end{equation*}
	Then the mean value theorem and Schwartz's inequality yield that for any $\vartheta_i^{\dagger,1},\vartheta_i^{\dagger,2}\in B_r(\vartheta_i)$, $i\in\Lambda_{j(r)}$ and $\lambda\in[-\pi,\pi]\backslash\{0\}$,
	\begin{equation}\label{UC.Mean.Value.}
	\left|k_{\vartheta_i^{\dagger,1}}^n(\lambda) - k_{\vartheta_i^{\dagger,2}}^n(\lambda)\right|\leq \left\|\nabla k_{\vartheta_i^\dagger}^n(\lambda)\right\|_{\mathbb{R}^2}\left\|\vartheta_i^{\dagger,1}-\vartheta_i^{\dagger,2}\right\|_{\mathbb{R}^2}\leq r c_1\left|\lambda\right|^{-\beta_1(H_i^\dagger)}\leq r c_2\left|\lambda\right|^{-\beta_1(H_i) -\xi/2},
	\end{equation}
	where $c_2:=c_1\pi^\xi$ and $\vartheta_i^\dagger\equiv(H_i^\dagger,\eta_i^\dagger)\in B_r(\vartheta_i)$ is determined by the relation $\vartheta_i^\dagger=\vartheta_i^{\dagger,1} +t(\vartheta_i^{\dagger,1}-\vartheta_i^{\dagger,2})$ with $t\equiv t(\vartheta_i^{\dagger,1},\vartheta_i^{\dagger,2})\in(0,1)$. 
	Since $\vartheta_i^\dagger\equiv(H_i^\dagger,\eta_i^\dagger)\in\Theta_0(\xi)$ implies $-\beta_1(H_i^\dagger)-\alpha(H_0)-\xi/2>-1$, Lemma~\ref{Basic.BilinearForm} and (\ref{UC.Mean.Value.}) yield that the third term of the last quantity of (\ref{Decomposition.Uniform.Convergence}) is dominated by
	\begin{align}
	\max_{i\in\Lambda_{j(r)}}Q_n\left(\widetilde{\mathbf{Y}}_n, \sup_{\vartheta\in B_r(\vartheta_i)}\left|k_{\vartheta}^n-k_{\vartheta_i}^n\right|\right)\leq&
	r\frac{c_2}{2\pi}\max_{i\in\Lambda_{j(r)}}\int_{-\pi}^\pi I_n\left(\lambda,\widetilde{\mathbf{Y}}_n\right)\left|\lambda\right|^{-\beta_1(H_i)-\xi/2}\, \mathrm{d}\lambda \nonumber \\
	\leq& r\frac{c_2}{2\pi}\left(\max_{i\in\Lambda_{j(r)}}Q_{H_0,\xi}(H_i) + \max_{i\in\Lambda_{j(r)}}R_{n,\xi}(H_i)\right), \label{Uniform.Convergence.Inequality.Negligibility}
	\end{align}
	where 
	\begin{align*}
	&Q_{H_0,\xi}(H_i) := \int_{-\pi}^\pi \eta_0^2f_{H_0}(\lambda)\left|\lambda\right|^{-\beta_1(H_i)-\xi/2}\, \mathrm{d}\lambda,\\
	&R_{n,\xi}(H_i) := \left|\int_{-\pi}^\pi I_n\left(\lambda,\widetilde{\mathbf{Y}}_n\right)\left|\lambda\right|^{-\beta_1(H_i)-\xi/2}\, \mathrm{d}\lambda - Q_{H_0,\xi}(H_i)\right|.
	\end{align*}
	Moreover, Lemma~\ref{Lemma.SPD} and $\vartheta_i^\dagger\equiv(H_i^\dagger,\eta_i^\dagger)\in\Theta_0(\xi)$, $i\in\Lambda_{j(r)}$, yield that there exists a constant $c_3\equiv c_3(\xi)>0$, which is independent of $r\in\mathbb{N}$, such that the first term of the last quantity of (\ref{Uniform.Convergence.Inequality.Negligibility}) is dominated by
	\begin{equation*}
	r\frac{c_2}{2\pi}\cdot\max_{i\in\Lambda_{j(r)}}Q_{H_0,\xi}(H_i) \leq r\frac{c_3}{2\pi}\int_{-\pi}^\pi \left|\lambda\right|^{-1+\xi/2}\, \mathrm{d}\lambda =rc_3\pi^{\xi/2-1}.
	\end{equation*}
	As a result, the first term of the last quantity of (\ref{Uniform.Convergence.Inequality.Negligibility}) converges to $0$ as $r\downarrow 0$ irrespectively to the asymptotic parameter $n\in\mathbb{N}$. Moreover, Corollary~\ref{Pointwise.Convergence.DeltaY} yields that for each $r\in(0,\xi/2)$, the second term of the last quantity of (\ref{Uniform.Convergence.Inequality.Negligibility}) converges to $0$ as $n\to\infty$. Therefore, the conclusion follows.
\end{proof}
\section{Proof of Theorem~\ref{Consistency.Estimator}}\label{Section.Proof.Main.Theorem}
Main purpose in this section is to give a proof of Theorem~\ref{Consistency.Estimator}. We prepare notation used in its proof in Section~\ref{Section.Notation.Main.Proof} and several limit theorems of estimation and its score functions are summarized in Section~\ref{Appendix.Proof.Preliminary.Propositions}. A key proposition and Theorem~\ref{Consistency.Estimator} are proven in Section~\ref{Appendix.Proof.Main.Theorem}.
\subsection{Notation of Parameter Space and Estimation Function}\label{Section.Notation.Main.Proof}
Recall $\Theta:=\Theta_H\times\Theta_\eta$ is a compact set of the form $\Theta_H:=[H_-,H_+]\subset(0,1]$, $\Theta_\eta:=[\eta_-,\eta_+]\subset(0,\infty)$ and $\Theta_{\widetilde\vartheta}^n:=\Theta_H\times\Theta_{\widetilde\nu}^n$ with $\Theta_{\widetilde\nu}^n:=[\eta_- \delta_n^{H_+-H_0},\eta_+ \delta_n^{H_- - H_0}]$. 
Following the argument in Velasco and Robinson~\cite{VR}, we divide the parameter space $\Theta_H$ into the following two subsets:
\begin{align*}
&\Theta_{H,1}(\xi):=\left\{H\in\Theta_H:\alpha(H)-\alpha(H_0)\geq -1+\xi\right\}, \\
&\Theta_{H,2}(\xi):=\left\{H\in\Theta_H:\alpha(H)-\alpha(H_0)< -1+\xi\right\},
\end{align*}
where $\alpha(H) :=2H-1$ and $\xi\in(0,1)$. Moreover, we also divide the rescaled parameter space $\Theta_{\widetilde\vartheta}^n$ into the following four subsets:
\begin{equation*}
\Theta_j^n(\xi,L):=\Theta_{H,j}(\xi)\times\Theta_{\widetilde\nu,0}^n(L),\hspace{0.2cm} \Theta_k^n(L):=\Theta_H\times\Theta_{\widetilde\nu,k}^n(L)
\end{equation*}
for $L\in\mathbb{R}_+$ and $j, k\in\{1,2\}$, where
\begin{equation*}
\Theta_{\widetilde\nu,0}^n(L):=\Theta_{\widetilde\nu}^n\cap[1/L,L],\hspace{0.2cm}\Theta_{\widetilde\nu,1}^n(L):=\Theta_{\widetilde\nu}^n\cap(0,1/L),\hspace{0.2cm} \Theta_{\widetilde\nu,2}^n(L):=\Theta_{\widetilde\nu}^n\cap(L,\infty).
\end{equation*}
Denote by $U_{n,0}(\widetilde\vartheta)\equiv U_{n,0}(H,\widetilde\nu)$, defined in Section~\ref{SubSection.Construction.Estimator}, and
\begin{equation*}
U_{\vartheta_0}(\widetilde\vartheta):=\frac{1}{4\pi}\int_{-\pi}^\pi\left\{\log\left(\widetilde\nu^2f_H(\lambda)\right) +\frac{\eta_0^2f_{H_0}(\lambda)}{\widetilde\nu^2f_H(\lambda)}\right\}\, \mathrm{d}\lambda.
\end{equation*}
Note that $\vartheta_0=(H_0,\eta_0)=(H_0,\widetilde\nu_0)$ and for any $\xi\in(0,1)$ and $L\in\mathbb{R}_+$, we can show that $U_{\vartheta_0}(\widetilde\vartheta)$ satisfies the identifiability condition with respect to the parameter $\widetilde\vartheta$ on $\Theta_1^n(\xi,L)$, i.e. for any $\iota>0$,
\begin{equation}\label{Limit.Identifiability.Cond}
\inf_{\widetilde\vartheta\in\Theta_1^n(\xi,L),\|\widetilde\vartheta-\vartheta_0\|\geq\iota}U_{\vartheta_0}(\widetilde\vartheta)>U_{\vartheta_0}(\vartheta_0),
\end{equation}
by using Lemma~\ref{Lemma.SPD} (\ref{SPD.Identifiability.Cond}) and the elementary inequality $\log{x}\leq x-1$ for any $x>0$ that is actually an equality only when $x=1$.
\subsection{Convergence of Estimation and Its Score Functions}\label{Appendix.Proof.Preliminary.Propositions}
\begin{prop}\label{Uniform.Convergence.Contrast}
	Let $\mathcal{K}$ be a compact interval of $(0,\infty)$ and $\vartheta_0$ be an interior point of $\Theta$. Under the conditions $(H.\ref{Assumption.HighFrequency})-(H.\ref{Assumption.mn})$, the following uniform convergences on $\Theta_{H,1}(\xi)\times\mathcal{K}$ hold for any $\xi\in(0,1)$:
	\begin{align*}
	&\sup_{\widetilde\vartheta\in\Theta_{H,1}(\xi)\times\mathcal{K}}\left|U_{n,0}(\widetilde\vartheta) -U_{\vartheta_0}(\widetilde\vartheta)\right|=o_{P}(1) \ \ \mbox{as}\ \ n\to\infty, \\ &\sup_{\widetilde\vartheta\in\Theta_{H,1}(\xi)\times\mathcal{K}}\left|\nabla^2U_{n,0}(\widetilde\vartheta) -\nabla^2U_{\vartheta_0}(\widetilde\vartheta)\right|=o_{P}(1) \ \ \mbox{as $n\to\infty$.}
	\end{align*}
\end{prop}
\begin{proof}
	Let us consider only the first claim because the second one is proven in the similar way. 
	Here we have
	\begin{align}\label{Decomp.Error.EstFunc}
	U_{n,0}(\widetilde\vartheta) -U_{\vartheta_0}(\widetilde\vartheta) 
	=& \frac{1}{4\pi}\int_{-\pi}^\pi \left\{\log{h_{\widetilde\vartheta}^n(\lambda)}-\log{\left({\widetilde\nu}^2f_H(\lambda)\right)}\right\}\, \mathrm{d}\lambda \\ &+\frac{1}{4\pi}\left\{Q_n\left(\widetilde{\mathbf{Y}}_n,1/h_{\widetilde\vartheta}^n\right) - Q_{\vartheta_0}\left(1/(\widetilde\nu^2f_H)\right)\right\}, \nonumber
	\end{align}
	where $Q_{\vartheta_0}$ is defined in (\ref{Def.Q.Limit}). Under the assumption $(H.\ref{Assumption.HighFrequency})$ and $(H.\ref{Assumption.mn})$, we obtain 
	\begin{equation*}
	\lim_{n\to\infty}\sup_{\lambda\in[-\pi,\pi]\backslash\{0\},\widetilde\vartheta\in\Theta_{H,1}(\xi)\times\mathcal{K}}\left|\log{h_{\widetilde\vartheta}^n(\lambda)}-\log{\left({\widetilde\nu}^2f_H(\lambda)\right)}\right|=0
	\end{equation*}
	so that the first term of rhs of (\ref{Decomp.Error.EstFunc}) is negligible as $n\to\infty$. Note that we have
	\begin{equation*}
	\sup_{n\in\mathbb{N},\lambda\in[-\pi, \pi]\backslash\{0\}}\left\{|\lambda|^{1-2H}\left|1/h_{\widetilde\vartheta}^n(\lambda)\right|+|\lambda|^{1-2H+\iota}\left\|\nabla\left(1/h_{\widetilde\vartheta}^n\right)(\lambda)\right\|_{\mathbb{R}^2}\right\}<\infty
	\end{equation*}
	for any $\iota>0$. Let us fix sufficiently small $\iota>0$. Then we can show that the second term of rhs of (\ref{Decomp.Error.EstFunc}) is also negligible as $n\to\infty$ by using Proposition~\ref{Uniform.Convergence.QuadraticForm.DeltaY} in the case of the function $1/h_{\widetilde\vartheta}^n$, $\beta_0(H):=1-2H$ and $\beta_1(H):=1-2H+\iota$. Therefore, the conclusion follows.
\end{proof}
\begin{prop}\label{Scaled.Score.Negligibility}
	Let $\vartheta_0$ be an interior point of $\Theta$. Under the conditions $(H.\ref{Assumption.HighFrequency})-(H.\ref{Assumption.mn})$, the score function $\nabla{U}_{n,0}(\widetilde\vartheta)$ has the following asymptotic behavior at the point $\vartheta_0$: there exists a constant $\psi>0$ such that
	\begin{equation*}
	\nabla{U}_{n,0}(\vartheta_0) = o_{P}(\delta_n^\psi)\ \ \mbox{as}\ \ n\to\infty. 
	\end{equation*}
\end{prop}
\begin{proof}
	At first, we decompose $u^{\mathrm{T}}\nabla{U}_{n,0}(\vartheta_0)$ with $u\in\mathbb{R}^2$ into the following three parts: 
	\begin{equation*}
	u^{\mathrm{T}}\nabla{U}_{n,0}(\vartheta_0)=\frac{1}{4\pi}\left(A_1^n+A_2^n+A_3^n\right),
	\end{equation*}
	where
	\begin{align*}
	&A_1^n :=\int_{-\pi}^\pi u^{\mathrm{T}}\nabla\log{h_{\vartheta_0}^n(\lambda)}\, \mathrm{d}\lambda +E\left[Q_n\left(\widetilde{\mathbf{G}}_n,k_{\vartheta_0}^n\right)\right], \\
	&A_2^n :=Q_n\left(\widetilde{\mathbf{Y}}_n,k_{\vartheta_0}^n\right) -Q_n\left(\widetilde{\mathbf{G}}_n,k_{\vartheta_0}^n\right),\ \ A_3^n :=Q_n\left(\widetilde{\mathbf{G}}_n,k_{\vartheta_0}^n\right) - E\left[Q_n\left(\widetilde{\mathbf{G}}_n,k_{\vartheta_0}^n\right)\right]
	\end{align*}
	and $k_{\vartheta_0}^n:=u^{\mathrm{T}}\nabla\left(1/h_{\vartheta_0}^n\right)$. 
	Note that Lemma~\ref{Lemma.SPD}~(\ref{SPD.Asymptotic}) yields that for any $\iota>0$,
	\begin{equation*}
	\sup_{n\in\mathbb{N},\lambda\in[-\pi, \pi]\backslash\{0\}}|\lambda|^{1-2H_0+\iota}\left|k_{\vartheta_0}^n(\lambda)\right|<\infty.
	\end{equation*}
	Let us fix sufficiently small $\iota>0$. Then we can show that $A_3^n=O_P(1/\sqrt{n})$ as $n\to\infty$ and $A_2^n=o_P(\delta_n^\psi)$ as $n\to\infty$ for a certain constant $\psi>0$ by using Corollary~\ref{Convergence.Variance.Gaussian.Quadratic} and Proposition~\ref{Convergence.QuadraticForm.Remainder} in the case of the function $k_{\vartheta_0}^n$ and $\beta_0(H):=1-2H+\iota$ respectively. As a result, it suffices to prove that $A_1^n=O(1/\sqrt{n})$ as $n\to\infty$. Our proof is similar to that of Theorem 2 in Fox and Taqqu~\cite{FT86}. At first, we obtain
	\begin{align}
	E\left[Q_n\left(\widetilde{\mathbf{G}}_n,k_{\vartheta_0}^n\right)\right] =& \frac{1}{2\pi n}\sum_{i,j=1}^n \widehat{k_{\vartheta_0}^n}(j-i)\widehat{h_{\vartheta_0}^n}(j-i) =\frac{1}{2\pi}\sum_{|\tau|<n}\left(1-\frac{|\tau|}{n}\right)\widehat{k_{\vartheta_0}^n}(\tau)\widehat{h_{\vartheta_0}^n}(\tau) \nonumber \\
	=& \frac{1}{2\pi}\sum_{|\tau|<n}\widehat{k_{\vartheta_0}^n}(\tau)\widehat{h_{\vartheta_0}^n}(\tau) -\frac{1}{2\pi n}\sum_{|\tau|<n}|\tau|\widehat{k_{\vartheta_0}^n}(\tau)\widehat{h_{\vartheta_0}^n}(\tau). \label{Sqrt.n.Expectation.Negligibility.1}
	\end{align}
	Since the functions $k_{\vartheta_0}^n, h_{\vartheta_0}^n\in L^1[-\pi,\pi]$ are $2\pi$-periodic, it is well-known that $\widehat{k_{\vartheta_0}^n}(\tau)\widehat{h_{\vartheta_0}^n}(\tau)$ is the $\tau$th Fourier coefficient of the convolution $k_{\vartheta_0}^n\ast h_{\vartheta_0}^n$ defined by
	\begin{equation*}
	\left(k_{\vartheta_0}^n\ast h_{\vartheta_0}^n\right)(\lambda) :=\int_{-\pi}^\pi k_{\vartheta_0}^n(x) h_{\vartheta_0}^n(\lambda-x)\, \mathrm{d}x =\int_{-\pi}^\pi u^{\mathrm{T}}\nabla\left(1/h_{\vartheta_0}^n(x)\right) h_{\vartheta_0}^n(x-\lambda)\, \mathrm{d}x,
	\end{equation*}
	where we use the property that $h_{\vartheta_0}^n$ is an even function in the above equality. Moreover, we have
	\begin{equation}\label{Asymptotic.Decay.Convolution}
	\sup_{n\in\mathbb{N}}\left|\widehat{k_{\vartheta_0}^n}(\tau)\widehat{h_{\vartheta_0}^n}(\tau)\right|=O\left(|\tau|^{-2+\iota}\right)\ \ \mbox{as $|\tau|\to\infty$}
	\end{equation}
	for any $\iota>0$ from Lemma~\ref{Extension1.FT86} and Lemma~\ref{Extension2.FT86}. As a result, for each $n\in\mathbb{N}$, $(k_{\vartheta_0}^n\ast h_{\vartheta_0}^n)(\lambda)$ is expanded as the following Fourier series for a.e. $\lambda\in[-\pi,\pi]$:
	\begin{equation}\label{Conv.Fourier.Series}
	\left(k_{\vartheta_0}^n\ast h_{\vartheta_0}^n\right)(\lambda)=\frac{1}{2\pi}\sum_{\tau\in\mathbb{Z}}\widehat{k_{\vartheta_0}^n}(\tau)\widehat{h_{\vartheta_0}^n}(\tau)e^{\sqrt{-1}\tau\lambda}.
	\end{equation}
	Note that the continuity of the function $k_{\vartheta_0}^n\ast h_{\vartheta_0}^n$ on $[-\pi,\pi]$ implies the Fourier series expansion (\ref{Conv.Fourier.Series}) also holds for all $\lambda\in[-\pi,\pi]$. In particular, we obtain 
	\begin{align}\label{Fourier.Exp.Convolution}
	\frac{1}{2\pi}\sum_{\tau\in\mathbb{Z}}\widehat{k_{\vartheta_0}^n}(\tau)\widehat{h_{\vartheta_0}^n}(\tau) =&\left(k_{\vartheta_0}^n\ast h_{\vartheta_0}^n\right)(0) \\
	=&\int_{-\pi}^\pi u^{\mathrm{T}}\nabla\left(1/h_{\vartheta_0}^n(\lambda)\right) h_{\vartheta_0}^n(\lambda)\, \mathrm{d}\lambda =-\int_{-\pi}^\pi u^{\mathrm{T}}\nabla\log{h_{\vartheta_0}^n(\lambda)}\, \mathrm{d}\lambda. \nonumber
	\end{align}
	From the equalities (\ref{Sqrt.n.Expectation.Negligibility.1}) and (\ref{Fourier.Exp.Convolution}), we obtain
	\begin{align}
		\sqrt{n}A_3^n =& \sqrt{n}\left(\int_{-\pi}^\pi u^{\mathrm{T}}\nabla\log{h_{\vartheta_0}^n(\lambda)}\, \mathrm{d}\lambda +\frac{1}{2\pi}\sum_{|\tau|<n}\widehat{k_{\vartheta_0}^n}(\tau)\widehat{h_{\vartheta_0}^n}(\tau)\right) -\frac{1}{2\pi\sqrt{n}}\sum_{|\tau|<n}|\tau|\widehat{k_{\vartheta_0}^n}(\tau)\widehat{h_{\vartheta_0}^n}(\tau)  \nonumber \\
		=& -\frac{\sqrt{n}}{2\pi}\sum_{|\tau|\geq n}\widehat{k_{\vartheta_0}^n}(\tau)\widehat{h_{\vartheta_0}^n}(\tau)   -\frac{1}{2\pi\sqrt{n}}\sum_{|\tau|<n}|\tau|\widehat{k_{\vartheta_0}^n}(\tau)\widehat{h_{\vartheta_0}^n}(\tau) . \label{Sqrt.n.Expectation.Negligibility.2}
	\end{align}
	Then we can show that both terms of (\ref{Sqrt.n.Expectation.Negligibility.2}) are negligible as $n\to\infty$ in the similar way to the proof of Theorem 2 in Fox and Taqqu~\cite{FT86} by using the relation (\ref{Asymptotic.Decay.Convolution}). Therefore, the conclusion follows.
\end{proof}
\subsection{Proof of Theorem~\ref{Consistency.Estimator}}\label{Appendix.Proof.Main.Theorem}
Before proving Theorem~\ref{Consistency.Estimator}, we show the following result.
\begin{prop}\label{Consistency.ThetaTilde}
	Let $\vartheta_0$ be an interior point of $\Theta$. Under the conditions $(H.\ref{Assumption.HighFrequency})-(H.\ref{Assumption.mn})$,
	\begin{equation*}
	\widetilde\vartheta_n:=(\widehat{H}_n,\delta_n^{-H_0}\widehat\nu_n)\rightarrow(H_0,\eta_0)\ \ \mbox{as}\ \ n\to\infty\ \ \mbox{in probability}.
	\end{equation*}
\end{prop}
Following the argument of Velasco and Robinson~\cite{VR}, we divide the proof of Proposition~\ref{Consistency.ThetaTilde} into the following two steps.
\begin{step}\label{Step1}
	Let $\xi\in(0,1)$ and $L>0$. Define a random variable $\widetilde\vartheta_n^{(1)}(\xi,L)$ by
	\begin{equation*}
	\widetilde\vartheta_n^{(1)}(\xi,L):=\argmin_{\widetilde\vartheta\in\Theta_1^n(\xi,L)}U_{n,0}(\widetilde\vartheta).
	\end{equation*}
	Then, for each $\xi\in(0,1)$ and $L>0$, $\widetilde\vartheta_n^{(1)}(\xi,L)\to\vartheta_0$ in probability as $n\to\infty$. 
\end{step}
\begin{proof}[Proof of Step~\ref{Step1}]
	Note that for any $\xi\in(0,1)$ and $L>0$, the parameter space $\Theta_1^n(\xi,L)$ no longer depends on the asymptotic parameter $n\in\mathbb{N}$ if $n$ is sufficiently large. As a result, the conclusion immediately follows from Proposition~\ref{Uniform.Convergence.Contrast} and the identifiability condition of the limit function $U_{\vartheta_0}(\widetilde\vartheta)$ on $\Theta_1^n(\xi,L)$, see (\ref{Limit.Identifiability.Cond}).
\end{proof}
\begin{step}\label{Step2}
	There exist constants $\xi_0\in(0,1)$ and $L_0>0$ such that for any $\xi\in(0,\xi_0)$ and $L\geq L_0$, $\widetilde\vartheta_n - \widetilde\vartheta_n^{(1)}(\xi,L)\rightarrow 0$ in probability as $n\to\infty$. 
\end{step}
\begin{proof}[Proof of Step~\ref{Step2}]
	Without loss of generality, we can assume $\Theta_2^n(\xi,L)$ and $\Theta_k^n(L)$, $k=1,2$, are non-empty sets for each $\xi\in(0,1)$ and $L\in(0,\infty)$. 
	Note that for any $\iota>0$, we have
	\begin{equation}\label{Step2.1}
	P\left[\|\widetilde\vartheta_n-\widetilde\vartheta_n^{(1)}(\xi,L)\|_{\mathbb{R}^2}>\iota\right]\leq P\left[\inf_{\widetilde\vartheta\in\Theta_{\widetilde\vartheta}^n\backslash\Theta_1^n(\xi,L)}U_{n,0}(\widetilde\vartheta)\leq \inf_{\widetilde\vartheta\in\Theta_1^n(\xi,L)}U_{n,0}(\widetilde\vartheta)\right].
	\end{equation}
	Since we can show that
	\begin{align*}
	\left\{\inf_{\widetilde\vartheta\in\Theta_{\widetilde\vartheta}^n\backslash\Theta_1^n(\xi,L)}U_{n,0}(\widetilde\vartheta)\leq U_{n,0}(\widetilde\vartheta_n^{(1)})\right\} \subset\left\{\inf_{\widetilde\vartheta\in\Theta_2^n(\xi,L)}U_{n,0}(\widetilde\vartheta)\leq U_{n,0}(\widetilde\vartheta_n^{(1)})\right\} \cup\bigcup_{k=1}^2\left\{\inf_{\widetilde\vartheta\in\Theta_k^n(L)}U_{n,0}(\widetilde\vartheta)\leq U_{n,0}(\widetilde\vartheta_n^{(1)})\right\},
	\end{align*}
	the rhs of (\ref{Step2.1}) is dominated by
	\begin{align}
	&P\left[\inf_{\widetilde\vartheta\in\Theta_2^n(\xi,L)}U_{n,0}(\widetilde\vartheta) \leq U_{n,0}(\widetilde\vartheta_n^{(1)})\right] +\sum_{k=1}^2P\left[\inf_{\widetilde\vartheta\in\Theta_k^n(L)}U_{n,0}(\widetilde\vartheta) \leq U_{n,0}(\widetilde\vartheta_n^{(1)})\right] \nonumber \\
	\leq& P\left[\inf_{\widetilde\vartheta\in\Theta_2^n(\xi,L)}U_{n,0}(\widetilde\vartheta) \leq U_{\vartheta_0}(\vartheta_0)+\acute\iota\right] +\sum_{k=1}^2P\left[\inf_{\widetilde\vartheta\in\Theta_k^n(L)}U_{n,0}(\widetilde\vartheta) \leq U_{\vartheta_0}(\vartheta_0)+\acute\iota\right] +3P\left[\left|U_{n,0}(\widetilde\vartheta_n^{(1)}) - U_{\vartheta_0}(\vartheta_0)\right| \geq\acute\iota\right] \label{Consistency.FourTerms}
	\end{align}
	for any $\acute\iota>0$. Then the first term of (\ref{Consistency.FourTerms}) is negligible as $n\to\infty$ from Proposition~\ref{Uniform.Convergence.Contrast} and Step~\ref{Step1}. 
	Therefore, it suffices to prove that the other terms are negligible as $n\to\infty$ if we take sufficiently large $L>0$ and small $\xi\in(0,1)$ respectively. We divide the proof into the following three lemmas.
	\begin{lem}\label{Main.Lemma.Theta2}
		For any $\acute\iota, L>0$, there exists a constant $\xi_0\equiv\xi_0(L,\acute\iota,\vartheta_0)\in(0,1)$ such that for any $\xi\in(0,\xi_0)$, 
		\begin{equation*}
		\lim_{n\to\infty} P\left[\inf_{\widetilde\vartheta\in\Theta_2^n(\xi,L)}U_{n,0}(\widetilde\vartheta)\leq U_{\vartheta_0}(\vartheta_0) + \acute\iota \right]=0.
		\end{equation*}
	\end{lem}
	\begin{proof}
		Fix $\acute\iota, L>0$. Note that Lemma~\ref{Lemma.SPD} and the assumption (H.\ref{Assumption.mn}) yield that there exist constants $c_1, c_2>0$, which are independent of $(H,\widetilde\nu,\lambda)\in\Theta_H\times[1/L,L]\times[-\pi,\pi]\backslash\{0\}$ and $n\in\mathbb{N}$, such that
		\begin{equation*}
		c_1|\lambda|^{-\alpha(H)}<h_{H,\widetilde\nu}^n(\lambda)<c_2|\lambda|^{-\alpha(H)}
		\end{equation*}
		holds for any $(H,\widetilde\nu,\lambda)\in\Theta_H\times[1/L,L]\times[-\pi,\pi]\backslash\{0\}$ and $n\in\mathbb{N}$, where $h_{H,\widetilde\nu}^n$ is given in (\ref{Rescaled.SPD}). 
		Then for any $\widetilde\vartheta_1=(H_1,\widetilde\nu_1)$ and $\widetilde\vartheta_2=(H_2,\widetilde\nu_2)$ satisfying $\alpha(H_1)\geq\alpha(H_2)$, we obtain
		\begin{align}
		U_{n,0}(\widetilde\vartheta_2)\geq&\frac{1}{4\pi}\int_{-\pi}^\pi\log\left(c_1|\lambda|^{-\alpha(H_2)}\right)\, \mathrm{d}\lambda + \frac{1}{4\pi c_2}\int_{-\pi}^\pi I_n\left(\lambda,\widetilde{\mathbf{Y}}_n\right)|\lambda|^{\alpha(H_2)}\, \mathrm{d}\lambda \nonumber \\
		\geq& \frac{\log{c_1}}{2} -\frac{\alpha(H_1)}{2}\left(\log\pi - 1\right) + \frac{1}{4\pi^{1+2H_+} c_2}\int_{-\pi}^\pi I_n\left(\lambda,\widetilde{\mathbf{Y}}_n\right)|\lambda|^{\alpha(H_1)}\, \mathrm{d}\lambda, \label{Lower.Estimate1}
		\end{align}
		where we use the elementary inequality $|\lambda|^{\alpha(H_1)-\alpha(H_2)}\leq\pi^{2H_+}$ for any $\lambda\in[-\pi,\pi]$ in the second inequality. 
		Let $H_1^\dagger\equiv H_1^\dagger(\xi):=H_0+(\xi-1)/2\in\partial\Theta_1^n(\xi,L)$, where $\partial\Theta_1^n(\xi,L)$ denotes the boundary of the set $\Theta_1^n(\xi,L)$. Since the relation $\alpha(H_1^\dagger)\geq\alpha(H_2)$ holds for any $\widetilde\vartheta_2=(H_2,\widetilde\nu_2)\in\Theta_2^n(\xi,L)$, we can obtain the following inequality from (\ref{Lower.Estimate1}):
		\begin{align}\label{Lemma.Theta2.temp2}
		\inf_{\widetilde\vartheta_2\in\Theta_2^n(\xi,L)}U_{n,0}(\widetilde\vartheta_2) \geq\frac{\log{c_1}}{2} -\frac{\alpha(H_0)-1+\xi}{2}\left(\log\pi - 1\right) + \frac{1}{4\pi^{1+2H_+} c_2}\int_{-\pi}^\pi I_n\left(\lambda,\widetilde{\mathbf{Y}}_n\right)|\lambda|^{\alpha(H_0)-1+\xi}\, \mathrm{d}\lambda.
		\end{align}
		Moreover, Corollary~\ref{Pointwise.Convergence.DeltaY} in the case of $k_\vartheta(\lambda):=|\lambda|^{\alpha(H)-1+\xi}$ and $\beta_0(H):=-\alpha(H)+1-\xi$ yields that the third term of the rhs of (\ref{Lemma.Theta2.temp2}) converges to 
		\begin{equation*}
		\frac{1}{4\pi^{1+2H_+} c_2}\int_{-\pi}^\pi \eta_0^2f_{H_0}(\lambda)|\lambda|^{\alpha(H_0)-1+\xi}\, \mathrm{d}\lambda
		\end{equation*}
		in probability as $n\to\infty$ and we obtain the following inequality:
		\begin{equation}\label{Lemma.Theta2.temp3}
		\frac{1}{4\pi^{1+2H_+} c_2}\int_{-\pi}^\pi \eta_0^2f_{H_0}(\lambda)|\lambda|^{\alpha(H_0)-1+\xi}\, \mathrm{d}\lambda\geq \frac{c_1}{4\pi^{1+2H_+} c_2}\int_{-\pi}^\pi |\lambda|^{-1+\xi}\, \mathrm{d}\lambda \geq\frac{c_1}{2\pi^{1+2H_+} c_2}\cdot\frac{\pi^\xi}{\xi}.
		\end{equation}
		Note that we can make the rhs of (\ref{Lemma.Theta2.temp3}) arbitrarily large if we take $\xi>0$ sufficiently small. In particular, if we take $\xi\equiv\xi(\acute\iota)\in(0,1)$ such that
		\begin{equation*}
		\frac{c_1}{2\pi^{1+2H_+} c_2}\cdot\frac{\pi^\xi}{\xi} > U_{\vartheta_0}(\vartheta_0) + \acute\iota
		\end{equation*}
		holds, then we can obtain the following convergence from (\ref{Lemma.Theta2.temp2}), (\ref{Lemma.Theta2.temp3}) and Corollary~\ref{Pointwise.Convergence.DeltaY} again:
		\begin{equation}\label{Step2.1.Last}
		\lim_{n\to\infty} P\left[\inf_{\widetilde\vartheta\in\Theta_2^n(\xi,L)}U_{n,0}(\widetilde\vartheta)> U_{\vartheta_0}(\vartheta_0) + \acute\iota \right]=1.
		\end{equation}
		Therefore, the conclusion follows from (\ref{Step2.1.Last}).
	\end{proof}
	\begin{lem}
		For any $\acute\iota>0$, there exists a constant $L_0\equiv L_0(\acute\iota,\vartheta_0)>0$ such that for any $L\geq L_0$,
		\begin{equation*}
		\lim_{n\to\infty}P\left[\inf_{\widetilde\vartheta\in\Theta_2^n(L)}U_{n,0}(\widetilde\vartheta)\leq U_{\vartheta_0}(\vartheta_0)+\acute\iota \right]=0.
		\end{equation*}
	\end{lem}
	\begin{proof}
		Fix $\acute\iota>0$. Since the inequality
		\begin{equation}\label{Step2.2}
		U_{n,0}(\widetilde\vartheta)\geq \frac{1}{4\pi}\int_{-\pi}^\pi\log\left(\widetilde\nu^2f_H(\lambda)\right)\, \mathrm{d}\lambda 
		\geq \log L +\min_{H\in\Theta_H}\left\{\frac{1}{4\pi}\int_{-\pi}^\pi\log{f_H(\lambda)}\, \mathrm{d}\lambda\right\}.
		\end{equation}
		holds for each $\widetilde\vartheta=(H,\widetilde\nu)\in\Theta_H\times(L,\infty)$ and we can make the rhs of (\ref{Step2.2}) arbitrarily large if we take $L>0$ sufficiently large, the conclusion immediately follows.
	\end{proof}
	\begin{lem}
		For any $\acute\iota>0$, there exists a constant $L_0\equiv L_0(\acute\iota,\vartheta_0)>0$ such that for any $L\geq L_0$,
		\begin{equation*}
		\lim_{n\to\infty}P\left[\inf_{\widetilde\vartheta\in\Theta_1^n(L)}U_{n,0}(\widetilde\vartheta)\leq U_{\vartheta_0}(\vartheta_0)+\acute\iota \right]=0.
		\end{equation*}
	\end{lem}
	\begin{proof}
		Fix $\acute\iota>0$. Without loss of generality, we can assume $\Theta_1^n(L)=\Theta_H\times[\eta_- \delta_n^{H_+-H_0},1/L)$. At first, we can show that for any $\widetilde\vartheta=(H,\widetilde\nu)\in\Theta_1^n(L)$,
		\begin{align}
		U_{n,0}(\widetilde\vartheta)\geq& \frac{1}{4\pi}\int_{-\pi}^\pi\log\left(\widetilde\nu^2f_H(\lambda)\right)\, \mathrm{d}\lambda + \frac{1}{4\pi}\int_{-\pi}^\pi\frac{I_n\left(\lambda,\widetilde{\mathbf{Y}}_n\right)}{\widetilde\nu^2f_H(\lambda)+2(m_n\delta_n^{2H_0})^{-1}\ell(\lambda)}\, \mathrm{d}\lambda \nonumber \\
		\geq& \frac{1}{\widetilde\nu^2}\left(\widetilde\nu^2\log\widetilde\nu +\frac{1}{4\pi}\int_{-\pi}^\pi\frac{I_n\left(\lambda,\widetilde{\mathbf{Y}}_n\right)}{f_H(\lambda)+2(\widetilde\nu^2 m_n\delta_n^{2H_0})^{-1}\ell(\lambda)}\, \mathrm{d}\lambda\right)  + C_1 \nonumber \\
		\geq& L^2\left(\widetilde\nu^2\log\widetilde\nu +\frac{1}{4\pi}\int_{-\pi}^\pi\frac{I_n\left(\lambda,\widetilde{\mathbf{Y}}_n\right)}{f_H(\lambda)+2(\eta_-^2 m_n\delta_n^{2H_+})^{-1}\ell(\lambda)}\, \mathrm{d}\lambda\right) +C_1, \label{Main.Lemma2.temp1}
		\end{align}
		where
		\begin{equation*}
		C_1:=\min_{H\in\Theta_H}\left\{\frac{1}{4\pi}\int_{-\pi}^\pi\log{f_H(\lambda)}\, \mathrm{d}\lambda\right\}.
		\end{equation*}
		Then it suffices to prove that there exist a constant $r>0$ such that
		\begin{equation}\label{Main.Lemma2.temp2}
		\lim_{n\to\infty}P\left[\inf_{H\in\Theta_H}\left\{\int_{-\pi}^\pi\frac{I_n\left(\lambda,\widetilde{\mathbf{Y}}_n\right)}{f_H(\lambda)+2(\eta_-^2 m_n\delta_n^{2H_+})^{-1}\ell(\lambda)}\, \mathrm{d}\lambda\right\}>r\right]=1.
		\end{equation}
		Indeed, if we take sufficiently large $L>0$, the inequality (\ref{Main.Lemma2.temp1}) and (\ref{Main.Lemma2.temp2}) imply 
		\begin{equation*}
		\lim_{n\to\infty}P\left[\inf_{\widetilde\vartheta\in\Theta_1^n(L)}U_{n,0}(\widetilde\vartheta)> U_{\vartheta_0}(\vartheta_0)+\acute\iota \right]=1
		\end{equation*}
		in the similar way to the proof of Lemma~\ref{Main.Lemma.Theta2}. Therefore, the conclusion immediately follows. 
		Since Lemma~\ref{Lemma.SPD} and the assumption (H.\ref{Assumption.mn}) yield that there exists a constant $c>0$ such that for any $n\in\mathbb{N}$, $\lambda\in[-\pi,\pi]\backslash\{0\}$ and $H\in\Theta_H$, 
		\begin{equation*}
		f_H(\lambda)+\frac{2}{\eta_-^2 m_n\delta_n^{2H_+}}\ell(\lambda)<c|\lambda|^{-\alpha(H)},
		\end{equation*}
		the similar argument in the proof of Lemma~\ref{Main.Lemma.Theta2} implies that it suffices to prove that there exist constants $r_0>0$ and $\xi\in(0,1)$ such that
		\begin{equation}\label{Main.Lemma2.temp3}
		\lim_{n\to\infty}P\left[\inf_{H\in\Theta_{H,1}(\xi)}\left\{\int_{-\pi}^\pi I_n\left(\lambda,\widetilde{\mathbf{Y}}_n\right)|\lambda|^{\alpha(H)}\, \mathrm{d}\lambda\right\}>r_0\right]=1
		\end{equation}
		instead of (\ref{Main.Lemma2.temp2}). In the rest of this proof, we will show (\ref{Main.Lemma2.temp3}). Since Proposition~\ref{Uniform.Convergence.QuadraticForm.DeltaY} in the case of the function $k_\vartheta(\lambda):=|\lambda|^{\alpha(H)}$ and $\beta_0(H):=-\alpha(H)$ yields that 
		\begin{equation*}
		\inf_{H\in\Theta_{H,1}(\xi)}\left\{\int_{-\pi}^\pi I_n\left(\lambda,\widetilde{\mathbf{Y}}_n\right)k_\vartheta(\lambda)\, \mathrm{d}\lambda\right\} =\inf_{H\in\Theta_{H,1}(\xi)}\left\{\int_{-\pi}^\pi\eta_0^2f_{H_0}(\lambda)k_\vartheta(\lambda)\, \mathrm{d}\lambda\right\} +o_P(1)
		\end{equation*}
		as $n\to\infty$ and we can take sufficiently small $r_0>0$ satisfying
		\begin{equation*}
		\inf_{H\in\Theta_{H,1}(\xi)}\left\{\int_{-\pi}^\pi\eta_0^2f_{H_0}(\lambda)k_\vartheta(\lambda)\, \mathrm{d}\lambda\right\}>r_0,
		\end{equation*}
		the convergence (\ref{Main.Lemma2.temp3}) follows. Therefore, we finish the proof.
	\end{proof}
	As a result, the conclusion of Proposition~\ref{Consistency.ThetaTilde} follows from (\ref{Consistency.FourTerms}) and the above three Lemmas.
\end{proof}
In the rest of this section, we prove the consistency of the estimator $\widehat\vartheta_n=(\widehat{H}_n,\widehat{\eta}_n)$ as $n\to\infty$.
\begin{proof}[Proof of Theorem~\ref{Consistency.Estimator}]
	Note that the following equality holds:
	\begin{equation*}
	\log\widehat\eta_n - \log\eta_0 = \log\widetilde\eta_n - \log\eta_0 -(\log\delta_n)(\widehat{H}_n-H_0).
	\end{equation*}
	Therefore, it suffices to prove that $\widehat{H}_n-H_0=o_{P}(|\log\delta_n|^{-1})$ as $n\to\infty$ from the above equality and the delta method. 
	In the rest of this proof, we attempt to prove the following convergence:
	\begin{equation}\label{Goal.Proof.Main.Theorem}
	(\log\delta_n)\left(\widetilde\vartheta_n-\vartheta_0\right)\stackrel{n\to\infty}{\longrightarrow}0
	\end{equation}
	in probability, where $\widetilde\vartheta_n:=(\widehat{H}_n,\delta_n^{-H_0}\widehat\nu_n)$. At first, Taylor's theorem yields that
	\begin{equation}\label{Taylor.Score.Estimation.Error}
	\nabla{U}_{n,0}(\widetilde\vartheta_n) - \nabla{U}_{n,0}(\vartheta_0) = \int_0^1\nabla^2{U}_{n,0}\left(\vartheta_0+u\left(\widetilde\vartheta_n-\vartheta_0\right)\right)\, \mathrm{d}u\cdot \left(\widetilde\vartheta_n-\vartheta_0\right).
	\end{equation}
	Here $\nabla{U}_{n,0}(\widetilde\vartheta_n)=o_{P}(\delta_n^{\psi_1})$ as $n\to\infty$ for any $\psi_1>0$ because $\widetilde\vartheta_n:=(\widehat{H}_n,\delta_n^{-H_0}\widehat\nu_n)$ is a minimizer of $U_{n,0}(\widetilde\vartheta)$ and Proposition~\ref{Consistency.ThetaTilde} yields that $\widetilde\vartheta_n$ converges to the interior point $\vartheta_0$ as $n\to\infty$. 
	Moreover, Proposition~\ref{Scaled.Score.Negligibility} yields that $\nabla{U}_{n,0}(\vartheta_0)=o_{P}(\delta_n^{\psi_2})$ as $n\to\infty$ for a certain constant $\psi_2>0$. 
	As a result, the following convergence holds from (\ref{Taylor.Score.Estimation.Error}):
	\begin{equation}\label{Proof.Theorem.Temp1}
	\int_0^1\nabla^2{U}_{n,0}\left(\vartheta_0+u\left(\widetilde\vartheta_n-\vartheta_0\right)\right)\, \mathrm{d}u\cdot (\log\delta_n)\left(\widetilde\vartheta_n-\vartheta_0\right) = o_{P}(1)\ \ \mbox{as}\ \ n\to\infty.
	\end{equation} 
	Since $\nabla^2{U}_{\vartheta_0}\left(\vartheta_0\right)$ is invertible, the convergence (\ref{Goal.Proof.Main.Theorem}) follows from (\ref{Proof.Theorem.Temp1}), Proposition~\ref{Uniform.Convergence.Contrast}, Proposition~\ref{Consistency.ThetaTilde} and Slutsky's theorem.
\end{proof}
\section{Proof of Theorem~2.1}\label{Section.LogRV.Stable.Conv}
In this appendix, we give a proof of Theorem~2.1. Actually, we will show the following limit theorem that is a stronger version of Theorem~2.1.
\begin{thm}\label{General.logRV.CLT}
	Under the same assumption in Theorem~2.1, a sequence of c\`adl\`ag processes
	\begin{equation*}
		\sqrt{m_n}\left(\log\hat\sigma_\cdot^2-\log\int_{\cdot\delta_n}^{(\cdot+1)\delta_n} \sigma_u^2\, \mathrm{d}u\right)
	\end{equation*}
	converges in law to a continuous Gaussian process $G=\{G_s\}_{s\in[0,\infty)}$ given by $G_s:=\sqrt{2}(\acute{B}_{s+1}-\acute{B}_{s})$, $s\in[0,\infty)$, as $n\to\infty$, where $\acute{B}$ is a standard Brownian motion independent of $\mathcal{F}$.
\end{thm}
We recall the martingale functional central limit theorem in Section~\ref{Summary.Stable.Conv}, a preliminary result used in the proof of Theorem~\ref{General.logRV.CLT} is summarized in Section~\ref{Sec.Preliminary.Stable.Conv} and we prove Theorem~\ref{General.logRV.CLT} in Section~\ref{Sec.Proof.Stable.Conv.LogRV}.
\subsection{Summary of Martingale Functional Central Limit Theorem}\label{Summary.Stable.Conv}
In this subsection, we recall the well-known martingale functional central limit theorem and give its concise proof in the case where local martingales are continuous.
\begin{thm}[Martingale Functional Central Limit Theorem]\label{Martingale-FCLT}
	Let $(\Omega,\mathcal{F},P)$ be a probability space, $\mathbb{F}^n=\{\mathcal{F}_s^n\}_{s\in[0,\infty)}$ be a sequence of filtrations on $(\Omega,\mathcal{F})$ satisfying the usual conditions and $\{Z^n\}_{n\in\mathbb{N}}$ be a sequence of continuous $\mathbb{F}^n$-local martingales. If there exists a continuous function $v:[0,\infty)\to[0,\infty)$ such that for any $s\in[0,\infty)$,
	\begin{equation}\label{Assump.Martingale.FCLT}
		\langle Z\rangle_s\stackrel{n\to\infty}{\longrightarrow}v_s\ \ \mbox{in probability},
	\end{equation}
	then a sequence of the $C_{[0,\infty)}$-valued random variables $\{Z^n\}_{n\in\mathbb{N}}$ converges in law to the time-changed Brownian motion $\acute{B}_v$, where $\acute{B}$ is a standard Brownian motion.
\end{thm}
\begin{proof}
	At first, Dambis-Dubins-Schwarz's theorem, see Karatzas and Shreve~\cite{KS}, Theorem 3.4.6, yields that there exists a sequence of standard Brownian motions $\{B^n\}_{n\in\mathbb{N}}$ such that for each $n\in\mathbb{N}$,
	\begin{equation*}
		Z^n=B^n_{\langle Z^n\rangle}\ \ \mbox{$P$-a.s.}.
	\end{equation*}
	Note that, since $\langle Z\rangle$ is non-negative and non-decreasing and $v$ is continuous, the assumption (\ref{Assump.Martingale.FCLT}) implies that for any $s\in[0,\infty)$,
	\begin{equation}\label{Assump2.Martingale.FCLT}
		\sup_{0\leq u\leq s}|\langle Z\rangle_u-v_u|=o_P(1)\ \ \mbox{as}\  \ n\to\infty
	\end{equation}
	by using Theorem V\hspace{-.1em}I.2.15 in Jacod and Shiryaev~\cite{JS}. Moreover, (\ref{Assump2.Martingale.FCLT}) and the Slutsky's theorem yield that $C_{[0,\infty)}^2$-valued process $(B^n,\langle Z^n\rangle)$ converges in law to $(\acute{B},v)$ as $n\to\infty$, where $\acute{B}$ is a standard Brownian motion. Therefore, the conclusion follows from the above convergence in law and the continuous mapping theorem since $\psi:C_{[0,\infty)}^2\to C_{[0,\infty)}$ defined by $\psi(z,v):=z\circ v$ is continuous in the similar argument to Billingsley~\cite{Billingsley}, p.145.
\end{proof}
\begin{rem}
	In Theorem~\ref{Martingale-FCLT}, it is always possible to take a standard Brownian motion $\acute{B}$ independent of $\mathcal{F}$.
\end{rem}
\subsection{Notation and Preliminaries}\label{Sec.Preliminary.Stable.Conv}
In this subsection, we summarize notation and a preliminary result used in the proof of Theorem~\ref{General.logRV.CLT}. In the rest of this section, we consider a sequence of filtrations $\mathbb{F}^n:=\{\mathcal{F}_{s\delta_n}\}_{s\in[0,\infty)}$ and sequences of $\mathbb{F}^n$-martingales $M^n=\{M_s^n\}_{s\in[0,\infty)]}$ and $B^n=\{B_s^n\}_{s\in[0,\infty)]}$ defined by
\begin{equation*}
	M^n_s:=\delta_n^{-\frac{1}{2}}M_{s\delta_n},\ \ B^n_s:=\delta_n^{-\frac{1}{2}}B_{s\delta_n}.
\end{equation*}
Moreover, we set $\tau_j^n:=j/m_n$ for $j\in\mathbb{N}\cup\{0\}$ and $N_s[\tau^n]:=\max\{j\in\mathbb{N}\cup\{0\}:\tau_j^n\leq s\}$ for $s\in[0,\infty)$.  
In the following lemma, we will show that the assumption of the asset price process $S$ introduced in Section~2.1 in the original article implies the similar conditions introduced in Fukasawa~\cite{F10-SPA}. Note that, by localization argument, we can also assume without loss of generality that $\kappa$ is bounded and so the volatility process $\sigma^2$ is the H\"older-continuous.
\begin{lem}\label{Lemma.Increment.Martingale}
	For any $k,n\in\mathbb{N}$ and $s\in[0,\infty)$, as $n\to\infty$,
	\begin{align*}
		&\sup_{j=0,1,\cdots,N_s[\tau^n]}\left|E[(M^n_{\tau_{j+1}^n}-M^n_{\tau_j^n})^{2k}|\mathcal{F}_{\tau_j^n}^n] -\sigma_{\tau_j^n\delta_n}^{2k}(2k-1)!!\left(\frac{1}{m_n}\right)^k\right| =o_P\left(\left(\frac{1}{m_n}\right)^k\right), \\
		&\sup_{j=0,1,\cdots,N_s[\tau^n]}\left|E[(M^n_{\tau_{j+1}^n}-M^n_{\tau_j^n})^{2k-1}|\mathcal{F}_{\tau_j^n}^n]\right| =o_P\left(\left(\frac{1}{m_n}\right)^{k-1/2}\right).
	\end{align*}
\end{lem}
\begin{proof}
	Since we have
	\begin{equation*}
		M^n_s-M^n_v=\sigma_{v\delta_n}(B^n_s-B^n_v)+\delta_n^{-\frac{1}{2}}\int_{v\delta_n}^{s\delta_n}(\sigma_u-\sigma_{v\delta_n})\, \mathrm{d}B_u,\ \ 0\leq v\leq s<\infty,
	\end{equation*}
	the binomial theorem yields that for any $k\in\mathbb{N}$,
	\begin{align}
		&E[(M^n_{\tau_{j+1}^n}-M^n_{\tau_j^n})^k|\mathcal{F}_{\tau_j^n}^n] -\sigma_{\tau_j^n\delta_n}^kE[(B^n_{\tau_{j+1}^n}-B^n_{\tau_j^n})^k|\mathcal{F}_{\tau_j^n}^n] \nonumber \\
		=&\sum_{r=1}^k {}_k\mathrm{C}_r \sigma_{\tau_j^n\delta_n}^{k-r}E\left[(B^n_{\tau_{j+1}^n}-B^n_{\tau_j^n})^{k-r}\left(\delta_n^{-\frac{1}{2}}\int_{\tau_j^n\delta_n}^{\tau_{j+1}^n\delta_n}(\sigma_u-\sigma_{\tau_j^n\delta_n})\, \mathrm{d}B_u\right)^r\biggl|\mathcal{F}_{\tau_j^n}^n\right]. \label{Martingale.Conditional.Increments.Decomp}
	\end{align}
	Since the Brownian motion $B$ enjoys stationary and independent increments properties, we have
	\begin{equation}\label{Bm.Conditional.Increments}
		E[(B^n_{\tau_{j+1}^n}-B^n_{\tau_j^n})^{2k}|\mathcal{F}_{\tau_j^n}^n]=(2k-1)!!\left(\frac{1}{m_n}\right)^k,\ \ E[(B^n_{\tau_{j+1}^n}-B^n_{\tau_j^n})^{2k-1}|\mathcal{F}_{\tau_j^n}^n]=0
	\end{equation}
	for any $k\in\mathbb{N}$. Moreover, the Burkholder-Davis-Gundy inequality and the H\"older-continuity of $\sigma$ yield that for any $k\in(0,\infty)$, there exists a constant $C_k>0$ such that for each $s\in[0,\infty)$,
	\begin{align} \label{BDG.Lem1}
		&\sup_{j=1,\cdots,N_s[\tau^n]}E\left[\left|\delta_n^{-\frac{1}{2}}\int_{\tau_j^n\delta_n}^{\tau_{j+1}^n\delta_n}(\sigma_u-\sigma_{\tau_j^n\delta_n})\, \mathrm{d}B_u\right|^k\biggl|\mathcal{F}_{\tau_j^n}^n\right]  \\
		&\leq C_k\sup_{j=1,\cdots,N_s[\tau^n]}E\left[\left|\delta_n^{-1}\int_{\tau_j^n\delta_n}^{\tau_{j+1}^n\delta_n}(\sigma_u-\sigma_{\tau_j^n\delta_n})^2\, \mathrm{d}u\right|^{k/2}\biggl|\mathcal{F}_{\tau_j^n}^n\right] =o_P\left(\left(\frac{1}{m_n}\right)^{k/2}\right)  \nonumber
	\end{align}
	as $n\to\infty$. Then the conclusion follows from (\ref{Bm.Conditional.Increments}) and (\ref{BDG.Lem1}) by using Cauchy-Schwarz's inequality to the rhs of (\ref{Martingale.Conditional.Increments.Decomp}).
\end{proof}
\subsection{Proof of Theorem~\ref{General.logRV.CLT}}\label{Sec.Proof.Stable.Conv.LogRV}
Before proving Theorem~\ref{General.logRV.CLT}, we will show the following theorem.
\begin{thm}\label{Stable.Conv.Thm2}
	Consider sequences of continuous $\mathbb{F}^n$-local martingales $Z^n=\{Z_s^n\}_{s\in[0,\infty)}$ and continuous stochastic processes $\Sigma^n=\{\Sigma_s^n\}_{s\in[0,\infty)}$ respectively given by
	\begin{equation*}
		Z_s^n:=\sqrt{m_n}\left(\sum_{j=0}^\infty\left(M^n_{\tau_{j+1}^n\wedge{s}}-M^n_{\tau_{j}^n\wedge{s}}\right)^2 - 
		\frac{1}{\delta_n}\int_0^{s\delta_n} \sigma_u^2\, \mathrm{d}u
		\right),\ \ \Sigma_s^n:=\frac{1}{\delta_n}\int_{s\delta_n}^{(s+1)\delta_n}\sigma_u^2\, \mathrm{d}u.
	\end{equation*}
	Then a sequence of the $C_{[0,\infty)}$-valued random variables $Y^n=\{Y_s^n\}_{s\in[0,\infty)}$ given by $Y_s^n:=(Z_{s+1}^n-Z_s^n)/\Sigma_s^n$, $s\in[0,\infty)$, converges in law to the continuous Gaussian process $G=\{G_s\}_{s\in[0,\infty)}$ defined in Theorem~\ref{General.logRV.CLT}. 
\end{thm}
\begin{proof}
	Since we have
	\begin{equation*}
		\frac{1}{\delta_n}\int_0^{s\delta_n} \sigma_u^2\, \mathrm{d}u=\langle M^n\rangle_s,\ \ s\in[0,\infty),
	\end{equation*}
	It\^o's formula yields that
	\begin{equation*}
		Z_s^n=2\sqrt{m_n}\sum_{j=0}^\infty\int_{\tau_{j}^n\wedge{s}}^{\tau_{j+1}^n\wedge{s}}(M^n_u-M^n_{\tau_{j}^n\wedge{s}})\, \mathrm{d}M^n_u.
	\end{equation*}
	Since Taylor's theorem yields that 
	\begin{align*}
		\frac{1}{\Sigma_s^n} =&\frac{1}{\sigma_{s\delta_n}^2} -\int_0^1\frac{(\Sigma_s^n-\sigma_{s\delta_n}^2)}{(\sigma_{s\delta_n}^2+z(\Sigma_s^n-\sigma_{s\delta_n}^2))^2}\, \mathrm{d}z \\
		=&\frac{1}{\sigma_{\tau_j^n\delta_n}^2} -\int_0^1\frac{(\sigma_{s\delta_n}^2-\sigma_{\tau_j^n\delta_n}^2)}{(\sigma_{\tau_j^n\delta_n}^2+z(\sigma_{s\delta_n}^2-\sigma_{\tau_j^n\delta_n}^2))^2}\, \mathrm{d}z -\int_0^1\frac{(\Sigma_s^n-\sigma_{s\delta_n}^2)}{(\sigma_{s\delta_n}^2+z(\Sigma_s^n-\sigma_{s\delta_n}^2))^2}\, \mathrm{d}z,
	\end{align*}
	we can decompose $Y^n$ into the following three parts:
	\begin{align}
		Y_s^n=&2\sqrt{m_n}\sum_{j=0}^\infty\frac{1}{\Sigma_s^n}\int_{(\tau_j^n\vee{s})\wedge(s+1)}^{(\tau_{j+1}^n\vee{s})\wedge(s+1)}\left(M^n_u-M^n_{\tau_{j}^n\wedge{s}}\right)\, \mathrm{d}M^n_u \nonumber \\
		=&(\widetilde{Z}_{s+1}^n-\widetilde{Z}_s^n)-R_s^n-(Z_{s+1}^n-Z_s^n)\int_0^1\frac{(\Sigma_s^n-\sigma_{s\delta_n}^2)}{(\sigma_{s\delta_n}^2+z(\Sigma_s^n-\sigma_{s\delta_n}^2))^2}\, \mathrm{d}z \label{Decomp.Thm2}
	\end{align}
	for each $s\in[0,\infty)$, where a sequence of continuous $\mathbb{F}^n$-local martingales $\widetilde{Z}^n=\{\widetilde{Z}_s^n\}_{s\in[0,\infty)}$ and continuous process $R^n=\{R_s^n\}_{s\in[0,\infty)}$ are given by
	\begin{align*}
		&\hspace{-1cm}\widetilde{Z}_s^n:=2\sqrt{m_n}\sum_{j=0}^\infty\int_{\tau_{j}^n\wedge{s}}^{\tau_{j+1}^n\wedge{s}}\left(\frac{M^n_u-M^n_{\tau_{j}^n\wedge{s}}}{\sigma_{\tau_j^n\delta_n}^2}\right)\, \mathrm{d}M^n_u,\\
		&\hspace{-1cm}R_s^n:=2\sqrt{m_n}\sum_{j=0}^\infty\int_{(\tau_j^n\vee{s})\wedge(s+1)}^{(\tau_{j+1}^n\vee{s})\wedge(s+1)}(M^n_u-M^n_{\tau_{j}^n\wedge{s}})\, \mathrm{d}M^n_u \cdot\int_0^1\frac{(\sigma_{s\delta_n}^2-\sigma_{\tau_j^n\delta_n}^2)}{(\sigma_{\tau_j^n\delta_n}^2+z(\sigma_{s\delta_n}^2-\sigma_{\tau_j^n\delta_n}^2))^2}\, \mathrm{d}z.
	\end{align*}
	First of all, we will show that
	\begin{equation}
		\widetilde{Z}^n\stackrel{n\to\infty}{\rightarrow}\sqrt{2}\acute{B}\ \ \mbox{in law}.\label{Stable.Conv.Main}
	\end{equation} 
	Then Theorem~\ref{Martingale-FCLT} yields that, in order to prove (\ref{Stable.Conv.Main}), it suffices to prove that for each $s\in[0,\infty)$,
	\begin{align}
		&\langle \widetilde{Z}^n\rangle_s=2s+o_P(1)\ \ \mbox{as}\ \ n\to\infty. \label{Jacod.Condition.QV}
	\end{align}
	By It\^o's formula, we have
	\begin{equation*}
		\langle \widetilde{Z}^n\rangle_s =4m_n\sum_{j=0}^\infty\int_{\tau_{j}^n\wedge{s}}^{\tau_{j+1}^n\wedge{s}}\left(\frac{M^n_u-M^n_{\tau_{j}^n\wedge{s}}}{\sigma_{\tau_j^n\delta_n}^2}\right)^2\, \mathrm{d}\langle M^n\rangle_u =\sum_{j=0}^{N_s[\tau^n]}\mathcal{B}_j^n+o_P(1)\ \ \mbox{as}\ \ n\to\infty,
	\end{equation*}
	where
	\begin{equation*}
		\mathcal{B}_j^n :=\frac{2}{3}m_n\frac{(M^n_{\tau_{j+1}^n}-M^n_{\tau_{j}^n})^4}{\sigma_{\tau_j^n\delta_n}^4} -\frac{8}{3}m_n\int_{\tau_{j}^n}^{\tau_{j+1}^n}\frac{(M^n_u-M^n_{\tau_{j}^n})^3}{\sigma_{\tau_j^n\delta_n}^4}\, \mathrm{d}M^n_u.
	\end{equation*}
	Since Lemma~\ref{Lemma.Increment.Martingale} and the Burkholder-Davis-Gundy inequality yield that as $n\to\infty$,
	\begin{align*}
		&\sum_{j=0}^{N_s[\tau^n]}E[\mathcal{B}_j^n|\mathcal{F}_{\tau_j^n}^n]=\frac{2}{3}m_n\sum_{j=0}^{N_s[\tau^n]}\frac{1}{\sigma_{\tau_j^n\delta_n}^4}E[(M^n_{\tau_{j+1}^n}-M^n_{\tau_j^n})^4|\mathcal{F}_{\tau_j^n}^n] =2s+o_P(1),\\
		&\sum_{j=0}^{N_s[\tau^n]}E[|\mathcal{B}_j^n|^2|\mathcal{F}_{\tau_j^n}^n]=o_P(1)
	\end{align*}
	hold, the convergence (\ref{Jacod.Condition.QV}) follows from Lemma 2.3. in~\cite{F10-SPA} and the above two convergences. Therefore, the convergence (\ref{Stable.Conv.Main}) follows.
	
	In the rest of this proof, we would like to show that the second and third terms of (\ref{Decomp.Thm2}) are negligible as $n\to\infty$. Namely, we will prove the following three convergences: for any $s\in[0,\infty)$ and $\iota>0$,
	\begin{align}
		&\sup_{0\leq u\leq s}\left|\int_0^1\frac{(\Sigma_u^n-\sigma_{u\delta_n}^2)}{(\sigma_{u\delta_n}^2+z(\Sigma_u^n-\sigma_{u\delta_n}^2))^2}\, \mathrm{d}z\right| =o_P\left(\delta_n^{H_0-\iota}\right)\ \ \mbox{as}\ \ n\to\infty, \label{Holder.Vol.Lem1} \\
		&\sup_{0\leq s\leq u}\left|Z_s^n\right| =O_P\left(1\right)\ \ \mbox{as}\ \ n\to\infty, \label{Sup.Z.Conv} \\
		&\sup_{0\leq s\leq u}\left|R_s^n\right| =o_P\left(1\right)\ \ \mbox{as}\ \ n\to\infty. \label{ucp.Conv.Reminder}
	\end{align}
	Indeed, if (\ref{Holder.Vol.Lem1}), (\ref{Sup.Z.Conv}) and (\ref{ucp.Conv.Reminder}) hold, then the continuous processes appeared in the second and third terms of (\ref{Decomp.Thm2}) converge in probability to the function that is identically zero as $n\to\infty$ so that the convergence of $Y^n$ follows from (\ref{Stable.Conv.Main}) and the continuous mapping theorem. 
	
	At first, (\ref{Holder.Vol.Lem1}) immediately follows from the H\"older-continuity of the volatility process $\sigma^2$. 
	Next, we will prove (\ref{Sup.Z.Conv}). In the similar argument to the first term of (\ref{Decomp.Thm2}), we can show that
	\begin{equation}
		\langle Z^n\rangle_s=2\int_0^s\sigma_u^4\, \mathrm{du}+o_P(1)\ \ \mbox{as}\ \ n\to\infty. \label{Jacod.Condition.QV2}
	\end{equation}
	Then (\ref{Sup.Z.Conv}) follows from (\ref{Jacod.Condition.QV2}) and Doob's inequality. Finally, we will prove (\ref{ucp.Conv.Reminder}). By It\^o's formula, we have
	\begin{equation*}
		R_s^n=\sum_{j=N_{s}[\tau^n]+1}^{N_{s+1}[\tau^n]}\mathcal{C}_{j,s}^n+o_P(1)\ \ \mbox{as}\ \ n\to\infty,
	\end{equation*}
	where
	\begin{equation*}
		\mathcal{C}_{j,s}^n:=2\sqrt{m_n}\int_{\tau_{j}^n}^{\tau_{j+1}^n}(M^n_u-M^n_{\tau_{j}^n})\, \mathrm{d}M^n_u \cdot\int_0^1\frac{(\sigma_{s\delta_n}^2-\sigma_{\tau_j^n\delta_n}^2)}{(\sigma_{\tau_j^n\delta_n}^2+z(\sigma_{s\delta_n}^2-\sigma_{\tau_j^n\delta_n}^2))^2}\, \mathrm{d}z.
	\end{equation*}
	Since Lemma~\ref{Lemma.Increment.Martingale} and the Burkholder-Davis-Gundy inequality yield
	\begin{align*}
		&\sum_{j=N_{s}[\tau^n]+1}^{N_{s+1}[\tau^n]}E[\mathcal{C}_{j,s}^n|\mathcal{F}_{\tau_j^n}^n]=2\sqrt{m_n}\sum_{j=N_{s}[\tau^n]+1}^{N_{s+1}[\tau^n]}\int_0^1\frac{(\sigma_{s\delta_n}^2-\sigma_{\tau_j^n\delta_n}^2)}{(\sigma_{\tau_j^n\delta_n}^2+z(\sigma_{s\delta_n}^2-\sigma_{\tau_j^n\delta_n}^2))^2}\, \mathrm{d}z E\left[\int_{\tau_{j}^n}^{\tau_{j+1}^n}(M^n_u-M^n_{\tau_{j}^n})\, \mathrm{d}M^n_u \biggl|\mathcal{F}_{\tau_j^n}^n\right]=0,\\
		&\sum_{j=N_{s}[\tau^n]+1}^{N_{s+1}[\tau^n]}E[|\mathcal{C}_{j,s}^n|^2|\mathcal{F}_{\tau_j^n}^n] \\ &=4m_n\sum_{j=N_{s}[\tau^n]+1}^{N_{s+1}[\tau^n]}\left(\int_0^1\frac{(\sigma_{s\delta_n}^2-\sigma_{\tau_j^n\delta_n}^2)}{(\sigma_{\tau_j^n\delta_n}^2+z(\sigma_{s\delta_n}^2-\sigma_{\tau_j^n\delta_n}^2))^2}\right)^2\, \mathrm{d}z E\left[\left(\int_{\tau_{j}^n}^{\tau_{j+1}^n}(M^n_u-M^n_{\tau_{j}^n})\, \mathrm{d}M^n_u\right)^2\biggl|\mathcal{F}_{\tau_j^n}^n\right] =o_P(1)\ \ \mbox{as}\ \ n\to\infty,
	\end{align*}
	the convergence (\ref{ucp.Conv.Reminder}) follows from an easy modification of Lemma 2.3. in~\cite{F10-SPA} and the above two convergences. Therefore, we finish the proof.
\end{proof}
Let us embed the realized variance $\hat\sigma^2$ into a continuous-time stochastic process
\begin{equation*}
	\hat\sigma^2_s :=\sum_{j=0}^{m_n-1} \left|\log S_{\delta_n\tau_{\lfloor m_ns\rfloor+j+1}^n} - \log S_{\delta_n\tau_{\lfloor m_ns\rfloor+j}^n}\right|^2,\ \ s\in[0,\infty).
\end{equation*}
Then we can obtain the following limit theorem.
\begin{thm}\label{Stable.Conv.Thm3}
	A sequence of c\`adl\`ag processes $\widetilde{Y}^n=\{\widetilde{Y}_s^n\}_{s\in[0,\infty)}$ given by
	\begin{equation*}
		\widetilde{Y}_s^n:=\sqrt{m_n}\left(\frac{\hat\sigma_{s}^2-\int_{s\delta_n}^{(s+1)\delta_n} \sigma_u^2\, \mathrm{d}u}{\int_{s\delta_n}^{(s+1)\delta_n} \sigma_u^2\, \mathrm{d}u}\right),\ \ s\in[0,\infty),
	\end{equation*}
	converges in law to the continuous Gaussian process $G=\{G_s\}_{s\in[0,\infty)}$ defined in Theorem~\ref{General.logRV.CLT}.
\end{thm}

\begin{proof}
	Note that we have
	\begin{align*}
		&\hspace{-1cm}\frac{\sqrt{m_n}}{\delta_n}\left(\sum_{j=0}^\infty(\log{S}_{\tau_{j+1}^n\wedge(s\delta_n)}-\log{S}_{\tau_j^n\wedge(s\delta_n)})^2-\int_0^{s\delta_n} \sigma_u^2\, \mathrm{d}u\right) \\
		&\hspace{-1cm}=Z_s^n +2\sqrt{m_n}\sum_{j=0}^\infty(M^n_{\tau_{j+1}^n\wedge{s}}-M^n_{\tau_j^n\wedge{s}})(A_{\tau_{j+1}^n\wedge{s}}^n-A_{\tau_j^n\wedge{s}}^n) +\sqrt{m_n}\sum_{j=0}^\infty(A_{\tau_{j+1}^n\wedge{s}}^n-A_{\tau_j^n\wedge{s}}^n)^2,
	\end{align*}
	where $A_s^n:=\delta_n^{-1/2}A_{s\delta_n}$, $s\in[0,\infty)$.
	By using Lemma~\ref{Lemma.Increment.Martingale}, we can show that
	\begin{align*}
		&\sqrt{m_n}\sum_{j=0}^\infty(A_{\tau_{j+1}^n\wedge{s}}^n-A_{\tau_j^n\wedge{s}}^n)^2 =o_P(1)\ \ \mbox{as}\ \ n\to\infty, \\
		&2\sqrt{m_n}\sum_{j=0}^\infty(M^n_{\tau_{j+1}^n\wedge{s}}-M^n_{\tau_j^n\wedge{s}})(A_{\tau_{j+1}^n\wedge{s}}^n-A_{\tau_j^n\wedge{s}}^n) =o_P(1)\ \ \mbox{as}\ \ n\to\infty
	\end{align*}
	uniformly in $u\in[0,s]$ for any $s>0$ in the similar way to the proof of Lemma 3.9. and Theorem 3.10. in~\cite{F10-SPA} respectively. Then we obtain
	\begin{align*}
		\frac{\sqrt{m_n}}{\delta_n}\left(\hat\sigma_{s}^2-\int_{s\delta_n}^{(s+1)\delta_n} \sigma_u^2\, \mathrm{d}u\right) &=\frac{\sqrt{m_n}}{\delta_n}\left(\sum_{j=0}^\infty(\log{S}_{\tau_{j+1}^n\wedge\{(s+1)\delta_n\}}-\log{S}_{\tau_j^n\wedge\{(s+1)\delta_n\}})^2-\int_0^{(s+1)\delta_n} \sigma_u^2\, \mathrm{d}u\right) \\
		&\quad-\frac{\sqrt{m_n}}{\delta_n}\left(\sum_{j=0}^\infty(\log{S}_{\tau_{j+1}^n\wedge(s\delta_n)}-\log{S}_{\tau_j^n\wedge(s\delta_n)})^2-\int_0^{s\delta_n} \sigma_u^2\, \mathrm{d}u\right) +o_P(1)\\
		&=(Z_{s+1}^n-Z_s^n) +o_P(1)
	\end{align*}
	as $n\to\infty$ uniformly in $s\in[0,u]$ for any $u>0$. Therefore, the conclusion follows from Theorem~\ref{Stable.Conv.Thm2} and the continuous mapping theorem since $1/\Sigma_s^n=O_P(1)$ as $n\to\infty$ uniformly in $u\in[0,s]$ for any $s>0$.
\end{proof}
In the end of this appendix, we prove Theorem~\ref{General.logRV.CLT} by using Theorem~\ref{Stable.Conv.Thm3}.
\begin{proof}[Proof of Theorem~\ref{General.logRV.CLT}]
	By Taylor's theorem, we obtain
	\begin{align*}
		&\sqrt{m_n}\left(\log\hat\sigma_s^2-\log\int_{s\delta_n}^{(s+1)\delta_n} \sigma_u^2\, \mathrm{d}u\right) =\sqrt{m_n}\log\left(1+\frac{\hat\sigma_{s}^2-\int_{s\delta_n}^{(s+1)\delta_n} \sigma_u^2\, \mathrm{d}u}{\int_{s\delta_n}^{(s+1)\delta_n} \sigma_u^2\, \mathrm{d}u}\right) \\
		=&\sqrt{m_n}\left(\frac{\hat\sigma_{s}^2-\int_{s\delta_n}^{(s+1)\delta_n} \sigma_u^2\, \mathrm{d}u}{\int_{s\delta_n}^{(s+1)\delta_n} \sigma_u^2\, \mathrm{d}u}\right) +\sqrt{m_n}\left(\frac{\hat\sigma_{s}^2-\int_{s\delta_n}^{(s+1)\delta_n} \sigma_u^2\, \mathrm{d}u}{\int_{s\delta_n}^{(s+1)\delta_n} \sigma_u^2\, \mathrm{d}u}\right)^2\int_0^1(1-z)\left\{1+z\left(\frac{\hat\sigma_{s}^2-\int_{s\delta_n}^{(s+1)\delta_n} \sigma_u^2\, \mathrm{d}u}{\int_{s\delta_n}^{(s+1)\delta_n} \sigma_u^2\, \mathrm{d}u}\right) \right\}^{-2}\, \mathrm{d}z
	\end{align*}
	for each $s\in[0,\infty)$. Since we have
	\begin{equation*}
		\sup_{0\leq s\leq s_0}\left|\int_0^1(1-z)\left\{1+z\left(\frac{\hat\sigma_{s}^2-\int_{s\delta_n}^{(s+1)\delta_n} \sigma_u^2\, \mathrm{d}u}{\int_{s\delta_n}^{(s+1)\delta_n} \sigma_u^2\, \mathrm{d}u}\right) \right\}^{-2}\, \mathrm{d}z\right|=O_P(1)\ \ \mbox{as}\ \ n\to\infty
	\end{equation*}
	for each $s_0\in[0,\infty)$, the conclusion follows from Theorem~\ref{Stable.Conv.Thm3} and the continuous mapping theorem.
\end{proof}
\section{Approximate Formula of Estimation Function $U_n(H,\nu)$}\label{Appendix.Numerical.Experiments}
In this appendix, we derive the approximate formula of the estimation function (15) in the original article. Since the spectral density $g_{H,\nu}(\lambda)$ and the periodogram $I_n(\lambda)$ are symmetric with respect to $\lambda\in[-\pi,\pi]$, we have
\begin{align*}
	U_n(H,\nu)=&\frac{1}{2\pi}\int_0^\pi\left(\log{g_{H,\nu}(\lambda)} +\frac{I_n(\lambda,\mathbf{Y}_n)}{g_{H,\nu}(\lambda)}\right)\, \mathrm{d}\lambda \\
	=&\frac{1}{2\pi}\int_\psi^\pi\left(\log{g_{H,\nu}(\lambda)} +\frac{I_n(\lambda,\mathbf{Y}_n)}{g_{H,\nu}(\lambda)}\right)\, \mathrm{d}\lambda +B_{H,\nu}^1\left(\psi\right) +B_{H,\nu}^2\left(\psi\right)
\end{align*}
for any $\psi\in(0,\pi]$, where
\begin{equation*}
	B_{H,\nu}^1(\psi):=\frac{1}{2\pi}\int_0^\psi\log{g_{H,\nu}(\lambda)}\, \mathrm{d}\lambda,
	\hspace{0.2cm}B_{H,\nu}^2(\psi):=\frac{1}{2\pi}\int_0^\psi\frac{I_n(\lambda,\mathbf{Y}_n)}{g_{H,\nu}(\lambda)}\, \mathrm{d}\lambda.
\end{equation*}
In the rest of this subsection, we will show $B_{H,\nu}^1(\psi)\approx A_{H,\nu}^1(\psi)$ and $B_{H,\nu}^2(\psi)\approx A_{H,\nu}^2(\psi)$ as $\psi\downarrow 0$. 
At first, we consider the first approximation. Note that the Taylor expansion yields that
\begin{equation}\label{Approximation.Formula.SPD}
	g_{H,\nu}(\lambda)=\nu^2C_H|\lambda|^{1-2H} +\frac{|\lambda|^2}{m\pi} +O(|\lambda|^{3+2H})\hspace{0.2cm}\mbox{as $|\lambda|\to 0$}.
\end{equation}
Then we obtain the first approximation from the Taylor expansion as $\psi\downarrow 0$ as follows:
\begin{align*}
	B_{H,\nu}^1(\psi)\approx&\frac{1}{2\pi}\int_0^\psi\log\left(\nu^2C_H\lambda^{1-2H} +\frac{\lambda^2}{m\pi}\right)\, \mathrm{d}\lambda \\
	=&\frac{1}{2\pi}\left\{\psi\log(\nu^2C_H) +\psi(\log\psi-1)(1-2H) +\int_0^\psi\log\left(1+\frac{1}{\nu^2C_Hm\pi}\lambda^{1+2H}\right)\, \mathrm{d}\lambda\right\} \\
	\approx&\frac{1}{2\pi}\left\{\psi\log(\nu^2C_H) +\psi(\log\psi-1)(1-2H) +\frac{\psi^{2+2H}}{\nu^2C_Hm\pi(2+2H)}\right\}.
\end{align*}
Next we consider the second approximation. Since $g_{H,\nu}$ is an even function, $B_{H,\nu}^2(\psi)$ is represented by
\begin{equation*}
	B_{H,\nu}^2(\psi)=\frac{1}{2\pi}\left(b_{H,\nu}(0,\psi)\widehat\gamma_n(0)+2\sum_{\tau=1}^{n-1}b_{H,\nu}(\tau,\psi)\widehat\gamma_n(\tau)\right),
\end{equation*}
where 
\begin{equation*}
	b_{H,\nu}(\tau,\psi):=\frac{1}{2\pi}\int_0^\psi\frac{\cos(\tau\lambda)}{g_{H,\nu}(\lambda)}\, \mathrm{d}\lambda.
\end{equation*}
Since the Taylor expansion as $\psi\downarrow 0$ yields that
\begin{align}
	b_{H,\nu}(\tau,\psi)=&\frac{1}{2\pi}\sum_{j=0}^\infty\frac{(-1)^j\tau^{2j}}{(2j)!}\int_0^\psi\frac{\lambda^{2j}}{g_{H,\nu}(\lambda)}\, \mathrm{d}\lambda \label{b.cos.expansion} \\
	\approx&\frac{1}{2\pi}\sum_{j=0}^\infty\frac{(-1)^j\tau^{2j}}{(2j)!}\int_0^\psi\frac{\lambda^{2j}}{\nu^2C_H|\lambda|^{1-2H} +\frac{|\lambda|^2}{m\pi}}\, \mathrm{d}\lambda \nonumber \\
	\approx&\frac{1}{2\pi}\sum_{j=0}^\infty\frac{(-1)^j\tau^{2j}}{(2j)!}\int_0^\psi\frac{\lambda^{-1+2j+2H}}{\nu^2C_H}\left(1-\frac{1}{\nu^2C_Hm\pi}\lambda^{1+2H}\right)\, \mathrm{d}\lambda \nonumber \\
	=&\frac{1}{2\pi}\sum_{j=0}^\infty\frac{(-1)^j\tau^{2j}}{(2j)!}\frac{1}{\nu^2C_H}\left(\frac{\psi^{2j+2H}}{2j+2H} -\frac{\psi^{1+2j+4H}}{\nu^2C_Hm\pi(1+2j+4H)}\right),\label{B2.Series}
\end{align}
we obtain the second approximation when the series in (\ref{B2.Series}) is truncated after finite terms. Note that the truncation error of the Taylor expansion in (\ref{b.cos.expansion}) is dominated as follows:
\begin{equation*}
	\sup_{\tau\in\{0,1,\cdots,n-1\}}\left|b_{H,\nu}(\tau,\psi)-\frac{1}{2\pi}\sum_{j=0}^J\frac{(-1)^j}{(2j)!}\int_0^\psi\frac{(\tau\lambda)^{2j}}{g_{H,\nu}(\lambda)}\, \mathrm{d}\lambda\right|\leq \frac{(n\psi)^{2J+1-1}}{(2J+1)!}\cdot\frac{1}{2}\int_0^\psi\frac{1}{g_{H,\nu}(\lambda)}\, \mathrm{d}\lambda
\end{equation*}
for any $J\in\mathbb{N}$ and $\psi>0$. As a result, for fixed $n\in\mathbb{N}$, we can make the truncation error arbitrary small uniformly with respect to $\tau\in\{0,1,\cdots,n-1\}$ as $J\in\mathbb{N}$ is taken sufficiently large even in the case of the finite sample.
\end{document}